\documentclass{amsart}

\title{Nilpotent Group-Counterexamples to Zil'ber's Conjecture}

\author{Andreas Baudisch}
\address{
Institut f\"ur Mathematik,Humboldt-Universit\"at zu Berlin,
  D-10099 Berlin, Germany}

\email{baudisch@mathematik.hu-berlin.de}
\date{\today}


\usepackage[latin1]{inputenc}

\usepackage{amssymb, amsmath, tikz}
\usetikzlibrary{matrix,arrows, calc}
\usetikzlibrary{fit}
\usetikzlibrary{positioning}

\usepackage{enumerate,xspace}

\RequirePackage{ifpdf}
\ifpdf
   \usepackage[pdftex]{hyperref}
\else
   \usepackage[hypertex]{hyperref}
\fi

\newcommand{\Frai}{ Fra\"{\i}ss\'e  }
\theoremstyle{plain}
\newtheorem{theorem}{Theorem}[section]
\newtheorem{cor}[theorem]{Corollary}

\newtheorem{lemma}[theorem]{Lemma}

\theoremstyle{definition}

\newtheorem{definition}[theorem]{Definition}

\newcommand{\nc}{\newcommand}
\nc{\M}{\mathbb{M}}
\nc{\C}{\mathfrak{C}}
\nc{\ga}[2]{\gamma_{#1^*}^{#2^*}}
\nc{\gam}[2]{\gamma_{#1}^{#2}}

\nc{\N}{\mathbb{N}}
\nc{\cb}{\operatorname{Cb}}
\nc{\Kom}{\mathbb{K}} 
\nc{\h}{\mathcal{H}}
\nc{\K}{\mathcal{K}}
\nc{\J}{\mathcal{J}}
\nc{\tp}{\operatorname{tp}}
\nc{\stp}{\operatorname{stp}}
\nc{\RM}{\operatorname{RM}}
\nc{\x}{{\bf c}}

\def\G{\mathbb G}
\def\F{\mathbb F}

\def\M{\mathfrak{M}}

\nc{\Mr}{\operatorname{MR}}
\nc{\U}{\operatorname{U}}
\nc{\p}{\operatorname{p}}
\nc{\A}{\mathcal{A}} 
\nc{\Cent}{\operatorname{C}} 
\nc{\X}{\mathcal{X}}
\nc{\cc}{\mathcal{C}}
\nc{\lmto}[1]{\xrightarrow[{#1}]{}} 
\nc{\ggen}[1]{\langle #1\rangle} 
\nc{\rest}[2]{#1\!\!\upharpoonright_{#2}}

\renewcommand{\iff}{if and only if\xspace}

\def\Ind#1#2{#1\setbox0=\hbox{$#1x$}\kern\wd0\hbox to 0pt{\hss$#1\mid$\hss}
\lower.9\ht0\hbox to 0pt{\hss$#1\smile$\hss}\kern\wd0}
\def\Notind#1#2{#1\setbox0=\hbox{$#1x$}\kern\wd0\hbox to
0pt{\mathchardef\nn="0236\hss$#1\nn$\kern1.4\wd0\hss}\hbox
to 0pt{\hss$#1\mid$\hss}\lower.9\ht0
\hbox to 0pt{\hss$#1\smile$\hss}\kern\wd0}
\def\ind{\mathop{\mathpalette\Ind{}}}
\def\nind{\mathop{\mathpalette\Notind{}}}

\nc{\unab}{\ind^{\otimes}}

\begin{document}
\begin{abstract}
We construct uncountably categorical 3-nilpotent groups of exponent $p>3$. 
They are not one-based and do not allow the interpretation of an infinite field. Therefore they are counterexamples to Zil'ber's Conjecture.  First 2-nilpotent new uncountably categorical groups were contructed in \cite{Bau96}.  Here we use the method of the additive Collapse developed in \cite{Bau09}. Essentially we work with 3-nilpotent graded Lie algebras over the field with p elements.
 
\end{abstract}
\maketitle

\section{Introduction}
\noindent Zil'ber's Conjecture is the following statement:\\  
Let $T$ be an uncountably categorical theory  in a countable language. If it  is  not one-based, then it is possible to interpret an infinite field in $T$.\\
The first counterexample was given by E.Hrushovski in 1988 (see \cite{Hr1}). It is a relational strongly minimal theory that does not even allow to interpret an infinite group.\\
In \cite{Bau96} a first group-counterexample is given . Using the classical results on groups of finite Morley rank it is easy  to see that such a group-counterexample is essentially a simple group or a nilpotent group of finite exponent. The groups constructed in \cite{Bau96} are nilpotent of class 2 and of exponent $p> 2$.  In fact we worked with alternating bilinear maps. In the terminology of this paper they are 2-nilpotent graded Lie algebras over the field $\F(p)$ with p elements ($p > 2$). If we try a similar construction in higher nilpotency classes 
additional  difficulties arise.  Here we will give 3-nilpotent counterexamples.\\
In his paper \cite{Hr2} on the fusion of two strongly minimal sets E. Hrushovski developed new ideas for such constructions as above. Together with A.Martin-Pizarro and M.Ziegler we used his ideas to obtain  fields of prime characteristic
of Morley rank 2 equipped with a definable additive subgroup of rank 1\cite{BMPZ3}. Furthermore we realised the fusion of two strongly minimal sets with DMP over a common vector space to obtain again a strongly minimal set\cite{BMPZ4}. Finally  bad fields were construced by M.Hils, A.Martin-Pizarro, F.Wagner and me \cite{BHMPW}. \\
In \cite{Bau09} a common frame is built  for the constuctions of the new uncountably categorical groups, the red fields, and the fusion over a vector space.
There we  have a starting theory $T$ that fulfills certain conditions. Notions like codes and difference sequences as in \cite{BMPZ3} and \cite{BMPZ4} are introduced and finally a collaps gives the desired theory of finite Morley rank.\\
In this paper we will follow this strategy of \cite{Bau09}. The main part (section 2 - 11) is devoted to the construction of an elementary theory of a 3-nilpotent graded Lie algebra over $\F(p)$, as a starting theory for the collaps.
It has infinite Morley rank.\\
In section 2 we consider c-nilpotent graded Lie algebras $M = M_1 \oplus \ldots \oplus M_c$ over a finite field $K$
in a suitable elementary language $L$.
The $M_i$ are $K$-vector spaces. For $x \in M_i$ and $y \in M_j$ we have $[x,y] \in M_{i+j}$.  A function $\delta$ over the set of  all finite substructures $A$ of $M$
into the natural numbers is defined. For $c = 2$ the definition is a minor deviation of the definition in \cite{Bau96}. 
Submodularity for the case $c = 2$ is shown. Furthermore we define $A$ is  strong in $M$, if  $\delta(A) \le \delta(B)$ for
all finite $A \subseteq B \subseteq M$.  For later amalgamation a class $\K_c$ is defined. In Lie algebras $M$ in $\K_c$ we  have $0 < \delta(A)$ for all substrucures $A \not= \langle 0 \rangle$ of $M$. \\
In section 3 amalgamation is introduced. The existence of the free amalgam of c-nilpotent graded Lie algebras over a fixed field is shown in \cite{Bau16}. We study the structure of of the free amalgam in the case $c = 3$.\\
In section 4 we describe the results for the case $c = 2$ from \cite{Bau96} and \cite{Bau09}. \\
From now on we work with  3-nilpotent Lie algebras.
In section 5 a functor $\F$ from $\K_{c-1}$ into $\K_c$ is defined. For $A \in \K_c$ let $A^* = A/A_c$. If $B = \langle B_1 \oplus \ldots \oplus B_{c-1} \rangle \in \K_c$ and $A^* \cong B^*$, then there is a homomorphism of $\F(A^*)$ onto $B$. In the case of 3-nilpotent Lie algebras we  show that a strong embedding of $B^*$  into $A^*$   in the sense of $\K_2$ 
implies  that the natural homomorphism of $\F(B^*)$ into $\F(A^*)$ is an embedding. 
This is part of  the main result 
(Theorem \ref{key3}) in this section. It was proved in A.Amantini's dissertation \cite{Ama}.\\
In section 6 we work again in $\K_3$. Submodularity of the $\delta$-function for  substructures 
$A \in M$, where $A^* \le M^*$,  is shown. Then  we prove the amalgamation property for $\K_3$ with strong embeddings.\\
Using this amalgamation we obtain in section 7 a countable strong  Fra\"{\i}ss\'e-Hrushovski limit $\M$ - the desired Lie algebra.
We study the theory $T_3$ of $\M$. $\M$ is uniquely determined by richness: If $B $ is strong in $A$ and both are in $\K_3$, then a strong embedding $f$ of $B $ into $\M$ can be extended to an strong embedding of $A$ into $\M$. We give an axiomatization of $T_3$.
In section 8 non-forking and canonical bases are investigated. 
$T_3$ is $\omega$-stable and CM-trivial. \\
Let $\C$ be a monster model of $T_3$. In section 9 we define a pregeometry $cl$, using the $\delta$-function. Its domain is
the union $R(\C) = \C_1 \cup C_2$ of the first two vector spaces of the graduation. 
The pregeometry is defined for all finite subspaces $A = \langle A_1 A_2 \rangle$. The smallest L-subspaces of this form 
are generated by a single  element in $R(\C) = \C_1 \cup \C_2$.\\
Our aim is to find a strong substructure $P^{\mu}(\C)$
of $\C$, that has an uncountably categorical theory with the desired properties. For this we have to ensure that 
in the structure  $P^\mu(\C)$ the geometrical closure is the algebraic closure.
Therefore 
we study  so-called minmal prealgebraic extension over substructures in section 10. We build formulas $\phi \in \X^{home}$ that describe these extensions. They are strongly minimal. To realise our aim above we define an expansion $\C^{\mu}$ of $\C$ 
by adding  a new predicate $P^{\mu}$, such that in the substructure $P^{\mu}(\C^{\mu})$ the number of solutions of the $\phi \in \X^{home}$ with  parameters in $P^{\mu}(\C^{\mu})$ is finite. The bound is given by a function $\mu$
on codes, that are modifications of the formulas in $\X^{home}$. 
There are uncountably many possible $\mu$-functions.\\
In \cite{Bau09} conditions are formulated that provide the existence of $\C^{\mu}$. There we work with a pregeometry over a vector space, but here $R(\C)$ is the union of two vector spaces. Therefore we have to modify the conditions from \cite{Bau09}. In section 11 we show that $T_3$ satisfies these  new conditions C(1)- C(7) for the collaps. \\
In the following sections we use our full knowledge of $T_3$ and not only these conditions.
In section 12 we indroduce codes $\alpha$, that are modifications of the formulas in $\X^{home}$. Difference sequences are realizations of a spezial formulas 
$\psi_\alpha(\bar{x}_0, \ldots, \bar{x}_\lambda)$ that describes some properties of a
sequence $\bar{a}_0 - \bar{f},  \ldots, \bar{a}_\lambda - \bar{f}$
,where $\bar{a}_0, \ldots, \bar{a}_\lambda, \bar{f}$ is a Morley sequence of 
realizations of a code formula $\varphi_\alpha(\bar{x},b)$.\\
Then we introduce bounds for difference sequences in section 13. For this  Lemma \ref{add5.1} 
provides  an 
important combinatorial property of these sequences.
It needs new ideas for the proof.
We consider a class $\K^\mu$ of strong subalgebras of $\C$, where a suitable function $\mu$ on
the set of codes gives the desired bounds.\\
In section 14 we use the condition C(6) to  amalgamate strong subalgebras from $\K^\mu$ inside $\C$.
We get a countable $\K^\mu$-rich strong subalgebra $P^\mu(\C)$ of $\C$:\\
If $B \le A$ are in $\K^\mu$ and if there is an 
strong embedding of $B$ into $P^\mu(\C)$, then we can extend the 
embedding to a strong  embedding of $A$ into $P^\mu(\C)$. \\
In section 15 we extend our language by a predicate $P^\mu$ for this subalgebra. The new language is denoted by $L^\mu$. A $L^\mu$-structure $(M,P^\mu(M))$ is a $2\times$rich $L^\mu$-structure,
if $M$ is a rich model of $T_3$,
$P^\mu(M) \le M$ is $\K^\mu$-rich, and the geometrical dimension of $M$ over $P^\mu(M)$ is infinite.
These structures have a complete theory $T^\mu_3$, with a monster model $\C^\mu$.\\
In section 16 $T^\mu_3$ is axiomatized. We get $\omega$-stability.
Let $\Gamma(\C^{\mu})$ be  the $L$-substructure with domain  $P^{\mu}$. 
It is the desired new uncountablly categorical graded 3-nilpotent Lie algebra  over a finite field. It will be 
considered in section 17.
In this substructure $cl$ is part of the algebraic clousure. $\Gamma(\C^{\mu})$ is stably embedded in $\C^{\mu}$. In $\Gamma(\C^{\mu})$ the predicate $R$ has Morley rank 1 and Morley degree 2.
The geometry of the algebraic closure is not locally modular, but the theory is CM-trival. It is not possible
to interpret an infinite field. 
The Morley rank of this structure is 3. \\
In section 18 we get  group counterexamples by interpretation. 
We use the Baker-Campbell-Hausdorff-formula. Here the sum is finite. These groups are bi-interpretable with the Lie algebras of $T_3^\mu$ without graduation. Their theories are uncountably categorical. They are not one-based, CM-trivial, and do not allow the interpretation of a field.


\section{Definition of  $\delta$}
\noindent We consider c-nilpotent graded Lie algebras $M$ over a finite field $K$ in a language $L$. The non-logical
symbols of $L$ are the following: There are $+$, $-$,$0$ and unary functions for the scalar multiplication with the elements of $K$ to describe the underlying vector space. Furthermore we have unary predicates 
$U_i$ ($1 \le i \le c$) for the graduation. That means $M = \oplus_{1 \le i \le c} M_i$, where $M_i = U_i(M)$.  
Projections $pr_i$ are needed to ensure that substructures are again graded.  $[x,y]$ is the symbol for the Lie - multiplication. In $M$ $U_i(a)$ and $U_j(b)$ implies $U_{i+j}([a,b])$.  $\langle X \rangle$ denotes the $L$-substructure of $M$ generated by $X \subset M$.  $\langle X \rangle ^{lin}$
denotes the linear hull of $X$. If $a \in U_i(M)$, then we say that $a$ is homogeneous of degree $i$. A subset $X$ is homogeneous, if all its elements are homogenous. If $a \in M$ then  $a = \sum_{1 \le i \le c} r_i a_i$, where $a_i \in U_i(M) = M_i$. The degree of $a$ is the smallest $i$, such that $r_i \not= 0$. Finally $\langle \langle X \rangle \rangle$ is the ideal generated by $X$ in 
$M$. $ldim$ is used to denote the linear dimension.\\
$A,B,C,D$ denote  finite and $M,N$  arbitrary c-nilpotent graded Lie algebras over $K$. For every 
finite $A$
we define an integer $\delta (A)$ that is uniquely determined  by the isomorphism type of $A$. Note that we do not assume in general that $A = \langle A_1 \rangle$. \\

\noindent We use the Theorem of Sirsov-Witt: Every subalgebra of a free Lie algebra over a field is free. In \cite{Bau82}
a key lemma for the proof of this Theorem is formulated, that is useful from a model-theoretic point of view. Here we develop the results in a context with graduation.\\
If $F(X)$ is the the free Lie algebra freely generated by $X$, then $F(X)$ becomes a free graded Lie algebra, if we define $U_i(F(X))$ to be the vector space generated by all momomials over $X$ of degree $i$. $U_1(F(X))$ is the linear hull of $X$. Projections are defined in the usual way. Then we have
\[ Z_{c+1-i}(F(X)) = \Gamma_i(F(X)) = \oplus_{i \le j \le c} U_j(F(X)) ,\] 
where $Z_1(M) \subseteq \ldots \subseteq Z_c(M)$ is the upper central series and 
$\Gamma_1(M)  \supseteq \ldots \supseteq \Gamma_c(M)$ the lower central series. \\

\begin{definition} Let $Y$ be a homogeneous subset of $M$.
\begin{enumerate}
\item $Y$ is an (o)-system in $M$, if for all $i$ ($1 \le i \le c$) $U_i(Y) $ is linearly independent over
$\langle U_1(Y), \ldots, U_{i-1}(Y) \rangle_i$.
\item $Y$ generates $M$ freely, if  $Y$ is an o-system 
in $M$, that generates $M$ and for every  c-nilpotent graded Lie algebra $N$ over $K$
every map $f$ of $Y$ into $N$ with $U_i(y)$ implies $U_i(f(y))$ can be extended to an $L$-homomorphism of $M$ into $N$.
We call such a Lie algebra $M$ the  free algebra $F(Y)$. 
\end{enumerate}
\end{definition}

\noindent The fact that $Y$ generates $M$ freely can be expressed by linear independence of basic commutators.

\noindent In \cite{Bau82} the following Theorem is taken from the proof of the Sirsov-Witt-Theorem and applied for the following corollaries. Here we formulate it with graduation.

\begin{theorem}\label{sirsovwitt}
In a free graded c-nilpotent Lie Algebra $F(X)$ every (o)-system $Y$ freely generates 
$\langle Y \rangle$.
\end{theorem}

\begin{cor} Every element of  $F(Y)$ has a unique presentation as a linear combination of basic commutators over $Y$.
\end{cor}

\begin{definition} If $ X = \cup_{1 \le i \le c} X_i$ where $X_i \subseteq U_i(M)$ is a generating o-system for M, then we define $o-dim_i(M) = \mid X_i \mid$. $o-dim(M) = \sum_{1 \le i \le c} o-dim_i(M)$.
\end{definition}

\noindent It is easily seen, that $o-dim_i(M)$ is independent from the choice of $X$.

\begin{cor}\label{iso} If $M$ and $N$ are free and $o-dim_i(M) = o-dim_i(N)$ for $1 \le i \le c$, then $ M \simeq N$.
\end{cor}

\noindent We choose  a generating o-system $W$ for some c-nilpotent graded Lie algebra $M$. Let $F(X)$ be the free graded c-nilpotent Lie algebra over $K$ freely generated by $X$, where $X$ be an o-system with $\mid X_i \mid = \mid W_i\mid$   
for $1 \le i \le c$.
Then for every map of $X_i$ onto $W_i$ for all $i$ we get a homorphism  $f$ of $F(X)$ onto $M$. Hence 
$M \simeq F(X)/ker(f)$.

\begin{definition}\label{2.7} We call $F(X), ker(f)$ as above a canonical pair for $M$. As we 
will see below it is unique up to automorphisms of $F(X)$.
\end{definition}

\noindent If $g$ is any homorphism of $F(X)$ onto $M$, then there is a generating o-system $X^0$ in $F(X)$ such that
$g(X^0_i)  = W_i$. By Corollary \ref{iso} there is an automorphism $\alpha$ with $\alpha(X_i) = X^0_i$. Hence $f=g\alpha$ and $ker(g)$ is an automorphic image of $ker(f)$. It follows that the following definition is independent of $f$.

\begin{definition}\label{2.8} An ideal basis for $M \simeq F(X)/ker(f)$ is a homogeneous subset $Y$ of $ker(f)$, such that 
$Y_i$ generates $ker(f) \cap F(X)_i$ modulo $\langle \langle Y_2, \ldots, Y_{i-1}\rangle \rangle \cap F(X)_i$ for $2 \le i \le c$. 
For finite $M$ we define $ideal-dim_i(M) = \mid Y_i \mid$ and  
\[ \delta_i(M) = \sum_{1 \le j \le i} o-dim_j(M) - (\sum_{2 \le j \le i} ideal-dim_j(M))\]
and $\delta(M) = \delta_c(M)$.
\end{definition}

\noindent $o-dim_i(M)$, $ideal-dim_i$, and $\delta_i(M)$ are invariants of $M$. They depend only on the isomorphism-typ of $M$. In the approach in \cite{Bau96} and in \cite{Bau09} for $c = 2$ we consider only substructures $A$ with $A = \langle A_1 \rangle$. In that context  the 
proof of the next lemma is easier.

\begin{lemma}\label{subdelta2} Submodularity holds for $\delta_2$.
Let $M$ be a 2-nilpotent graded Lie algebra. Let $A$ and $C$ be subalgebras of $M$. Then
\[ \delta_2(\langle A C \rangle) \le \delta_2(A) + \delta_2(C) - \delta_2(A \cap C) .\]
\end{lemma}

\begin{proof} Define $A_i  \cap C_i = B_i$. Let $X_1 = X_1(AC) = X_1(C) X_1(B) X_1(A)$ be a vector basis for $\langle CA \rangle_1$ such that:
\begin{description}
\item $X_1(B)$ is a vector basis of $B_1$.
\item $X_1(B) X_1(C)$ is a vector basis of $C_1$.
\item $X_1(B) X_1(A)$ is a vector basis of $A_1$.
\end{description}
We define:
\begin{description}
\item $X_2^0(B) $ is a vector basis for $\langle C_1 \rangle_2 \cap \langle A_1 \rangle_2$ over $\langle B_1 \rangle_2$.
\item $X_2^{A}(B)$ is a vector basis for $\langle A_1 \rangle_2 \cap C_2$ over $\langle B_1\rangle_2 + \langle X_2^0(B) \rangle$.
\item $X_2^{C}(B)$ is a vector basis for $\langle C_1 \rangle_2 \cap A_2$ over $\langle B_1 \rangle_2 +
\langle X_2^0(B) \rangle$.
\item $X_2^1(B)$ is a vector basis for $B_2 \cap \langle X_1 \rangle_2$ over 
$\langle B_1 \rangle_2 + \langle X_2^0(B) X_2^{A}(B) X_2^{C}(B) \rangle$.
\item $X_2^2(B)$ is a vector basis of $B_2$ over $\langle X_1 \rangle_2$. 
\end{description}
 Note that $X_2^{A}(B) \cup X_2^{C}(B)$ is linearly independent over $\langle X_1(B) \rangle_2 + \langle X_2^0(B) \rangle_2$. 
Now $B_2$ has the following vector basis $X_2(B)$ over $\langle B_1  \rangle$:
\[ X_2^0(B) X_2^{A}(B) X_2^{C}(B) X_2^1(B) X_2^2(B) .\] 

Let $X_2^1(A)$ and $X_2^1(C)$ be choosen such that

\begin{description}
\item $ X_2^{C}(B) X_2^1(B)  X_2^1(A)$ is a vector basis of $A_2 \cap \langle A_1 C_1 \rangle_2$ over $\langle A_1 \rangle_2$.
\item $ X_2^{A}(B) X_2^1(B)  X_2^1(C)$ is a vector basis of $C_2 \cap \langle A_1 C_1 \rangle_2$
over $\langle C_1 \rangle_2$.
\end{description} 

Choose $X_2^2(A)$ and $X_2^2(C)$ such that

\begin{description}
\item $X_2(A) = X_2^{C}(B) X_2^1(B) X_2^2(B) X_2^1(A) X_2^2(A)$ is a vector basis of $A_2$ over $\langle A_1 \rangle_2$.
\item $X_2(C) = X_2^{A}(B) X_2^1(B) X_2^2(B)  X_2^1(C) X_2^2(C)$ is a vector basis of $C_2$ over $\langle C_1 \rangle_2$.
\end{description} 

Then $X_2(AC) = X_2^2(B) X_2^2(A) X_2^2(C)$ is a vector basis  of $\langle AC \rangle_2$ over $\langle A_1C_1 \rangle_2$. Hence $X_1(AC) X_2(AC)$ is a generating o-system of $\langle AC \rangle$. Then
\[ \langle AC \rangle \simeq F(X_1(AC) X_2(AC)) / I,\]
where $I$ is a subspace of $F(X_1(AC)_2$.
We get
\begin{enumerate}
\item $\delta_2(\langle AC \rangle = \mid X_1(AC) \mid + \mid X_2(AC) \mid - ldim(I)$

\item $\delta_2(A) = \mid X_1(B) \mid + \mid X_1(A)\mid + \mid X_2(A) \mid
- ldim(I \cap  F(X_1(B) X_1(A))$.

\item $\delta_2(C) = \mid X_1(B) \mid + \mid X_1(C) \mid + \mid X_2(C) \mid 
- ldim(I \cap F(X_1(B) X_1(C))$.

\item $\delta_2(B) = \mid X_1(B) \mid  + \mid X_2(B) \mid - ldim(I \cap F(X_1(B))$.

\item $o-dim_1(A) + o-dim_1(C) - o-dim_1(B) = o-dim_1(AC)$.

\item $o-dim_2(A) + o-dim_2(C) - o-dim_2(B) =$
\[ \mid X_2^{C}(B) \mid + \mid X_2^1(B) \mid + \mid X_2^2(B)\mid + 
\mid X_2^1(A) \mid  + \mid X_2^2(A) \mid + \] 
\[ \mid X_2^{A}(B) \mid + \mid X_2^1(B) \mid + \mid X_2^2(B)\mid + 
\mid X_2^1(C) \mid + \mid X_2^2(C) \mid - \] 
\[\left( \mid X_2^{C}(B) \mid +\mid X_2^{A}(B) \mid + \mid X_2^1(B) \mid + \mid X_2^2(B)\mid + \mid X_2^0(B) \mid \right) = \] 
\[ \mid X_2^1(B) \mid + \mid X_2^2(B)\mid + 
\mid X_2^1(A) +\mid X_2^2(A) \mid + 
\mid X_2^1(C) \mid +\mid X_2^2(C)  \mid - \mid X_2^0(B) \mid = \] 
\[\mid X_2(AC) \mid + \mid X_2^1(B)  X_2^1(A) X_2^1(C) \mid  - \mid X_2^0(B) \mid =\]
\[o-dim_2(AC) + \mid X_2^1(B)   X_2^1(A) X_2^1(C)\mid - \mid X_2^0(B) \mid. Hence \]

\item $\delta_2(A) + \delta_2(C) - \delta_2(B) = \delta_2(AC) + \mid X_2^1(B)  X_2^1(A) X_2^1(C)\mid - \mid X_2^0(B) \mid +$
\[\mid I \mid - \mid I \cap F(X_1(B) X_1(A))\mid - \mid I\cap F(X_1(B) X_1(C))\mid + 
\mid I \cap F(X_1(B)) \mid .\]
\end{enumerate}

Let $I^0$ be $ (I\cap F(X_1(B) X_1(A))) \oplus_{I \cap F(X_1(B))}  (I\cap F(X_1(B) X_1(C))$. Then $I^0$ is a subspace of $I$ and $I$ contains $\mid X_2^0(B) \mid$ - many elements $\psi_i = \psi_i(A) + \psi_i(C)$ with 
$\psi_i(A) \in F(X_1(B) X_1(A))$ and $\psi_i(C) \in F(X_1(B) X_1(C))$, that are linearly independent over $I^0$.
Hence

\[ \delta_2(AC) \le \delta_2(A) + \delta_2(C) - \delta_2(A \cap C) .\]

\end{proof} 

\noindent We summeraize the results of the computation above.

\begin{cor}\label{corsubdelta2}
As above $B = A \cap C$. We define:\\
$r_B = ldim((\langle A_1 C_1 \rangle_2 \cap B_2)/\langle A_1 \rangle_2 + \langle C_1 \rangle_2$,\\
$r_A = ldim((\langle A_1 C_1 \rangle_2 \cap A_2)/\langle A_1 \rangle_2 + \langle C_1 \rangle_2$,
\\
$r_C = ldim((\langle A_1 C_1 \rangle_2 \cap C_2)/\langle A_1 \rangle_2 + \langle C_1 \rangle_2$,
and\\
if $\langle A_1 C_1 \rangle = F(A_1 C_1)/I$, then $s = ldim(I/(I \cap  (F(A_1)) +  F(C_1))$. Then
\begin{enumerate}
\item $\delta_2\langle A C \rangle) - \delta_2(C) = \delta_2(A) - \delta_2(B) - (r_A + r_C - r_B + s).$
\item $\delta_2(\langle AC\rangle) = \delta_2(A) + \delta_2(C) - \delta_2(A \cap C)$ if and only if \\
$r_A = r_C = r_B = s = 0$ if and only if\\
$A_2$ and $C_2$ do not contain elements in $\langle A_1 C_1 \rangle \setminus
\langle A_1 \rangle_2 + \langle C_1 \rangle_2 $,\\
and if $\langle A_1 C_1 \rangle = F(A_1 C_1)/I$, then $I =(I \cap  (F(A_1)) +  F(C_1))$.
\end{enumerate}
\end{cor}

\begin{definition}\label{strong}
Let $M$ be a graded c-nilpotent Lie algebra over a finite field $K$.
Let $A$ be a substructure of $M$.  
\begin{enumerate} 
\item For $2 \le i \le c$ $A$ is i-strong in $M$ (short $A \le_i M$), if for all $2\le j \le i$ and all 
$A \subseteq C \subseteq M$ we have $\delta_j(A) \le \delta_j(C)$. 
\item We write $A \le M$, if $A \le_c M$. In this case we say $A$ is strong in $M$.
We use $A \le_1 M$ for $A \subseteq M$.

\end{enumerate}
\end{definition}

\noindent For the graded Lie algebras $M \in \K_3$, that we will consider later, we will have: \\
if $\delta_3(A) \le \delta_3(C)$ for all finite $C$ with $A \subseteq C \subseteq M$, then \\
$\delta_2(A) \le \delta_2(C)$ for all finite $C$ with $A \subseteq C \subseteq M$.

\begin{lemma}\label{rulesstrong}
Let $M$ be a graded c-nilpotent Lie algebra over a finite field $K$.
We assume for all (i-1)-strong substructures  of $M$ submodularity for $\delta_i$ is true. Furthermore for all  substructures
$A, C,E$ of $M$,where $A,C$ are finite, we have:
\begin{enumerate}
\item[(0)] If $A\subseteq E$, then there is some $A \subseteq A' \le_{i-1} E$ and $\delta_i(A') \le \delta_i(A)$.
\item[(2,i-1)] If $A \le_{i-1} C \le_{i-1} M$, then $A \le_{i-1} M$.
\item[(3,i-1)] If $A , C \le_{i-1} M$, then $A \cap C \le_{i-1} M$.
\end{enumerate}
Then for all  substructures
$A, C, E$ of $M$, where $A,C$ finite the following is true:
\begin{enumerate}
\item If $C \le_i M$ and $E \le_{i-1} M$, then $E \cap C \le_i E$.
\item If $A \le_i C \le_i M$, then $A \le_i M$.
\item If $A , C \le_i M$, then $A \cap C \le_i M$.

\end{enumerate}
\end{lemma}

\begin{proof}

ad 1) By assumption $C,E \le_{i-1} M$. Choose any $D$ such that $E \cap C \subseteq D \subseteq E$. By (3,i-1)  we have $E \cap C \le_{i-1} M$. By (0) we have some $D'$, such that $D \subseteq D' \le_{i-1} E$ and $\delta_i (D') \le \delta_i (D)$. By (2,i-1)
we get $D' \le_{i-1} M$. By submodularity for (i-1)-strong substructures of $M$ we have
\[ \delta_i(\langle D' C \rangle) \le \delta_i(D') + \delta_i(C) - \delta_i(D' \cap C) \]

Note $D' \cap C = D \cap C = E \cap C$. Hence
\[      0 \le \delta_i(\langle D'C \rangle) -\delta_i(C) \le \delta_i(D') - \delta_i(E \cap C) ,\]
since $C\le_i M$. It follows $\delta_i(E \cap C) \le \delta_i(D') \le \delta_i(D)$ for all $D$ between $E \cap C$ and $E$.\\

ad 2) By (2,i-1) we have $A \le_{i-1} M$. Also $C \le_{i-1} M$.  Consider $A \subseteq E \subseteq M$. If $E \subseteq C$ or $C \subseteq E$, then $\delta_i(A) \le \delta_i(E)$ or $\delta_i(A) \le \delta_i(C) \le \delta_i(E)$ respectively. Otherwise we apply (0) and obtain $E \subseteq E' \le_{i-1} M$ with  $\delta_i(E') \le  \delta_i(E)$.  By submodularity
of $\delta_i$ for (i-1)-strong substructures
\[ \delta_i(\langle E'C \rangle) \le \delta_i(E') + \delta_i(C) - \delta_i(E' \cap C) .\]
By $C \le_i M$ and since $\delta_i(A) \le \delta_i(E' \cap C)$
\[ 0 \le \delta_i(E'C) - \delta_i(C) \le \delta_i(E') - \delta_i(E' \cap C) \le \delta_i(E') - \delta_i(A).\] 
Hence $\delta_i(A) \le \delta_i(E') \le \delta_i(E)$.\\

ad 3) By 1) we have $A\cap C \le_i A \le_i M$ and by 2) $A \cap C \le_i M$.

\end{proof}

\noindent Note 1-strong substructures are substructures. Submodularity for $\delta_2$ for substructures is shown in
Lemma \ref{subdelta2}. Hence

\begin{cor}\label{2.13}
 For all  substructures
$A, C, E$ of $M$, where $A,C$ are finte the following is true:
\begin{enumerate}
\item If $C \le_2 M$ and $E \subseteq M$, then $E \cap C \le_2 E$.
\item If $A \le_2 C \le_2 M$, then $A \le_2 M$.
\item If $A , C \le_2 M$, then $A \cap C \le_2 M$.

\end{enumerate}
\end{cor}

Let $M$ be $F(X)/\langle \langle Y \rangle \rangle$, where $X$ is a generating o-system with 
$o-dim_i(M) = \\
\mid X_i\mid$ and $Y$ is an ideal basis of $ker(f)$, where $f$ is as above $f: F(X) \rightarrow M$, $f$ is onto, and $f(X)$ is a generating o-system of $M$.\\
Let $M'$ be $M/M_{i+1} \oplus \ldots \oplus M_c$ and $\tau$ the canonical homomorphism of $M$ onto $M'$. Then $\tau$ is one-to-one on $M_1 \oplus \ldots \oplus M_i$ and $ker(\tau) = M_{i+1}\oplus \ldots \oplus M_c$.

\begin{lemma}\label{canpair}
Assume $M \subseteq N$ and $F(X),ker(f)$ is a canonical pair for $N$ as above. Then there is a subalgebra
$H$ of $F(X)$, such that $H,H\cap ker(f)$ is a canonical pair for $M$.
\end{lemma}
\noindent Note that $H$ is free by Theorem\ref{sirsovwitt}.
\begin{proof}
Choose an generating o-system $Z$ for $M$ and let $Z^0$ be a $f$-preimage of $Z$ in $F(X)$. Let $H$ be the substructure of $F(X)$ generated by $Z^0$.
\end{proof}   

\begin{definition}
If $M \subseteq N$, then $X = \cup_{1 \le i \le c} X_i$ is an o-system for $N$ over $M$,if 
$U_i(X_i)$,
$X_1$ is minimal with $\langle M_1 X_1 \rangle_1 = N_1$ and $X_i$ is minimal with $\langle M X_1 \ldots X_i \rangle_i = N_i$.
Then $o-dim(N/M) = \sum_{1 \le i\le c} \mid X_i \mid$ and $o-dim_i(N/M) = \mid X_i \mid$. \\
In the context of Lemma \ref{canpair} $Y = \cup_{2 \le i \le c} Y_i$ is an ideal basis of $N$ over $M$, if $Y_i$ is a 
vector space basis of  $F(X)_i \cap ker(f)$ over $F(X)_i \cap (\langle \langle(ker(f) \cap H), Y_2, \ldots , Y_{i-1} \rangle \rangle)$.
Then $ideal-dim_i(N/M) = \mid Y_i \mid$ and $ideal-dim(N/M) = \sum_{2 \le i \le c} ideal-dim_i(N/M)$.

\end{definition}

\noindent We extend Definition \ref{strong}

\begin{definition}\label{strong2}
Let $M$ be a graded c-nilpotent Lie algebra over a finite field. 
Let $A$ be a finite substructure.
\begin{enumerate}
\item[(3)] $A$ is strong in $M$ restricted by $k$ (short $A \le^k M$), if for all 
$A \subseteq C \subseteq M$ 
with  $\sum_{1 \le i \le c-1} o-dim_i(C/A) \le k$, we have $\delta(A) \le \delta(C)$. 
\end{enumerate}
\end{definition}

\noindent Note that $\delta(C) = \delta(\langle C_1, \ldots , C_{c-1} \rangle) + o-dim_c(C)$.

\begin{definition}
Let $\K_c$ be the class of all graded c-nilpotent Lie algebras $M$ over a  fixed finite field $K$ 
considered as $L$ - structures.  such that the following is true:\\
For $A \subseteq M$ we have $min\{o-dim(A),2\} \le \delta(A)$

\end{definition}

\begin{lemma}\label{nozerodivisors} 
Assume $M \in \K_c$.
Given $i + j \le c$  in $M$ holds
\[ \forall x y \left( U_i(x) \wedge U_j(y) \wedge x \neq 0 \wedge y \neq 0 \rightarrow 
([x,y] \neq 0 \vee \bigvee_{k \in K, k \neq 0} x = ky)\right) .\] 
\end{lemma}
\noindent We say that there are no homogeneous zero-divisors.
\begin{proof}
Assume $U_i(a)$, $U_j(b)$, $i + j \le c$ and $i < j $ or $i = j $ and $a$ and $b$ are linearly 
independent . By assumption $o-dim(\langle a, b \rangle) =\delta(\langle a, b \rangle) = 2$. Hence $[a,b] \not= 0$.
\end{proof}

\begin{definition} Let $\h_c$ be the subclass of all $M \in \K_c$ with $M = \langle M_1, \ldots,  M_{c-1}\rangle$.
$\K_c^{fin} $ and $\h_c^{fin}$ are the subclasses of finite sutructures.
\end{definition}

$\K_c$ and $\K_c^{fin}$ are closed under substructures.

\begin{cor}
\begin{enumerate}
\item Let $M$ be a graded 2-nilpotent Lie algebra ovr a finite field. $M \in \K_2$ if and only if 
for every $A = \langle A_1 \rangle \subseteq M$ with $lin-dim(A_1) \ge 2$ we have $\delta_2(A) \ge 2$.
\item Let $M$ be a graded 3-nilpotent Lie algebra over a finite field. $M \in \K_3$ if and only if 
for every $A = \langle A_1 A_2 \rangle \subseteq M$ with $o-dim(A) \ge 2$ we have  $\delta_3(A) \ge 2$.
\end{enumerate}
\end{cor}




\section{Free amalgamation}

\noindent Let $\K$ be a class of structures as e.g. the class of all c-nilpotent graded Lie algebras over a fixed field, $\K_c$, or $\K_c^{fin}$.

\begin{definition} We define the amalgamation and the free amalgamation  for $\K$.

\begin{description}

\item[AP] {\em Amalgamation Property} Assume  $g_0: B \to A$ and $g_1: B \to C$ are embeddings for
$A, B, C \in \K$. Then there are some $D$ in $\K$ and embeddings $f_0: A \to D$ and $f_1: C \to D$ such that $f_0 \circ g_0 = f_1 \circ g_1$ for $B$.

\item[ APS] We have the strong amalgamation property  for $\K$ if in {\bf AP} 
$f_0(A) \cap f_1(C) = f_0 \circ g_0 (B) = f_1 \circ g_1(B)$ holds. 

\item[Free Amalgam] Let  $A, B, C, D \in \K$ and assume that $B$ is a common substructure of
$A$ and $C$. If $D$ is generated by $A$ and $C$ with $A \cap C = B$, then 
$D$ is the  free amalgam of $A$ and $C$ over $B$ (short $D = A \otimes_B C$) in $\K$, if for  all  homomorphisms
$f: A \to E$ and $g:C \to E$  into some $E \in \K$ with 
$f(b) = g(b)$ for $b \in B$ there is a homomorphism $h: D \to E$ that extends $f$ and $g$. 

\item[Closed] $\K$ is closed under free amalgamation, if for  $A,B,C \in \K$ and embeddings $g_0 : B \to A$ and $g_1 : B \to C$, there exists a  free amalgam $A' \otimes_{B'} C'$  in $\K$ and isomorphisms $f_0 : A \to A'$ and $f_1 : C \to C'$ , such that $f_0 \circ g_0(b) = f_1 \circ g_1(b)$ for $b \in B$ maps $B$ onto $B'$ . 
\end{description}

\end{definition}

\noindent The free amalgam is a strong amalgam by definition. The homomorphism $h: D \to E$ in the definition is unique, since $D$ is generated by $A$ and $C$.
Note that $A \otimes_B C$ is uniquely determined up to isomorphisms, if it exists. \\
We define as in \cite{Bau16}:
\begin{definition} For subsets $A, B, C$ in a structure $M$ we define
\[ A \unab_B C\] if and only if  
\[\langle A B C \rangle = \langle A B \rangle \otimes_{\langle B \rangle} \langle B C \rangle .\]
\end{definition}.

\noindent Let $L$ be countable.
K.Tent and M.Ziegler  defined a stationary independence relation for the investigation of automorphism groups in \cite{TZ12}.  We consider  finite subsets $A, B, C ,D$ of a $L$-structure $M$.

\begin{definition}\label{statind}  
A relation $A \ind_B C$ for finite subsets of $M$ is called  a stationary independence relation  in $M$ if it fulfils the following properties. 
\begin{description}
\item [Inv] {\em Invariance} $A\ind_B C$ depends only on the elementary type of $A,B,C$.

\item[Mon]{\em Monotonicity} $A \ind_B CD$ implies $A \ind_B C$ and $A\ind_{BC} D$.

\item[Trans] {\em Transitivity} $A \ind_B C$ and $A \ind_{BC} D$ imply $A \ind_B CD$.

\item[Sym] {\em Symmetry} $A \ind_B C$ if and only if $C\ind_B A$.

\item[Ex] {\em Existence} For $A, B, C$ there is some $A'$ in $M$ such that $\tp(A/B) = \tp(A'/B)$ and 
$A' \ind_B C$.

\item[Stat] {\em Stationarity}  If  $\tp(A/B) = \tp(A'/B)$, $ A \ind_B C$, and $A' \ind_B C$, then 
$\tp(A/BC) = \tp(A'/BC)$.
\end{description}
\end{definition}

\noindent In \cite{Bau16} the following is shown. Note that the graduation is essential for the amalgamation.

\begin{theorem}\label{Lieam1}
The class of  c-nilpotent graded Lie algebras over a field $K$ is closed under  free amalgamation. 
\end{theorem}

\begin{theorem}\label{Lieam2}
If $K$ is a finite field, then the \Frai limit $M_0$ of all finitely generated c-nilpotent graded Lie algebras 
exists. 
If $\C^{uni}$ is a monster model of $Th(M_0)$, then the free amalgam defines a stationary independence relation  $\unab$in $\C^{uni}$. 
\end{theorem}

\noindent Then every graded Lie algebra we consider is a substructure in $\C^{uni}$, and 
if $B\subseteq A$,  then every embedding of $B$ in $\C^{uni}$ can be extended to $A$. 
Furthermore $tp(A/B)$ is completely determined by it's quantifier-free part.
That means we can use the properties of a stationary independence relation for the free amalgam in the class of c-nilpotent graded Lie algebras over $K$.
\begin{definition}
Let $M$ be a graded c-nilpotent Lie algebra. By a vector space basis of $M$ we mean 
the union of the  vector space bases $X_i$ of all
subspaces $M_i$ of the graduation.
\end{definition}

\noindent In the proof of the Theorem \ref{Lieam1} we prove the following Major Case. The final construction of the free amalgam is an iteration of it.

\begin{description}
\item[Major  Case] Assume $A = \langle B a \rangle$ and $C = \langle B e \rangle$ with $U_i(a)$ , $U_j(e)$, and $i,j < c$. 
Furthermore we have $[a,b] \in B$ and $[e,b] \in B$ for $b \in B$. 
Then the free amalgam $D$ of $A$ and $C$ over $B$ exists. Let $X^B$ be a homogeneous vector space basis of $B$, Let $Y$ be a vector space basis of the free graded 
c-nilpotent Lie algebra freely generated by $a$ and
$e$.  We can assume that $Y$ is a set of basic monomials over $a,e$. Then $X^B Y$ is a vector space basis of $D$ and the Lie multiplication is inductively defined by the Jacobi identity, using $[a,b] \in B$ and $[e,b] \in B$ for $b \in B$. 
\end{description}

\begin{cor}\label{c=2,z}
Let $C, B , A$ be finite graded 2-nilpotent Lie algebras over a field. Let $X^B = X_1^B X_2^B $ be a vector space basis of $B$, $X^{A} =  X_1^{A} X_2^{A} $ a vector space basis of $A$ over $X^B$, and $X^C =  X_1^C X_2^C $ a vector space basis of $C$ over $X^B$. Then 
\[X^B X^C X^{A}  \{ [x,y] : x \in X_1^C , y \in X_1^{A} \} \]
is a vector space basis of $A \otimes_B C$.
\end{cor}

\noindent Using  Corollary \ref{corsubdelta2} we obtain:

\begin{cor}\label{c=2,y}
Let $M$ be a 2-nilpotent Lie algebra and
let $B \subseteq A, B \subseteq  C$ be subalgebras of $M$ with $A \cap C = B$. Then 
the following are equivalent:
\begin{enumerate}
\item $\langle AC \rangle = A \otimes_B C$. 
\item $\delta(\langle A C\rangle) - \delta(C) = \delta(A) - \delta(B)$. 
\item There are no elements of 
$\langle A_1 C_1 \rangle \setminus \langle  A_1 \rangle + \langle C_1 \rangle$ in $A_2$ or in $C_2$. 
Furthermore if $\langle A_1 C_1 \rangle = F(A_1 C_1)/I$, then $I \subseteq \langle A_1 \rangle^{F(A_1 C_1)} + \langle C_1 \rangle^{F(A_1 C_1)} $.
\end{enumerate}
\end{cor} 

\begin{proof} (2) and (3) are equivalent by Corollary \ref{corsubdelta2}.
If $\langle AC \rangle = A \otimes_B C$, then by Corollary \ref{c=2,z} we obtain (3).
Finally we show that (3) implies (1). Let $f$ be a homomorphism of $A$ into $E$ and $g$ of $C$ into $E$ such that $f$ and $g$ coincide over $B = A \cap C$. As above let $X^B = X_1^B X_2^B $ be a vector space basis of $B$, $X^{A} =  X_1^{A} X_2^{A} $ a vector space basis of $A$ over $X^B$, and $X^C =  X_1^C X_2^C $ a vector space basis of $C$ over $X^B$. Then 
by the  conditions of (3) $\{[x,y]: x \in X_1^{A}, y \in X_1^C \}$ is linearly independent over $\langle A_1 \rangle_2 + \langle C_1 \rangle_2$. Therefore and by the second condition
the homomorphism $h$ of $F(X_1(B) X_1(A) X_!(C))$ into $E$, given by $f$ and $g$, induces a homomorphism of $\langle A_1 C_1 \rangle$ into $E$. Then it is no problem to extend this to 
$\langle A C \rangle$.
\end{proof}

\begin{cor}\label{c=3,a}
Let $C, B , A$ be finite graded 3-nilpotent Lie algebras over a field with $A \cap C= B$. Then the following are equivalent:
\begin{enumerate}
\item $\langle A C \rangle = A \otimes_B C$.
\item  Let $X^B = X_1^B X_2^B X_3^B$ be a vector space basis of $B$, $X^{A} =  X_1^{A} X_2^{A} X_3^{A}$ a vector space basis of $A$ over $X^B$, and $X^C =  X_1^C X_2^C X_3^C$ a vector space basis of $C$ over $X^B$. Assume $X_1^{A}$ and $X_1^C$ are ordered. Then the following is a vector space basis of $\langle  A C \rangle$:
\[X^B X^C X^{A}  \{ [x,y] : x \in X_1^C , y \in X_1^{A} \}  \{ [w,z] : w \in X_2^C , z \in X_1^{A}  \:or \:
  w \in X_2^{A} , z \in X_1^C \}\]
\[\{ [[x,y],z] : x = z \in X_1^C, y \in X_1^{A} \: or \] \[ x = z \in X_1^{A}, y \in X_1^C  \: or \:
x < z \in X_1^C , y \in X_1^{A} \: or  \: x < z \in X_1^{A} , y \in X_1^C \}\]
\end{enumerate}
\end{cor}

\begin{proof}
It is easily seen, that there exists a graded Lie-algebra $M$ with the vector space basis (2). We have only to define 
$[x,y]$ in a canonical way and then we have to show that the Jacobi identity is true. Using the definiton we get that
$M \cong A \otimes_B C$.
\end{proof}

\noindent The following property of the free amalgam of 3-nilpotent graded Lie-algebras over a field
will be later useful:

\begin{theorem}\label{cap}
In a 3-nilpotent graded Lie algebra over a field we consider $U \subseteq V $, 
$\bar{a}$ a sequence of  elements, $\langle U \bar{a} \rangle
\cap V = U $, $B \subseteq U$, and $C \subseteq V$ such that
\begin{enumerate}
\item $\langle V \bar{a} \rangle = V \otimes_C \langle C \bar{a} \rangle$,
\item $\langle U \bar{a} \rangle = U \otimes_B \langle B \bar{a} \rangle$.
\end{enumerate} 
Then $\langle U \bar{a} \rangle = U \otimes_D \langle D \bar{a} \rangle$, where $D = C \cap B$. \\
The same is true for $c = 2$.
\end{theorem}

\begin{proof} W.l.o.g. $\bar{a}$ is a sequence of homogenous elements, since there are projections in the language.
By the assumptions a set of monomials  of homegeneous elements from 
$\langle U \bar{a} \rangle$    that is linearly independent over $U$
is linearly indepentent over $V$. Then $\langle \bar{a} \rangle^{lin} \cap U = \langle \bar{a} \rangle^{lin} \cap V $. Hence we can assume w.l.o.g., that $\bar{a}$ is linearly independent over
$U$ and $V$. Furthermore we suppose w.l.o.g., that $\bar{a}_2$ is linearly independent over
$V_2 + \langle \bar{a}_1 \rangle_2$ and $\bar{a}_3$ is linearly independent over
$\langle \bar{a}_1 \bar{a}_2 \rangle_3 + V_3$.\\
Assume w.l.o.g. $\bar{a} = \bar{e} \bar{e'} \bar{e''}$ with the following properties:
\begin{enumerate}
\item [(3)] $\bar{e}$ is a generating o-system for $\langle V \bar{a}\rangle$ over $V$.\\
Then $\bar{e_1} = \bar{a_1}$ and $\bar{e}$ is an o-system 
over $U$.
\item [(4)]$\bar{e} \bar{e'}$ is a generating o-system of $\langle U \bar{a} \rangle $ over $U$.
\end{enumerate}
Note that $\bar{e'}_1$ and $\bar{e''}_1$ are empty.\\
Assume $e \in \bar{e'}_2 \bar{e''}_2$. Then w.l.o.g. 
\[ e = \sum_j [a_1^j, v_1^j] + v_2 , \]
where $v_2 \in V_2$ and $v_1^j \in V_1$.\\
By (1) and Corollary \ref{c=3,a} we get $v_1^j \in C_1$ and 
$v_2 = e - \sum_j [a_1^j, v_1^j] \in \langle C \bar{a} \rangle \cap V = C$.\\
If $e \in \bar{e''}_2$, then $e \in \langle U \bar{a}_1 \rangle$. Then by (2) and Corollary \ref{c=3,a}
$v_1^j \in B_1$ and $v_2 \in \langle B \bar{a}  \rangle \cap U = B$. By the considerations above 
$v_1^j \in D_1$ and $v_2 \in D_2$.\\
We can rewrite (1) as $\langle V \bar{a} \rangle = V \otimes_C \langle C \bar{a}_1 \bar{e}_2 \bar{a}_3
\rangle$ \\
and (2) as $\langle U \bar{a} \rangle = U \otimes_B \langle B \bar{a}_1 \bar{e}_2 \bar{e'}_2 \bar{a}_3
\rangle$.\\
Let $X_i^D$ be a vector space basis of $D_i$. We can extend $X_i^D$ by $X_i^B$, such that 
$X_i^D X_i^B$ is a vector space basis of $B_i$. Analogously we get $X_i^{C \cap U}$, such that 
$X_i^D X_i^{C \cap U}$ is a vector space basis of $ C_i \cap U_i$. 
Let $X_i^D X_i^{C \cap U} X_i^C$ be a vector space basis of $C_i$.
Then 
$ X_i^C X_i^{C \cap U} X_i^D X_i^B$ is a vectorspace basis of $\langle C B \rangle$. \\
Let $X_i^D X_i^{C \cap U} X_i^B X_i^U$ be a vector space basis of $U_i$.
Let $X_i^D X_i^{C \cap U} X_i^B X_i^U X_i^C X_i^V$ be a vector space basis of $V_i$.\\
To get a vector space basis 
of $\langle C \bar{a}_1 \rangle_2 $ over $C_2$, we choose $Y_2^D Y_2^{C \cap U} Y_2^C$
linearly independent, such that $Y_2^D$ is a vector space basis of $\langle D \bar{a}_1 \rangle_2$
over $C_2$, $Y_2^{C \cap U}$ is a vectorspace basis of $\langle (C \cap U) \bar{a}_1 \rangle_2$ over 
$C_2 Y_2^D$, and $Y_2^C$ is a vector space basis of $\langle C \bar{a}_1 \rangle_2$ over 
$C_2 Y_2^D Y_2^{C \cap U}$. By (1) $Y_2^D Y_2^{C \cap U} Y_2^C$ is a vector space basis over $V_2$.\\
By Corollary \ref{c=3,a} and (1) we get that the union of the following sets is a 
vector space basis for $\langle V \bar{a} \rangle_2 $:\\
$X_2^D X_2^{C \cap U} X_2^C X_2^B X_2^U X_2^V$  for $V_2$, \\
$X_2^D X_2^{C \cap U} X_2^C Y_2^D Y_2^{C \cap U} Y_2^C \bar{e}_2$ for 
$\langle C \bar{a} \rangle_2$, and \\
$\{ [x,y] : x \in X_1^B X_1^U  X_1^V , y \in \bar{a}_1\}$.\\
If we use (2) and Corollary \ref{c=3,a},\\
then we get the following vector space basis for 
$\langle U \bar{a} \rangle_2$, as the union of:\\
$X_2^D X_2^B X_2^{C \cap U}  X_2^U $  for $U_2$, \\
$X_2^D X_2^{B}  Y_2^D  \{ [x,y] : x \in X_1^B, y \in \bar{a}_1 \} \bar{e}_2 \bar{e'}_2$ for 
$\langle B \bar{a} \rangle_2$, and \\
$\{ [x,y] : x \in  X_1^U  X_1^{U \cap C} , y \in \bar{a}_1\}$. \\
We can rewrite the vector space basis as the union of:\\
$X_2^D X_2^{C \cap U} X_2^B X_2^U $  for $U_2$, \\
$X_2^D   Y_2^D   \bar{e}_2 \bar{e'}_2$ for 
$\langle D \bar{a} \rangle_2$, and \\
$\{ [x,y] : x \in  X_1^B X_1^U  X_1^{U \cap C} , y \in \bar{a}_1\}$. \\
By Corollary \ref{c=3,a} we have $\langle U^* \bar{a}^* \rangle = U^* \otimes_{D^*} 
\langle D^* \bar{a}^* \rangle$.\\
\noindent Note, that $Y_2^{C \cap U} = \{ [x,y] : x \in X_1^{C \cap U}, y \in \bar{a}_1 \}$.\\
Furthermore we have shown, that $Y_2^B = \{[x,y] : x \in X_1^B, y \in \bar{a}_1  \}$ is a 
vector space basis of $\langle B \bar{a}_1 \rangle_2 $ over $\langle B_2 Y_2^D \rangle_2$.\\

\noindent We consider $e \in \bar{e'}_3 \bar{e''}_3$. Then w.l.o.g.
\[ e = \Delta + v_3, \]
where $v_3 \in V_3$ and $\Delta$ is a linear combination of monomials of the form 
$[w,z]$ with $w \in V_2 , z \in \bar{a}_1$  or 
$z  \in  \bar{a}_2, w\in V_1 $ or $[[v,y],u]$ where $v,u \in V_1, y \in \bar{a}_1$  or $[[x,v],z]$ where
$x,z \in \bar{a}_1,  v \in V_1$.\\
By Corollary \ref{c=3,a} we get that $v,u,w$ and $v_3$ are in $C$ similarly as above. \\
If $e \in \bar{e''}_3$,
then (1) and (2) imply, that all these $u,v,w$ and $v_3$ are in $D$.\\ 
Hence we can replace (1) by $\langle V \bar{a} \rangle = V \otimes_C \langle C \bar{e} \rangle$
and (2) by $\langle U \bar{a} \rangle = U \otimes_B \langle B \bar{e} \bar{e'}\rangle$.\\
Now we define analogously as above, that $Y_3^D$ is a vector space basis of 
$\langle D \bar{a}_1 \bar{e}_2 \rangle_3$ over $C_3$ ,  $Y_3^{C \cap U}$ is a vector space basis 
of $(\langle C \cap U) \bar{a}_1 \bar{e}_2 \rangle_3$ over $C_3 Y_3^D$, and $Y_3^C$ is a
vectorspace basis of $\langle C \bar{a}_1 \bar{e}_2 \rangle$ over $C_3 Y_3^D Y_3^{C \cap U}$. 
Hence $Y_3^D Y_3^{C \cap U} Y_3^C$ is a vector space basis of $\langle C \bar{a}_1 \bar{e}_2 
\rangle_3$ over $C_3$ and by (1) over $V_3$. 
We assume, that $\bar{a}_1$ is ordered by $a_1^j < a_1^k$ for $j < k$.
Furthermore we order $X_1^B X_1^U X_1^V$.\\
By (1) and Corollary \ref{c=3,a} the union of the following sets is a vector space basis  (5) of 
$\langle V \bar{a} \rangle_3$:\\
$X_3^D X_3^{C \cap U} X_3^C X_3^B X_3^U X_3^V$  for $V_3$, \\
$X_3^D X_3^{C \cap U} X_3^C Y_3^D Y_3^{C \cap U} Y_3^C \bar{e}_3$ for 
$\langle C \bar{a} \rangle_3$,  \\
$\{ [x,y] : x \in X_2^B X_2^U X_2^V, y \in \bar{a}_1 \; or \; x \in Y_2^D Y_2^{C \cap U}  
Y_2^C \bar{e}_2, 
y \in X_1^B X_1^{U} X_1^V \}$,\\
$\{ [[a_1^j,z] , a_1^k] : j \le k , z \in X_1^B X_1^{U} X_1^V\},$\\
$\{[[x,a_1^j],y]: x \le y \in X_1^B X_1^{U} X_1^V  \}$.\\
Let $Y_3^B$ be a vector space basis  of $\langle B \bar{a} \rangle_3$ over $X_3^D X_3^B Y_3^D$.
By (2) und  (5) we have \\
\[ Y_3^B =  \{[x,y] : x \in Y_2^D  \bar{e}_2 \bar{e'}_2, y \in X_1^B \; or \; x \in X_2^B, y \in \bar{a}_1 \} \cup \]
\[ \{ [[a_1^j,z] , a_1^k] : j \le k , z \in X_1^B \} \{[[x,a_1^j],y]: x \le y \in X_1^B\} \bar{e}_3 \bar{e'}_3 \]
We apply   Corollary \ref{c=3,a} to (2) to get 
the union of the  following sets as a  
vector space basis (6) of $\langle U \bar{a} \rangle_3$: \\
$X_3^D X_3^B X_3^{C \cap U} X_3^{U}$ for $U_3$,\\
$X_3^D X_3^B Y_3^D Y_3^B$
 for 
$\langle B \bar{a} \rangle_3$, and \\
$\{ [x,y] : x \in X_2^{C \cap U} X_2^U , y \in \bar{a}_1 \; or \; x \in   Y_2^D
Y_2^B \bar{e}_2 \bar{e'}_2, 
y \in X_1^{C \cap U} X_1^{U}  \}$ \\
$\{ [[a_1^j,z] , a_1^k] : j \le k , z \in X_1^{C \cap U} X_1^{U} \},$\\
$\{[[x,a_1^j],y]: x \le y \in X_1^{C \cap U} X_1^{U}   \}$. \\
We can rewrite this vector space basis using the structure of $Y_2^B$ and $Y_3^B$, as 
described above:\\
$X_3^D X_3^B X_3^{C \cap U} X_3^{U}$ for $U_3$,\\
$X_3^D  Y_3^D \bar{e}_3 \bar{e'}_3$ for $\langle D \bar{a} \rangle_3$, and \\
$\{[x,y] : x \in Y_2^D  \bar{e}_2 \bar{e'}_2, y \in X_1^B X_1^{C \cap U} X_1^{U}
\; or \; x \in X_2^B X_2^{C \cap U} X_2^{U}, y \in \bar{a}_1 \} \\
\{ [[a_1^j,z] , a_1^k] : j \le k , z \in X_1^B X_1^{C \cap U} X_1^{U}\} \\
\{[[x,a_1^j],z]: x \le z \in X_1^B X_1^{C \cap U} X_1^{U}\} $. \\
By Corollary \ref{c=3,a} we get
\[  \langle U \bar{a} \rangle = U \otimes_D \langle D \bar{a} \rangle. \]
\end{proof}

\noindent We consider further consequences of Corollary \ref{c=3,a}.

\begin{cor}\label{botimesx1}
Let $B$ and $\langle x \rangle $ with $U_i(x)$ be finite graded 3-nilpotent Lie algebras. Let $Z = Z_1 Z_2 Z_3$ be an ordered vectorspace basis of $B$.\\
If $H$ is the ideal generated by $x$ in $D = B \otimes \langle x \rangle $, then 
it is freely generated by the following set $X$:
\begin{enumerate}
\item If $U_3(x)$,then $X = X_3 = \{x\}$.
\item If $U_2(x)$, then $X = \{x\} \cup \{[z,x] : z \in Z_1 \}$.
\item If $U_1(x)$, then $X = \{x\} \cup \{[z,x] : z \in Z_2 \} \cup \{[[z_1,x],z_2] : z_1 \le z_2\}$.

\end{enumerate}
The underlying vectorspace of $D = B \otimes \langle x \rangle $ is $B \oplus H$. In case (1) and (2) $X$ is also a vectorspace basis  of $H$ and in case (3)
$X \cup \{[[x,z],x]: z \in Z_1 \}$ is a vector space basis of $H$.
\end{cor}

\noindent The algebraic background is well known. See e.g. \cite{Bah}.

\begin{lemma}\label{botimesx2} 
Assume $B \in \K_3$, $\langle x \rangle \in \K_3$ with $U_i(x)$, and    $D = B \otimes \langle x \rangle$. Then 
\begin{enumerate}
\item $B \le D$,  $\delta(D) = \delta(B) + 1$.
\item $D \in \K_3$.
\end{enumerate}
\end{lemma}

\begin{proof}
As above let $H$ be the ideal generated by $x$ in $D = B \otimes \langle x \rangle $.\\
ad1)
Let $B \subseteq E \subseteq D$. Let $Y$ be a generating o-system for $B$. Then $B \simeq F(Y)/I$ and 
$D \cong F(Yx)/\langle \langle I \rangle \rangle^{F(Yx)}$. For $E$ we choose   
$W$ in $H$ such that $YW$ is a generating o-system for $E$. Then $E \cong F(YW)/\langle \langle I \rangle \rangle^{F(YW)}$.
Hence $\delta(B) \le \delta(E)$, $\delta_2(B) \le \delta_2(E)$, and $\delta(D) = \delta(B) + 1$.\\
\\
\noindent ad2)
Now we consider any $E = \langle E_1 E_2 \rangle \subseteq D$. There is an o-system $W$ over $B \cap E$, that generates 
$E$ over $E \cap B$. Then $W$ extends every o-system for $B \cap E$ to a generating o-system for $E$. We have $\delta(E) = \delta(B \cap E) + \mid W \mid$ .

\end{proof}

\begin{definition} Assume $b,e \in B$ with $U_i(b)$, $U_j(e)$, and $i < j \le c$ and there is no solution of
$[b,x] = e$ in B. Such a pair $b,e$ is called a divisor problem for $B$.
\end{definition}

\begin{definition}
$B(e:b)$ is $B \otimes \langle x \rangle$ factorized by $[b,x] = e$.
\end{definition}

\noindent We describe another way to define $B(e:b)$ for $B \in \K_c$.  
Let $\langle b e x\rangle$ be the graded Lie algebra defined by $U_{j-i}(x)$, and $[b,x] = e$. Then $\langle b e x\rangle \in \K_c $, $\langle b e \rangle  \le \langle b e x\rangle$, and $\delta(\langle b e \rangle) = \delta(\langle b e x\rangle)$. By Theorem \ref{Lieam1}
$B \otimes_{ \langle b e \rangle} \langle b e x\rangle$ exists. It is $B(e:b)$.

\noindent Now we consider again $\K_3$. As above we assume that $Z$ is an ordered vector space basis of $B$. Let  $b$ be
the first element of $Z$ and $e$ the second. We use the description of $B \otimes \langle x \rangle $ above.
Let $X^-$ be the subset of all elements of $X$, where $b$ does not occure. Let $H^-$ be the graded Lie algebra freely generated by $X^-$.

\noindent Note that the underlying vectorspace of $B(e:b)$ is $B \oplus H^-$ and the Lie multiplication is given the Lie multiplication in $B$ and in $H^-$ and by the action of $B$ on $H^-$ in $B \otimes \langle x \rangle$ factorized by   $[b,x] = e$.

\begin{lemma}\label{e:b} 
Assume $B \in \K_3$, $b \in U_i(B)$, $e \in U_j(B)$, $1 \le i < j \le 3$, and there is no solution of 
$[b,?] = e$ in $B$.
\begin{enumerate}
\item $B(e:b) \cong B \otimes_{\langle b, e \rangle} \langle  b, e, x \rangle$, definded as above.
\item The underlying vector space of $B(e:b)$ is $B \oplus H^-$. If $j - i = 2$, then $X^-$ is also a vector space basis of $H^-$. If $j - i = 1$, then $X^- \cup \{ [[x,z],x] : z \in Z_1 \setminus \{b\}\}$ is a vectorspace basis of $H^-$.
\item If $E \subseteq B(e:b)$, then $\delta_i(E) \ge \delta_i(E \cap B)$ and $\delta(E) \ge min\{2, o-dim(E)\}$.
\item $B \le B(e:b)$,  $\delta(B(e:b) = \delta(B) $.
\item $B(e:b) $  is in $\K_3$. 
\end{enumerate}
\end{lemma}

\begin{proof}
ad (1) For every suitable $A = \langle B a \rangle$ with $[b,a] = e$ there is a homomorphism $h$ of $B \otimes_{ \langle b e \rangle} \langle b e x\rangle$ onto $A$ with $h(B) = B$ pointwise and $h(x) = a$.\\

ad (2) is  clear.\\

ad (3) $D = B \otimes \langle  x  \rangle$. $Z$, $H$ , $X$ , $X^-$,and $H^-$ are defined as above. By definition $B(e:b) = D/\langle \langle [x,b] - e \rangle \rangle$.
Let $E$ be a subspace of $B(e:b)$.\\
If $b \notin E$, then there is a subset $W$, such that for  every o-sytem $V$ of $B \cap E$ 
the set $V \cup W$ is a
generating o-system for $E$ and $\delta(E) = \delta(E \cap B) + \mid W \mid$. For $\delta_2$ use a similar argument.
The assertion is proved for this case.\\
Now we assume that $b \in E$. First we assume that there is an element $d \in U_{j-i} (E) \setminus B$. Then w.l.o.g.
$d =  x + c$ with $c \in U_{j-i}(B)$. Then $E = \langle (E \cap B) d W \rangle$, where $VW$ is an o-system for $E$,
if $V$ is an o-system for $\langle (E \cap B) d \rangle$. We get 
$\delta(B \cap E) + \mid W \mid) \le \delta(E)$. \\
If $o-dim(B \cap E) = 1$, then $B \cap E = \langle b \rangle$. Then $\delta(E) = 2 + \mid W \mid$. \\
Finally $b \in E$ and $U_{j-i}(E) = U_{j-i}(E \cap B)$. By similar considerations as above $\delta(E) = \delta(E \cap B) 
+ \mid W \mid$. Similar considerations work for $\delta_2$.\\

(4) and (5) follow from (3).\\

\end{proof}

\begin{lemma}\label{c=2,a}
If $B \in \K_2$, then  $B(e:b)$ is in $\K_2$.
\end{lemma}

\begin{lemma}\label{4.15}
Assume $A \cap C = B$ are substructures in a 3-nilpotent graded Lie algebra $M$ over a field. Furthermore we define for 
$c \in C$   that $A^+ = \langle A c \rangle = A \otimes_B  \langle B c \rangle$ and 
$B^+ = \langle B c \rangle^C$. Then the following are equivalent:
\begin{enumerate}
\item $\langle AC \rangle = A \otimes_B C$.
\item $\langle AC \rangle = \langle A^+ C \rangle = A^+ \otimes_{B^+} C$.
\end{enumerate}
\end{lemma}

\begin{proof}
Assume (1): $A \unab_B C$. By  {\bf Mon} we have $A \unab_{\langle B c \rangle } C$. This is $\langle AC \rangle = \langle A^+ C \rangle = A^+ \otimes_{B^+} C$.\\
For the other direction we have $A \unab_{\langle B c \rangle } C$ and $A \unab_B \langle c \rangle$. 
By {\bf Trans } we get $A \unab_B C$, as desired.
\end{proof}

We will apply the Lemma in $\K_3$, where  $[b,c] = e$ for homogeneous $b, e$ in $B$.


\section{2-nilpotent graded Lie algebras}

We work in $\K_2$. $\delta$ is $\delta_2$. In Lemma \ref{subdelta2} submodularity for $\delta$ is shown.
Then Lemma \ref{rulesstrong} implies:

\begin{lemma}\label{c=2,b} Let $M$ be in $\K_2$.
Then for all  substructures
$A, C$ finite and $ E$ of $M$ the following is true:
\begin{enumerate}
\item If $C \le M$ and $E \subseteq M$, then $E \cap C \le E$.
\item If $A \le C \le M$, then $A \le M$.
\item If $A , C \le  M$, then $A \cap C \le M$.
\item If $A \le^k C \le M$, then $A\le^k M$.
\end{enumerate}
\end{lemma}

\begin{proof}
For (4) we need some proof. Assume $A \subseteq E \subseteq M $ and $o-dim_1(E/A) \le k$
or equivalently $ldim(E_1/A_1) \le k$. By (1) $\delta(E \cap C ) \le \delta(E)$. Since 
$o-dim_1(E \cap C)/A) \le k$, we have $\delta(A) \le  \delta(E \cap C)$.

\end{proof}

\noindent The lemma implies the existence of $CSS_2(A)$ in $M$ - the smallest strong subspace that contains $A$.\\
Lemma \ref{c=2,a} implies that for every $A \in \K_2$ there is some $A \subseteq C \in \h_2$, such that
$A \le C$, $\delta(A) = \delta(C)$, and $C$ is obtained by free adjunction of divisors. \\

\noindent All results for $c = 2$ in this paper are proved in \cite{Bau96}, exept Lemmas \ref{5.7} and
\ref{5.8}. In \cite{Bau09} different proofs are given, using a general method developed in that paper, to obtain 
the additive collaps. In both papers we work in $\h_2$.
Here we use $\K_2$. For $c = 2$ there is no big difference between the two approches. But for greater $c$ we have to start with $\K_c$.\\

\noindent In this section we summarize some results from \cite{Bau96}. By Definition \ref{strong2} $B \le^{ldim(A/B) + n} C$ means that $\delta(B) \le \delta(E)$ for all $B \subseteq E \subseteq C$ with $lin-dim(E_1/B_1 \le ldim(A/B) + n$.

\begin{theorem}\label{c=2,c}
\begin{enumerate}
\item $\K_2$ has the amalgamation property for strong embeddings.
\item If $A, B, C \in \K_2$, $B\le A$,and $B \le^{ldim(A/B) + n} C$, then there is an amalgam $D$ of $A$ and $C$ over $B$ in $\K_2$, such that $C \le D$ and $A \le^n D$. If no divisor problem in $B$ has solutions in both $A$ and $C$, then the free amalgam $D = A \otimes_B C$ fulfils the assertion. In this case
$\delta_2(D) = \delta_2(A) + \delta_2(C) - \delta_2(B).$
\end{enumerate}
\end{theorem}

\noindent Using this amalgamation we obtain the \Frai Hrushovski Limit $\M$ in the next theorem.

\begin{theorem}\label{c=2,d}
There is a countable structure $\M$ in $\K_2$, that satisfies the following condition:
\begin{description}
\item[rich] If $B \le A$ in $\K_2$ and there is a strong embedding $f$ of $B$ in $\M$, then there is a strong extension of $f$ that maps $A$ into $\M$.
\end{description}
$\M$ is uniquely determined up to isomorphisms.
\end{theorem}

\noindent If $b \in \M_1$ and $e \in \M_2$, then there exists $B \le \M$ with $b, e \in B$. Hence there is a strong embedding of $B(e:b) $ over $B$ into $\M$. That means $\M$ is closed under homogeneous divisors and 
$\M \in \h_2$.
We speak about rich structures.

\begin{theorem}\label{c=2,e}
If $M$ and $N$ are rich $\K_2$-structures, $\langle \bar{a} \rangle \le M$, $\langle \bar{b} \rangle \le N$,
and $\langle \bar{a} \rangle \cong \langle \bar{b} \rangle $, then 
\[(M,\langle \bar{a} \rangle)\equiv_{L_{\infty, \omega}}  (N,\langle \bar{b} \rangle ). \]
\end{theorem}

\noindent By the Theorem above there is a complete elementary theory $T_2$ of all rich $\K_2$-structures.
$T_2 = Th(\M)$.  Let $T_2(1)$ be an elementary description of $\K_2$ and
\begin{description}
\item[$T_2(2)$]  For all n  and all $B \le A$ in $\K_2$ there is an elementary sentences saying that every 
restricted by
$(n + ldim(A/B))$ strong embedding $f$ of $B$ in $M$ can be extended to an restricted by n  strong embedding of $A$ in $M$.
\end{description} 

\noindent The next theorem imlies that $T_2(1) \cup T_2(2)$ is an elementary axiomatization of $T_2$.

\begin{theorem}\label{c=2,f}
\begin{enumerate}
\item A rich $\K_2$-structure satisfies $T_2(1) \cup T_2(2)$.
\item Let $M$ be a model of $T_2(1)$. $M$ is rich if and only if M is an $\omega$ - saturated model of 
$T_2(2)$.
\end{enumerate}
\end{theorem}

\noindent Later we need the following:\\
\begin{definition}\label{5.6}
We work in some $M \in \K_2$. Let $A$, $C$, and $N$ be substrucures, where $A$ and $C$ are finite.
\begin{enumerate}
\item $\delta_2(A/C) = \delta_2(\langle A C \rangle) - \delta_2(C)$.
\item $\delta_2(A/N) = min\{\delta_2(A/C) : C\subseteq N, C finite, \langle A C \rangle \cap N = C \}.$
\end{enumerate}
\end{definition}

\noindent Note that in (1) $\delta_2(A/C)$ is not  equal to\\ 
$o-dim(A/C) - ideal-dim(A/C)$ in general.\\

\noindent
In general we have:\\
$o-dim(A/N) \le o-dim(A)$. \\
For $B \subseteq A$ it holds $o-dim(A) - o-dim(B) \le o-dim(A/B)$.

\begin{lemma}\label{5.7}
Let $M$ be a 2-nilpotent graded Lie algebra in $\K_2$. There exists a function $h(n)$ such that:\\
We consider substructures  $C$, $A$, and $B = C \cap A$ with $o-dim(A/B) \le n$ and
$\delta_2(A/C) \ge 0$.  Then there exist $D$, $K$, $D \cap K =H$, and $X_C \subseteq 
C_2 \cap \langle C_1 A_1 \rangle$
linearly independent over $\langle C_1 \rangle_2 + \langle A_1 \rangle_2 + B_2$, such that 
\begin{enumerate}
\item $B \subseteq H  \subseteq D \subseteq C$, $D_1 = C_1$, $C = \langle D X_C \rangle$, $K = \langle A H \rangle$, and 
$o-dim(H/B) \le h(n)$.
\item $\langle C A \rangle  = \langle D A \rangle = D \otimes_H \langle H A \rangle$. 
\item We can define $h(n) = (n + 1)n$.
\end{enumerate}
\end{lemma}

\begin{proof}
We define:\\
$r_A = ldim((\langle A_1 C_1 \rangle_2 \cap A_2)/\langle A_1 \rangle_2 + \langle C_1 \rangle_2$,\\
$r_C = ldim((\langle A_1 C_1 \rangle_2 \cap C_2)/\langle A_1 \rangle_2 + \langle C_1 \rangle_2$,\\
$r_B = ldim((\langle A_1 C_1 \rangle_2 \cap B_2)/\langle A_1 \rangle_2 + \langle C_1 \rangle_2$, and\\
if $\langle A_1 C_1 \rangle = F(A_1 C_1)/I$, then $s = ldim(I/(I \cap  (F(A_1) +  F(C_1))$. \\
By Corollary \ref{corsubdelta2} we get \\
$\delta_2(\langle A C \rangle) - \delta_2(C) = \delta_2(A) - \delta_2(B) - (r_A + r_C - r_B + s).$\\

\noindent Since $\delta_2(A/C) \ge 0$, we have $\delta_2(A/B) \ge 0$ and it is a bound for the size of
$r_A + r_C - r_B + s$.  \\
We consider an ideal basis $\{\psi_1,\ldots, \psi_s \}$ of $I$ over $I \cap (F(A_1) + F(C_1))$.
Let 
$\{ a_i^1 \in A_1:  1 \le i \le k \}$ be a vector basis  for $\langle C A \rangle_1$ over $C_1$. \\
$\psi_i$ has the following form: \\
\[ \psi_i = \sum_{l < j}  r_{l,j} [a_l^1,a_j^1] + \sum_l [a_l^1, b_l]  + c_i,\]
where $b_l\in C_1$ and $c_i \in F(C_1)_2$. We add all the $b_l$ and $c_i$ to $B$ and obtain $H^0$.\\
Let $d_1, \ldots, d_r$ with $r = r_A + r_C - r_B$ be elements in $\langle A_1 C_1 \rangle_2$ linearly 
independent over $\langle A_1 \rangle_2 + \langle C_1 \rangle_2$, such that \\
$d_1;\ldots,d_{r_B}$ are in $B_2$ , $d_{r_B + 1}, \ldots, d_{r_A}$  are  in $ A_2$ linearly independent over $B_2$ and $d_{r_A + 1}, \ldots, d_r$ are in $C_2$ linearly independnet over $B_2$ . For $1 \le i \le r_A$
\[ d_i = \sum_l [a^1_l,c^1_l] + a^2_i + c^2_i ,\]
where $c^1_l \in C_1$, $a^1_l \in A_1$, $a^2_i \in  \langle A_1 \rangle $, and $c^2_i \in \langle C_1 \rangle $. We add all these $c^1_l$ 
and $c^2_i$ to $H^0$ and obtain $H$. Define $\{d_{r_A + 1}, \ldots , d_r\} = X_C$. We choose $Y_C \subseteq C_2$  maximal linearly independent over $ \langle C_1 A_1 \rangle_2$ and define $D = \langle C_1 H Y_C \rangle$.  By Corollary \ref{c=2,y}
it follows the assertion:
\[ \langle C A \rangle  = \langle D A \rangle  = D \otimes_H \langle H A \rangle.\] 
We have added $(k + 1) \times (s + r_A)$ many elements to $B$. $k + 1 \le n+1$ and $s + r_A \le \delta_2(A) - 
\delta_2(B) \le o-dim(A) - o-dim(B) \le o-dim(A/B) \le n$.

\end{proof}

\begin{lemma}\label{5.8}
Let $M$ be a 2-nilpotent graded Lie algebra in $\K_2$.\\
We consider finite substructures  $C$, $A$, and  $B = C \cap A$ with $o-dim_1(A/B) = n$ and
$C \le M$. Then there exists $D$, $H$,  $K$, 
and $X_C X_B X_A$, such that 
\begin{enumerate}
\item $ H \subseteq D  \subseteq C$, $B_1 \subseteq H $, $D_1 = C_1$, 
$D \cap K = H$ and \\ 
$o-dim(H/B) \le h(n) = (n + 1)\times n$,
\item $H \subseteq K \subseteq \langle H A \rangle$, $K_1 = \langle H A \rangle_1$,
\item $\langle C A \rangle  = \langle D K \rangle = D \otimes_H \langle K \rangle$,
\item $X_C X_B X_A \subseteq \langle C_1 A_1 \rangle_2$ is linearly independent over 
$\langle D_1 \rangle_2 + \langle K_1 \rangle_2$, 
$X_B \subseteq B_2$, $\langle H X_B \rangle = \langle H B \rangle$,
\item $X_C \subseteq C_2$, $ \langle D X_B X_C \rangle = C$, and 
\item $X_A \subseteq A_2$, $\langle K X_B X_A \rangle = \langle H A \rangle$. 
\end{enumerate}
\end{lemma}

\begin{proof}
We use the following subalgebras of $\langle C A \rangle$:\\
$B^0 = \langle B_1, (B_2 \cap (\langle C_1 \rangle_2 + \langle A_1 \rangle_2)), Y_B \rangle$,where 
$Y_B$ is a maximal subset of $B_2$ linearly indepentent over $\langle C_1 A_1 \rangle_2$. \\
$C^0 = \langle C_1, B^0, Y_B, Y_C \rangle$, where $Y_C$ is a maximal subset of $C_2$ linearly independent
over $\langle C_1 A_1 Y_B\rangle_2$.\\
$A^0 = \langle A_1, B^0, Y_B, Y_A \rangle$, where $Y_A$ is a maximal subset of $A_2$ linearly independent
over $\langle C_1 A_1 Y_B\rangle_2$. Furthermore\\
$X_B$ is a maximal subset of $B_2 \cap \langle C_1 A_1 \rangle_2$ linearly independent over 
$B_2 \cap (\langle C_1 \rangle_2 + \langle A_1 \rangle_2)$.\\
$X_C$ is a maximal subset of $C_2 \cap \langle C_1 A_1 \rangle_2$ linearly independent over 
$C_2 \cap (\langle C_1 \rangle_2 + \langle A_1 \rangle_2 + X_B)$.\\
$X_A$ is a maximal subset of $C_A \cap \langle C_1 A_1 \rangle_2$ linearly independent over 
$A_2 \cap (\langle C_1 \rangle_2 + \langle A_1 \rangle_2 + X_B)$. Then\\
$B = \langle B^0 X_B \rangle$, $C = \langle C^0 X_B X_C \rangle$, and $A = \langle A^0 X_B X_A \rangle$.\\
Since $\langle C A \rangle = \langle C^0 A^0 \rangle$ and $\delta_2(C) = \delta_2(C^0) + \mid X_B X_C \mid$, we have $0 \le \delta_2(A/C) \le \delta_2(A^0/C^0)$. Furthermore $C^0 \cap A^0= B^0$.
Then we get as in the proof of 
Lemma \ref{5.7}\\
\[0 \le \delta_2\langle A^0 C^0 \rangle) - \delta_2(C^0) = \delta_2(A^0) - \delta_2(B^0) - s,\]
where $s$ is obtain as above:\\
If $\langle A_1 C_1 \rangle = F(A_1 C_1)/I$, then $s = ldim(I/(I \cap  (F(A_1) +  F(C_1))$. 
Since $C \le M$, we have $\langle C_1 \rangle  \le M$. Then 
\[ 0 \le \delta_2(\langle C_1 A_1 \rangle) - \delta_2(\langle C_1 \rangle ) = \delta_2(\langle A_1 \rangle)
- \delta_2(\langle C_1 \rangle \cap \langle A_1 \rangle) - s . \]
Hence $s \le lin-dim(A_1/B_1)$.\\
As above we add at most $(lin-dim(A_1/B_1) + 1) \times s$ many elements from $C^0$ to $B^0$ to obtain 
$H^0$ such that 
\[ \langle C A \rangle = \langle C^0 A^0 \rangle = \langle C^0 H^0 \rangle \otimes_{H^0}
\langle H^0 A^0 \rangle . \]
We define $H = H^0$, $D = \langle C^0 H \rangle$, and $K = \langle H A^0 \rangle $ and obtain
$\langle D K \rangle = D \otimes_H K$.\\ 
By construction $o-dim(H/B) \le (n+1)n$.

\end{proof}


\section{The functor $\F$}
\noindent If $A$ is in $\K_c$, then $A^* = A/A_c$ is in $\K_{c-1}$ and the canonical homomorphism 
$*$ of $A$ onto $A^*$ is injectiv on $A_1 \oplus \ldots \oplus A_{c-1}$ Let $\tau$ be this injectiv 
map of $A_1 \oplus \ldots \oplus A_{c-1}$ onto $A^*$. Let $F(X)$ be the free graded c-nilpotent Lie algebra over $K$  freely generated by an o-system $X$ , where $X_c = \emptyset$. Then $F(X)^*$ is the free 
graded (c-1)-nilpotent Lie algebra over $K$ freely generated by $\tau(X)$. \\

\noindent Now we assume, that $A \in \K_{c-1}$, $F(X) \in \K_c$ where $\tau(X)$ 
corresponds to  a 
generating o-system of $A$ 
and  $X_c = \emptyset$. Then  $A = F(X)^*/I$ for some 
$I$ in $F(X)^*$. That means
$F(X)^*,I$ is the canonical pair for $A$ in $\K_{c-1}$. 
Let $J$ be the ideal in $F(X)$ generated by $\tau^{-1}(I)$. We define:

\begin{definition} $\F(A) = F(X)/J$.
\end{definition} 

\noindent Let $B \in \K_c$ and $B = \langle B_1, \ldots, B_{c-1} \rangle$. If $A \in \K_{c-1}$ and $A = B^*$, then there is a homomorphism of $\F(A)$ onto $B$. In the case $A \in \K_2$ $A^*$ is a vector space and $\F(A^*)$ is the free 2-nilpotent Lie algebra over $A^*$.\\

\noindent $\F$ is a functor from $\K_{c-1}$ with embeddings into $\K_c$ with homomorphisms. Let $B$ and $A$ be in $\K_{c-1}$ and 
$f: B \rightarrow A$ be an embedding. W.l.o.g. $B$ is a substructure of $A$ :\\ 
 We consider the canonical pair $F(X)^*, I$ for $A$ as above. By Lemma\ref{canpair} there exits a subalgebra $H$ of $F(X)^*$ such that $H, I \cap H$ is the canonical pair for $B$. Let $K$ be the subalgebra of $F(X)$ generated by $\tau^{-1}(H)$. Then 
$\F(B) = K/J_{B}$ and $\F(A) = F(X)/J$ where $J_{B} = \langle \langle \tau^{-1}(I \cap H) \rangle \rangle^K$ and
$J$ as above.  Then $(J_{B})_i = J_i \cap K_i$ for $i < c$ and $(J_{B})_c \subseteq (J_c \cap K_c)$.
This gives the desired homomorhism $\F(f)$ of $\F(B) $ into $\F(A)$. 
For $c = 2$ this is an embedding. \\
Already for $c = 3$ there are examples where $\F(f)$ is not an embedding. See \cite{Ama}. This causes a lot of efforts.
We use the notation $\gamma_{B}^{A}$  for $\F(f)$.\\

\noindent Consicer $A \in \K_c$.
Let $A^-$ be $\langle A_1 \oplus \ldots \oplus A_{c-1} \rangle^{A}$. Then there is some $N(A) \subseteq \F(A^*)_c$, such that $A^- = \F(A^*)/N(A)$. Hence 
\[ \delta(A) = \delta_{c-1}(A) + ldim(A_c/A^-_c) - ldim(N(A)) .\]  \\
Furthermore $M \in \K_c$ implies $\F(M^*) \in \K_c$: This is a consequence of the next inequality. 
We consider $A_i$ as subspace in $M_i$ and in $\F(M^*)_i$. Then
\[ \delta(\langle A_1 A_2 \ldots A_{c-1}\rangle^M)  \le \delta(\langle A_1 A_2 \ldots A_{c-1} \rangle^{\F(M^*)}).\]

\noindent For $\delta_2$ submodularity is true
(Lemma \ref{subdelta2}). For $A \in \K_c^{fin}$ and $A \neq \langle 0 \rangle$ we have $\delta_2(A) > 0$. Assume $A \subseteq M \in \K_c$. By Corollary\ref{2.13}  $\bigcap\{C : A \subseteq C \le_2 M\}$ 
exists. It is the smallest 2-strong substructure of $M$, that contains $A$. We call it $CSS_2^M(A)$ or 
short $CSS_2(A)$. 

\begin{lemma}\label{3.2.a} Assume $D \le_2 M \in \K_c$ and $X Y$ is a vector space basis of $D_2$
over $\langle D_1 \rangle $. Then $\langle D_1  X \rangle \le_2 M$. Especially 
$\langle D_1 \rangle \le_2 M$.
\end{lemma}

\begin{proof}
Assume 
$\langle D_1 X \rangle \subseteq E$. There is $Y_0 \subseteq Y$ linearly independent over $E_2$,
such that $D_2 = \langle (D_2 \cap E_2) Y_0 \rangle^{lin}$. Then
\[ \delta_2(D) \le \delta_2(\langle E  Y_0 \rangle) \le \delta_2(E) + \mid Y_0 \mid,\] 
and 
\[ \delta_2(D) = \delta_2(\langle D_1 X \rangle) + \mid Y \mid.\]
Hence 
\[\delta_2(\langle D_1 X \rangle) \le \delta_2(E).   \]
\end{proof}

\noindent In the next lemma we describe the structure of $CSS_2(A)$ for $A \in \K_2$:

\begin{lemma}\label{3.2}
Let $B$  be a substructure of $M \in \K_2$.  
Then $CSS_2(B) = D = \langle D_1 B_2 \rangle$  and the following is true:
\begin{enumerate}
\item $B_1 \subseteq  D_1$, $B_2 \subseteq D_2 $ and $\langle D_1 \rangle \le_2 M$.
\item $CSS_2( \langle B_1 \rangle) = C = \langle C_1 \rangle \le_2 M$ and $C_1 \subseteq D_1.$
\item If $C_1 \not= D_1$, then 
\[ o-dim_2(\langle C_1B_2 \rangle) > o-dim_2(\langle D_1 B_2 \rangle) + \delta_2(\langle D_1 \rangle)
- \delta_2(\langle C_1 \rangle). \]
\end{enumerate}
\end{lemma}

\begin{proof}
(1) is true by Lemma \ref{3.2.a} and the definitions.\\
The first statement of (2) follows again from Lemma \ref{3.2.a}. If $C_1 \not\subseteq D_1$, then
\[ \delta_2(\langle D_1 C_1 \rangle) \le \delta_2(\langle D_1 \rangle) + \delta_2(\langle C_1 \rangle) 
- \delta_2(\langle D_1 \rangle \cap \langle C_1 \rangle)\]
by Lemma \ref{subdelta2}. Then 
\[ \delta_2(\langle D_1 C_1 \rangle) - \delta_2(\langle D_1 \rangle) \le \delta_2(\langle C_1 \rangle) 
- \delta_2(\langle D_1 \rangle \cap \langle C_1 \rangle) < 0,\]
since $B_1 \subseteq C_1 \cap D_1$, a contradiction.\\
If the inequality in (3) is not true, then $D = \langle C_1 B_2 \rangle$.

\end{proof}

\begin{lemma}\label{3.2.b}
Let $A$ be a 2-nilpotent graded Lie algebra with $A \in \K_2$. 
Assume that $X \subseteq A_1$ is a vector space basis of $A_1$ and $Y \subseteq A_2$ is a
vector space basis of $A_2$ over
$\langle X \rangle_2$. We consider $X$ and $Y$ also as subsets in $\F(A)$.
Then $\langle X Y \rangle^{\F(A)}_3$ is freely generated over $\langle X \rangle^{\F(A)}_3$ by 
$\{ [x,y] : x \in X , y \in Y \}$. In other words
$\F(A) = \langle  X   Y\rangle^{\F(A)} = \langle X \rangle^{\F(A)} \otimes \langle Y \rangle^{\F(A)}$.

\end{lemma}

\noindent In this paper we mainly consider the case $c = 3$. \\
The next Theorem gives an upper bound  for the size of the kernel of $\gamma_B^A$ for $A,B$ in $\K_2^{fin}$. It is essentially in \cite{Ama}.

\begin{theorem}\label{key3}
Let $B \subseteq A$ be 2-nilpotent graded Lie algebras in $\K_2^{fin}$, where $\F(A) \in \K_3$. Then

\begin{description}
\item [a)] If $B \le_2 A$, then $\gamma_B^{A}$ is an embedding of $\F(B)$ into $\F(A)$.
\item [b)] If $A = CSS_2^{A}(B)$ and $B \not= A$, then $ldim(ker(\gamma_B^{A})) < \delta_2(B) - \delta_2(A)$.
\end{description}

\end{theorem}

\noindent  We need the Theorem for 
$B = C^*$ and $A = D^*$ where $C \subseteq D$ in $\K_3$. Then  we have $\F(A) \in \K_3$. First we prove:

\begin{lemma}\label{3.4}
Let $\langle E_1, a \rangle  = A$ be in $\K_2^{fin}$, where $a \in A_1 \setminus E_1$.  Let $F(A_1)$ be the free
3-nilpotent graded Lie algebra over the vectorspace $A_1$.  
Furthermore $A = F(A_1)^*/I $.
Assume that
$C_1 \subseteq E_1 \subseteq A_1$, 
$c_1, \ldots, c_n$ is a vector space basis of $C_1$, such that $[c_i,a] + \psi_i$ with 
$\psi_i \in F(E_1)_2$ build a  vector space basis of
$I$ over $F(E_1)^*_2$.  Let $N_2(C)$ be $I \cap \langle C_1
\rangle_2$.\\ 
If $n = 1$, then $ldim(ker(\gamma_E^{A})) = 0$. Otherwise  
\[ ldim(ker(\gamma_E^{A})) \le ldim(N_2(C)) < n-1 .\]
\end{lemma}

\begin{proof}  By assumption $A = \langle A_1 \rangle $ and $E = \langle E_1 \rangle$ in $\K_2^{fin}$.
Mainly we work in $F(A_1)$ the free graded 3-nilpotent Lie algebra over $K$ freely generated by $A_1$.
In $F(A_1)$ we use the same notation for the $\tau^{-1}$-images of subsets and elements in $F(A_1)^*$, as
e.g. $I$  and $[c_i,a] + \psi_i$. \\
Then $\F(A) = F(A_1)/J$, where
$J = \langle \langle \ I \rangle \rangle^{F(A_1)}$. The non-zero elements of $ker(\gam E A) \subseteq \F(E_1)_3$ have preimages   $\mu \in J_3 \cap  \langle E_1 \rangle^{F(A_1)}_3$ that are not in 
$\langle \langle I \cap F(E_1) \rangle \rangle^{F(E_1)}$, where $F(E_1)$ is
considered as a subalgebra of $F(A_1)$. W.l.o.g.

\[ \mu = \sum_{1 \le i \le n} [e_i,([c_i, a] + \psi_i)] + [a, \theta ], \] 

\noindent where $e_i \in A_1$ and $\theta \in I_2 \cap E_2$. Since cancellation of $[a,[c_i,a]] $ is impossible, we have that $e_i \in E_1$. Furthermore $e_i \in C_1$. Otherwise $[e_i, [c_i, a]]$ cannot be killed by monomials 
$[a,[c_i,e_i]]$. Hence

\[ \mu = (\sum_{1 \le i \le n}( \sum_{1 \le j \le n, i \neq j} r_{i,j}^{\mu} [c_j, [c_i,a] + \psi_i]))  + [a, \theta^{\mu}] .\]

We have $i \neq j$, since we cannot cancel $[c_i,[c_i,a]]$. Since all monomials with $a$ have to vanish, we get

\[ \sum_{1 \le i \le n}( \sum_{1 \le j \le n, i \neq j} r_{i,j}^{\mu} [c_j, [c_i,a] ]) = - [a, \theta^{\mu}] \]

Then it follows that the $\theta^{\mu}$ are in $I_2 \cap \langle C_1 \rangle$.
Let s be the linear dimension of $ker(\gam E A)$.
If $n =1$, then there is no $\mu$ and $s = 0$. If we have s  many linearly independent elements in $ker(\gam E A)$, then we have s such $\mu$ as above linearly independent over $\langle \langle I_2 \cap E_2 \rangle \rangle^{F(E_1)}$. Hence the

\[ \sum_{1 \le i \le n}( \sum_{1 \le j \le n, i \neq j} r_{i,j}^{\mu} [c_j, [c_i,a] ) \]

\noindent are linearly independent and therefore also the $\theta^{\mu}$. By the definition of $\K_2$ 
and since $2 \le n$, we have
$\delta_2(\langle C_1 \rangle) \ge 2$. Therefore $s < n-1$ and $ldim(ker(\gam E A) = s < n-1 = \delta_2(B) - \delta_2(A)$.
\end{proof}

\noindent Now we show Theorem \ref{key3}

\begin{proof}
First we show, that it is sufficient to consider $\langle B_1 \rangle \subseteq \langle A_1 \rangle$.
We assume that the assertion is true in this case and $A \neq B$. \\

Case a) By Lemma \ref{3.2.a} we have $\langle B_1 \rangle \le_2 A$. Hence by assumption the
homomorphism of 
$\F(\langle B_1 \rangle)$ into $\F(\langle A_1 \rangle )$ is an embedding. 
By Lemma \ref{3.2.b} we can consider $\F(\langle A_1 \rangle)$ as a substructure of $\F(A)$ .\\
This is the first step of an induction on $o-dim_2(B)$ to show Case a). We can assume that $B_1 \neq
\langle 0 \rangle$, since n linearly independent elements in any $\F(A)_2$ generate a substructure 
$D$ with $D_1 = \langle 0 \rangle$, $D_2$ is a vector space with a basis of n elements, and 
$D_3 = \langle 0 \rangle$.\\
For the induction we consider $\langle B b \rangle \le A$ with $b \in A_2 \setminus B$ and the canonical
homomorphism $\gamma$ of $\F(B)$ into $\F(A)$ is an embedding. By Lemma \ref{3.2.b} 
$\langle B b \rangle = B \otimes \langle b \rangle$. We distingush the following:
\begin{enumerate}

\item[i)] There are $c \in B_1$ and $a \in A_1$ such that $[c,a] = b$.\\
Then $\langle B a \rangle \le A$. By induction there is an embedding $\gamma^+$ of $\F(\langle B a \rangle)$ into $\F(A)$ that extends $\gamma$. Furthermore we get 
$\langle B a \rangle = \langle B b \rangle \otimes_{\langle c , b \rangle} \langle c, a \rangle$. 
Then $\F(\langle b,c\rangle)$ can be considered as a substructure of $\F(\langle c,a \rangle)$ and of
$\F(\langle B b \rangle)$  since $\langle B b \rangle = B \otimes \langle b \rangle$. 
We see easily 
$\F(\langle B a \rangle) = \F(\langle B b \rangle ) \otimes_{\F(\langle b c \rangle)} \F(\langle c,a \rangle)$. Hence we obtain an embedding of $\F(\langle B b \rangle) $ into 
$\F(\langle B a \rangle)$ and then into $\F(A)$.

\item[ii)] There are no $c \in B_1$ and $a \in A_1$ with $[c,a] = b$.\\
By assumption there is some $c \in B_1$. We fix some new element $a$ with $U_1(a)$,
$[c,a] = b$, and  define 
$B^+ = \langle B b \rangle  \otimes_{\langle c,b \rangle } \langle a,c,b \rangle$ and 
$A^+ = A \otimes_{\langle c,b \rangle } \langle a,c,b \rangle$. \\
Then $B^+ \le A^+$.
By induction there is an embedding of $\F(B^+)$ into $\F(A^+)$. As above we can show that 
$\F(\langle B b \rangle)$ and 
$\F(A)$ can be considered as substructures of  $\F(B^+)$ and  $\F(A^+)$ respectively.

\end{enumerate}

\noindent Case b) 
Since $A = CSS_2(B)$ we have $A = \langle A_1 B \rangle$. We assume $B = \langle B_1 Y \rangle$,
where $Y = \{y_1, \ldots, y_n \}$ is a vector space basis of $B_2$ over $\langle B_1 \rangle$.
By induction we define the structure $\langle z y_i x_i \rangle$ by $U_1(x_i)$, 
$U_1(z)$, and $[z,x_i] = y_i$.
Furthermore let  $C$ be $\otimes^{1 \le i \le n}_{\langle z \rangle} \langle z x_i y_i \rangle$. We define
$B^+ =  C \otimes_{\langle y_1,\ldots,y_n \rangle} B$ and 
$A^+ =  C \otimes_{\langle y_1, \ldots,y_n \rangle} A$. Then 
$B^+ \subseteq A^+$, $\delta(B^+) = \delta(B) + 1$, and $\delta(A^+) = \delta(A) + 1$. \\
$B \le \langle B z \rangle \le \ldots \le \langle B z x_1 \ldots x_i \rangle$ and 
$A \le \langle A z \rangle \le \ldots \le \langle A z x_1 \ldots x_i \rangle$. Hence $B \le B^+$ and $A\le A^+$. By a) we can assume that $\F(B) \subseteq \F(B^+)$ and $\F(A) \subseteq \F(A^+)$. \\
We can apply the assumption since $B^+ = \langle B^+_1 \rangle$, $A^+ = \langle A^+_1 \rangle$, 
and $CSS^{A^+}(B^+) = A^+$. Hence
\[ lin-dim(ker(\gamma_B^{A}) \le lin-dim(ker(\gamma_{B^+}^{A^+}) < \delta(A) - \delta(B) .\]

\noindent Now we assume $B = \langle B_1 \rangle$ and $A = \langle A_1 \rangle$.  We use induction on $ldim(A_1/B_1)$:\\
{\bf Case 1)} $ldim(A_1/B_1) = 1$\\
We use Lemma \ref{3.4}. Let $B$ be $E$.

ad a) Since $B \le_2 A$, we have $n \le 1$. By Lemma \ref{3.4} $ldim(ker(\gamma_B^{A})) \le 0$, as desired.

ad b) Again by Lemma \ref{3.4} $ldim(ker(\gamma_B^{A})) < n-1 = \delta_2(B) - \delta_2(A)$. \\
Now we distinguish two cases for the induction step.

{\bf Case 2)} There is some $D = \langle D_1 \rangle $ such that $B \subseteq D \subseteq A$, $B \neq D \neq A$, and $\delta_2(A) \le \delta_2(D) \le \delta_2(B)$.

ad a) Since $B \le_2 A$ we have $\delta_2(B) = \delta_2(D) = \delta_2(A)$. Then $B \le_2 D$ and $D \le_2 A$.
By induction $\gam B D$ and $\gam D A$ are embeddings. Hence $\gam B A$ is an embedding.

ad b) We choose $D_1$ minimal with the properties of Case 2). We have $A = CSS_2^{A}(D)$ and $\delta_2(A) < \delta_2(D) $. By the minimality of $D$ we get either $\delta_2(D) < \delta_2(B)$ and  $CSS_2^{D}(B) = D$ or $\delta_2(D) = \delta_2(B)$ and $B \le_2 D$. By induction for a) and b) we have 
\[ ldim(ker(\gam B D)) \le \delta_2(B) - \delta_2(D) \]
\[ ldim(ker(\gam D A)) < \delta_2(D) - \delta_2(A). \]
Then \[ ldim(ker(\gam B A)) < \delta_2(B) - \delta_2(A). \]

{\bf Case 3)} For all $D$ with $D = \langle D_1 \rangle $ such that $B \subseteq D \subseteq A$, and $B \neq D \neq A$,  $\delta_2(D) < \delta_2(A)$ or $ \delta_2(B) < \delta_2(D)$.

{\bf Case 3.1} For some $D$ as above $ \delta_2(D) < \delta_2(A)$. This is only possible in a). We choose such a $D_1$ maximal. Then $B \le_2 D$ and
$D \le_2 A$. The assertion follows by induction.

{\bf Case 3.2)} For all $D$ with $D = \langle D_1 \rangle $ such that $B \subseteq D \subseteq A$, and $B \neq D \neq A$, we have $\delta_2(D) \ge \delta_2(A)$ and  $ \delta_2(B) < \delta_2(D)$.\\
If for some $D^0$ with  $D^0 = \langle D_1^0 \rangle $ such that $B \subseteq D^0 \subseteq A$, and $B \neq D^0 \neq A$ we have $\delta_2(D^0) = \delta_2(A)$, then we are in a) and $D^0 \le_2 A$ . We apply
induction. \\
Finally we can assume that for all $D$ with $D = \langle D_1 \rangle $ such that $B \subseteq D \subseteq A$, and $B \neq D \neq A$ we have  $\delta_2(A) < \delta_2(D)$ and  $ \delta_2(B) < \delta_2(D)$. 
Then $\delta_2(A) \le \delta_2(B)$, since otherwise $\delta_2(Bd) \le \delta_2(A)$ for every $d \in A_1 \setminus B_1$.
Note that $\delta_2(Bd) \le \delta_2(B) + 1$.\\

\noindent Now we choose $B_1 \subseteq E_1$ and $a \in A_1 \setminus E_1$ such that $\langle E_1 a \rangle = A$. Hence by assumption $\delta_2(A) < \delta_2(E)$. We will show that
$ldim(ker(\gam B A)) \le  max\{0,\delta_2(B) - \delta_2(A)-1\}$. Then a) and b) follow. Note that in a) 
$\delta_2(B) = \delta_2(A)$. \\
As in Lemma \ref{3.4}
we assume that
$C_1 \subseteq E_1 $,  $c_1, \ldots, c_n$ is a vector space  basis of $C_1$, such that $[c_i,a] + \psi_i$ with $\psi_i \in F(E_1)_2$ is a basis of
$I$ over $F(E_1)_2^*$, where $A = F(A_1)^*/I $. Now we work in $F(A_1)$. 
As above we use the same notations for subsets and elements of $F(X)$ as for their $\tau$-images in $F(X)^*$.
Define  $J = \langle \langle I \rangle \rangle^{F(A_1)}$ as above.
For every $D = \langle D_1 \rangle$ with $D_1 \subseteq A_1$ we use 
$N_2(D)$ for  $I_2 \cap \langle D_1
\rangle_2$. W.l.o.g. we assume, that there is some $m$, such that $c_1, \ldots , c_m \in B_1$ and 
$c_{m+1}, \ldots , c_n$ are linearly independent over $B_1$. Furthermore let $c_{m+1}, \ldots , c_n, e_1,
\ldots, e_r$be a basis of $E_1$ over $B_1$. For $B_1$ we choose a basis that extends 
$c_1, \ldots, c_m$.\\ 
First we assume that $0 < r$. We order our basis of $E_1$, as given above, in such a way that $x < e_r$ for 
all $x \not= e_r$. The set of all $[x,y]$ with $y < x$ is a vector space basis of $\langle E_1 \rangle_2$. Hence 
all monomials $[x,e_r]$, where $x \not= e_r$, are  in this basis. 
We consider the occurence of monomials $[x,e_r]$  in the $\psi_i$. 
We show that they do not occure and apply induction. We consider the $[c_i,a] + \psi_i$.
After linear transformations we can assume w.l.o.g. that there are 
$1 \le l_b \le m+1$, $m +1 \le l_a \le n + 1$, and  $\chi_i = [x_i,e_r]$ 
with $x_i$ from the vector basis of $E_1$ for $l_b \le i \le m$ and $l_a \le i \le n $ that are linearly independent such that :
\begin{description}
\item In $\psi_i $ for $1 \le i < l_b $  no $\chi_j$  occures.
\item  For $l_b \le i \le m$ $\chi_i $  occures  in $\psi_i$ and not in $\psi_j$ for $j \not= i$..
\item In $\psi_i$ for $m+1 \le i < l_a$  no $\chi_j$ occures.
\item $\chi_j$ for $l_a \le j \le n$ occures  in $\psi_j$ and does not occure in $\psi_i$ for $m+1 \le i \le n$ with 
$i \neq j$.
\end{description} 

\noindent As in Lemma \ref{3.4} we consider the preimage $\mu$ in $F(A_1)$ of an element in $ker(\gam E A)$.
As in the proof of that lemma we get
\[ \mu = (\sum_{1 \le i \le n}( \sum_{1 \le j \le n, i \neq j} r_{i,j}^{\mu} [c_j,  \psi_i]))  .\]
If $r_{i,j}^{\mu} \neq 0$ for some $l_b \le i \le m$, then $\chi_i$ cannot be canceled. Otherwise $l_b = m+1$.
If in this case $r_{i,j} \neq 0$ for some $l_a \le i \le n$, then $\chi_i$ cannot be canceled. Hence if $\mu $ represents an element of $ker(\gam B A)$, then it is impossible that $e_r$ occures in some $\psi_i$ with
$r_{i,j}^{\mu} \neq 0$ for some $j$. Let $E(-)$ be generated by $B_1$ and $c_{m+1}, \ldots , c_n, e_1, \ldots ,e_{r-1}$ and $A(-) = \langle E(-)_1 a \rangle$. We have shown, that $ker(\gam B A) = ker(\gam B {A(-)})$. By the
final assumption of {\bf Case 3.2)}
$B \le_2 A(-)$ and by induction $ker(\gam B {A(-)}) = 0$. The assertion is shown.

\noindent It remains the case where $E_1 = B_1, c_{m+1}, \ldots, c_n$. By Lemma \ref{3.4} 
$ldim(ker(\gam E A)) \le k < n-1$, where $k = ldim(N_2(C))$.\\ 
Then $0 < m$, since otherwise no $\mu$
as above can represent an element of $ker(\gam B A)$. First we assume that $2 \le m$.\\ 
Let $k_b$ be $ldim(N_2(c_1, \ldots, c_m))$ and 
$k_a = k - k_b$. Since $B \le_2 E$ $\gam B E$ is an embedding by induction. Hence in $\F(E)$ we have 
$ker(\gam B A) \subseteq (ker\gam E A)$. 
Since $E_1$ is generated by $c_{m+1}, \ldots , c_n$ over $B_1$, we have 
\[ ldim(N_2(E)) - ldim(N_2(B)) \ge \] 
\[ldim(N_2(\langle c_1, \ldots , c_n \rangle)) - ldim(N_2(\langle c_1, \ldots ,c_m \rangle))
= k - k_b = k_a .\]
We will use:\\

\[ldim(ker(\gam B A)) \le ldim(ker(\gam E A)) \le k < n-1,\]
\[ldim(N_2(E/B)) - k_a \ge 0 \] 
and since  $2 \le m$,  
\[ k_b - m +1 < 0.\]
It follows \\

\begin{tabular}{lcl}
$\delta_2(A)$ & = & $\delta_2(B) + ldim(A_1/B_1) -ldim(N_2(A/B))$ \\
& = & $\delta_2(B) + (n-m) + 1 - ldim(N_2(E/B)) - n$\\
& = & $\delta_2(B) + 1 - ldim(N_2(E/B)) -m$ \\
& = & $\delta_2(B) - k - (ldim(N_2(E/B)) - k_a) +k_b -m + 1$\\
& $<$ & $ \delta_2(B) - k $ \\
& $\le$& $ \delta_2(B) - ldim(ker(\gam B A))$
\end{tabular}\\

\noindent It remains the case $m = 1$. We use the proof of Lemma \ref{3.4}
and show that \\ 
$ker(\gamma_B^{A} ) = \langle 0 \rangle $. \\
The preimage of a 
nontrivial element in $ker(\gamma_B^{A}) \subseteq \F(B) \subseteq \F(E)$ in $F(\langle A_1 \rangle)$ has the form

\[ \mu =  [ c_1, \sum_{2 \le i \le n} r_{i,1}^{\mu} \psi_i] + ( \sum_{2 \le j \le n}  [c_j, \sum_{1 \le i \le n,i \neq j} r_{i,j}^{\mu} \psi_i])  .\]

\noindent Then the image of $\mu $ in $\F(B)$ has the form $[c_1, d]$, where $d \in \F(B)_2$.
This is impossible since $\F(A)$ is $\K_3$.
Therefore it has no homogeneous zero-divisors.
\end{proof}

\begin{cor}\label{3.5}
Let $D$ be a 3-nilpotent graded Lie algebra  in $ \K_3$.
\begin{enumerate}
\item If $B \subseteq D $ ,  $A = CSS_2^{D}(B)$ , and $B \not= A$, then $\delta(A) < \delta(B)$.
\item For graded Lie algebras $M$ in $\K_3$ we can define $B \le M$, if $\delta(B) \le \delta(A)$ for all $B \subseteq A \subseteq M$.

\end{enumerate}
\end{cor}

\begin{proof} 
ad(1)
We can assume $D = A$.
Note that $o-dim_3(D) \le o-dim(B)$. 
In $\F(D^*)$ we have $\ga B D (N_3(B)) \subseteq N_3(D)$ and 
\[ ldim(N_3(B)) \le ldim(ker(\ga B D)) + ldim(\ga B D (N_3(B))) \le 
 ldim(ker(\ga B D)) + ldim(N_3(D)). \]
By Theorem \ref{key3} $ldim(ker(\ga B D)) < \delta_2(B) - \delta_2(D)$. Hence
\[ \delta(D) = \delta_2(D) - ldim(N_3(D) + o-dim_3(D) <  \]
\[\delta_2(B) - ldim(ker(\ga B D)) - ldim(N_3(D) + 
o-dim_3(D) \le \]
\[ \delta_2(B) - ldim(N_3(B)) + o-dim_3(B) = \delta(B) .\] 
(2) follows from (1).
\end{proof}

\begin{lemma}\label{3.2.c}
Assume $A \le^n M \in \K_3$ and $A \subseteq C \subseteq D$ with $o-dim(D^*/A^*) \le n$. Then
$\delta(A) \le \delta(C)$.
\end{lemma}

\begin{proof} 
We can assume w.l.o.g.
that $D = \langle D_1 D_2 A_3\rangle$. \\
By Corollary \ref{3.5} $\delta(CSS_2^D(C)) \le \delta(C)$. Hence w.l.o.g. $C  \le_2 D$. \\ Then 
$\F(C^*)$ can be considered as a subspace of $\F(D^*)$. We have $C^- = \F(C^*)/N_C$
and $D^- = \F(D^*)/N_D$. Then $ldim(N_C) \le ldim(N_D)$.\\ 
If $o-dim(C^*/A^*) \le n$, then we are done. Otherwise we show $\delta(A) \le \delta(C)$.\\
Let 
$X = X_2^C(A) X_2^C(C)$ be a vector space basis of $C_2$ over $\langle C_1 \rangle_2$, where 
$X_2^C(A)$ is a vector space basis of $ A_2$ over $\langle C_1 \rangle_2$
and $X_2^C(C)$ is a vector space basis of $C_2$ over $\langle C_1 A_2 \rangle_2$. Let $X_3^C$ be a vector space basis of $C_3$ over $C^-_3$. Now
\[ \delta(C) = \delta_2(\langle C_1 \rangle) + \mid X_2^C(A) \mid + \mid X_2^C(C) \mid +
\mid X_3^C \mid - ldim(N_C) .\]
Furthermore
\[  \delta(D) \le \delta_2(\langle C_1 \rangle) + ldim(D_1/C_1) + \mid X_2^D(A) \mid  
+ \mid X_2^D(D) \mid
+ \mid X_3^D \mid - ldim(N_D),\]
where $X_2^D(A)$ is a vector space basis of $A_2$ over $\langle D_1 \rangle_2$,
$X_2^D(D)$ is a vector space basis of $D_2$ over $\langle D_1 A_2 \rangle_2$ and $X_3^D$
is a vector space basis of $D_3$ over $D^-_3$. \\
Then $\mid X_3^C \mid \ge \mid X^D_3 \mid$
and $\mid X_2^D(A) \mid \le \mid X_2^C(A)\mid$. By assumption
\[  ldim(C_1/A_1) +ldim(D_1/C_1)  + \mid X_2^D(D) \mid \le n < ldim(C_1/A_1) 
+ \mid X_2^C(C) \mid .\]
Hence $\delta(D) \le  \delta(C)$.

\end{proof}

\begin{cor}\label{3.2.d}
Assume $M \in \K_3$. $A \le^n M$ if and only if $\delta(A) \le \delta(C)$ for all $A \subseteq C 
\subseteq M$, that are contained in some $D \subseteq M$ with $o-dim(D^*/A^*) \le n$.
\end{cor}


\section{submodularity and amalgamation for $\K_3$}
\noindent From now on we concentrat on 3-nilpotent  graded Lie algebras.

\begin{lemma}\label{6.1} 
We work in a 3-nilpotent graded Lie algebra $M \in \K_3$.
Assume $C \cap A = B$ and $\langle C A \rangle^* = C^* \otimes_{B^*} A^*$, $\alpha$   is the the canonical homomorphism of $\F(B^*)$ into $\F(A^*)$, and $\gamma$ the canonical homomorphism of $\F(B^*)$ into $\F(C^*)$. We identify $\F(B^*)/ \langle ker(\alpha), ker(\gamma) \rangle$ with the isomorphic substructures in $\F(A^*)/\alpha(ker(\gamma))$ and in $\F(C^*)/ \gamma(ker(\alpha))$. Then
\[\F(\langle C A \rangle^*) \cong \F(C^*)/ \gamma(ker(\alpha)) \otimes_{\F(B^*)/ \langle ker(\alpha), ker(\gamma) \rangle} \F(A^*)/ \alpha(ker(\gamma)) .\]
The isomorphism is given by the fact, that we can identify the $^*$-images of both sides.
\end{lemma}

\begin{proof}
If we apply $^*$ to the right side, then the result is isomorphic to $\langle C A \rangle^*$. Hence there is a homomorphism $h$ from the left side onto the right side.\\
\noindent Conversely there are homomorphisms $f$ of $\F(C^*)/\gamma(ker(\alpha))$ and $g$ of 
$\F(A^*)/ \alpha(ker(\gamma))$ into  $\F(\langle A C\rangle^*)$. $f$ and $g$ coincide on the intersection isomorphic to \\
\noindent$\F(B^*)/ \langle ker(\alpha), ker(\gamma) \rangle$. Hence there is a homomorphism $k$ of the right side onto the left side. Since  the $^*$-images $h^* = (k^*)^{-1}$ are isomorphisms we get the assertion. Note that 
we have $\F(E^*) = \langle \F(E^*)_1 \F(E^*)_2 \rangle^{\F(E^*)}$.
\end{proof}

\begin{cor}\label{6.1.a}
We are in the setting of Lemma \ref{6.1}. \\
Let $X_1^{a}$ be an ordered vector space basis of $A_1 $ over $B_1$ , \\
$X_1^{c}$ be an ordered  vector space basis of $C_1 $ over $B_1$, \\ 
$X_2^{a}$ be a vector space basis of $A_2 $ over $B_2$, and \\
$X_2^{c}$ be a vector space basis of $C_2 $ over $B_2$. Then
\begin{enumerate}
\item $\{ [x,y] : x \in X_1^c , y \in X_1^{a} \}$ is a vector space basis of $\F(\langle C A \rangle^*)_2$ over
$C_2 + A_2$,
\item \[\{ [w,z] : w \in X_2^c , z \in X_1^{a}  \:or \:
  w \in X_2^{a} , z \in X_1^c \}\]
\[\{ [[x,y],z] : x = z \in X_1^c, y \in X_1^{a} \: or \] \[ x = z \in X_1^{a}, y \in X_1^c  \: or \:
x < z \in X_1^c , y \in X_1^{a} \: or  \: x < z \in X_1^{a} , y \in X_1^c \}\] 
is a vectorspace basis of $\F(\langle  C  A \rangle^*)_3$ over $\langle C_1 C_2 \rangle_3^{\F(\langle C A \rangle^*)} + \langle A_1 A_2 \rangle_3^{\F(\langle C A \rangle^*)}$.
\end{enumerate}
\end{cor}

\noindent The Corollary is a consequence of Lemma \ref{6.1}, Corollary \ref{c=2,z}, and 
Corollary \ref{c=3,a}.

\begin{lemma}\label{subdelta3}
Let $M$ be a 3-nilpotent graded Lie algebra in $ \K_3$ and $A$ and $C$ are  finite substructures of $M$ with 
$A \le_2 M$ and $C \le_2 M$.
Then 
\[ \delta(\langle A C\rangle) \le \delta(A) + \delta(C) - \delta(A \cap C) .\]
\end{lemma}

\begin{proof}
Let $B$ be $A \cap C$.
By Lemma \ref{subdelta2} we have

\begin{description}
\item[1)]$ \delta_2(\langle A C \rangle \le \delta_2(A) + \delta_2(C) - \delta_2(B) .$
\end{description}

\noindent We can assume w.l.o.g. that $M = \langle A C \rangle$. Let $M^- = \langle M_1 M_2 \rangle$ be isomorphic to $\F(M^*)/N$, where $N \subseteq \F(M^*)_3 $. 
Furthermore we define \\

\begin{tabular}{ll}
$A^- = \F(A^*)/N(A)$ & \\
$C^- = \F(C^*)/N(C)$ & \\
$B^- = \F(B^*)/N( B)$ & \\
\end{tabular}\\

\noindent Since $A \le_2 M$ and $C \le_2 M$,
we have $A \cap C \le_2 M$ by Lemma \ref{c=2,b}. Hence we can consider 
$\F(A^*), \F(C^*), \F(A^* \cap C^*) = \F(B^*)$ as substrucures of $\F(M^*)$  \\
by Theorem \ref{key3}.\\
Since $ \F((B^*)_i =\F(A^*)_i \cap \F(C^*)_i$ for $i = 1,2$, there is an homomorphism of $\F(B^*)$ onto $\F(A^*) \cap \F(C^*)$. Hence $\F(B^*) =\F(A^*) \cap \F(C^*)$.\\
Under these assumptions we have $N(A) = N \cap \F(A^*)_3$, $N(C) = N \cap \F(C^*)_3$, and
$ N( B) = N \cap \F( B^*)_3$.\\ Then $N(C) \cap N(A) = N(B)$. We get:

\begin{description}
\item[2)] $ldim(N(\langle A C\rangle)) \ge ldim(N(A)) + ldim(N(C)) - ldim(N( B)) + r.$
\end{description}

\noindent where $r$ is the linear dimension of all $\psi_a + \psi_b$ in  $N(\langle A C \rangle)$ over $N(A) + N(C)$ with 
$\psi_a \in \F(A^*)_3$ and $\psi_b \in \F(C^*)_3$ . \\

Now we have to compute the $o-dimensions_3$. In $M_3$ we define:

\begin{description}
\item $X_3^0(B) $ is a vector basis for $\langle C_1 C_2\rangle_3 \cap \langle A_1A_2 \rangle_3$ over $\langle B_1B_2 \rangle_3$.
\item $X_3^{A}(B)$ is a vector basis for $\langle A_1 A_2\rangle_3 \cap C_3$ over $\langle B_1 B_2\rangle_3 + \langle X_3^0(B) \rangle$.
\item $X_3^{C}(B)$ is a vector basis for $\langle C_1 C_2\rangle_3 \cap A_3$ over $\langle B_1 B_2\rangle_3 +\langle X_3^0(B) \rangle$.
\item $X_3^1(B)$ is a vector basis for $B_3 \cap \langle A_1 A_2 C_1 C_2 \rangle_3$ over 
$\langle B_1 B_2\rangle_3 + \langle X_3^0(B) X_3^{A}(B) X_3^{C}(B) \rangle$.
\item $X_3^2(B)$ is a vector basis of $B_3$ over $\langle A_1 A_2 C_1 C_2 \rangle_3$. 
\end{description}
 Note that $X_3^{A}(B) \cup X_3^{C}(B)$ is linearly independent over $\langle B_1 B_2 \rangle_2 + \langle X_3^0(B) \rangle_3$. 
Now $B_3$ has the following vector space  basis $X_3(B)$ over $\langle B_1 B_2 \rangle_3$:
\[ X_3^0(B) X_3^{A}(B) X_3^{C}(B) X_3^1(B) X_3^2(B) .\] 

Let $X_3^1(A)$ and $X_3^1(C)$ be choosen such that

\begin{description}
\item $ X_3^{C}(B) X_3^1(B)  X_3^1(A)$ is a vector basis of $A_3 \cap \langle A_1A_2 C_1 C_2 \rangle_3$ over $\langle A_1 A_2\rangle_2$.
\item $ X_3^{A}(B) X_3^1(B)  X_3^1(C)$ is a vector basis of $C_3 \cap \langle A_1 A_2C_1 C_2\rangle_2$
over $\langle C_1 C_2\rangle_2$.
\end{description} 

Choose $X_3^2(A)$ and $X_3^2(C)$ such that

\begin{description}
\item $X_3(A) = X_3^{C}(B) X_3^1(B) X_3^2(B) X_3^1(A) X_3^2(A)$ is a vector basis of $A_3$ over $\langle A_1 A_2\rangle_3$.
\item $X_3(C) = X_3^{A}(B) X_3^1(B) X_3^2(B)  X_3^1(C) X_3^2(C)$ is a vector basis of $C_3$ over $\langle C_1 C_2\rangle_3$.
\end{description} 

Then $X_3(AC) = X_3^2(B) X_3^2(A) X_3^2(C)$ is a vector basis  of $\langle AC \rangle_3$ over $\langle A_1 A_2 C_1 C_2 \rangle_3$.

It follows:

\begin{description}
\item[3)] $o-dim_3(A) + o-dim_3(C) - o-dim_3( B) =$
\[ \mid X_3^{C}(B) \mid + \mid X_3^1(B) \mid + \mid X_3^2(B)\mid + 
\mid X_3^1(A) \mid  + \mid X_3^2(A) \mid + \] 
\[ \mid X_3^{A}(B) \mid + \mid X_3^1(B) \mid + \mid X_3^2(B)\mid + 
\mid X_3^1(C) \mid + \mid X_3^2(C) \mid - \] 
\[\left( \mid X_3^{C}(B) \mid +\mid X_3^{A}(B) \mid + \mid X_3^1(B) \mid + \mid X_3^2(B)\mid + \mid X_3^0(B) \mid \right) = \] 
\[ \mid X_3^1(B) \mid + \mid X_3^2(B)\mid + 
\mid X_3^1(A) +\mid X_3^2(A) \mid + 
\mid X_3^1(C) \mid +\mid X_3^2(C) - \mid X_3^0(B) \mid = \] 
\[\mid X_3(AC) \mid + \mid X_3^1(B)  X_3^1(A) X_3^1(C) \mid  - \mid X_3^0(B) \mid =\]
\[o-dim_3(AC) + \mid X_3^1(B)   X_3^1(A) X_3^1(C)\mid - \mid X_3^0(B) \mid.  \]

\end{description}

Since $\mid X_3^0(B) \mid = r $, 1), 2), and 3) imply

\[ \delta(\langle A C\rangle) \le \delta(A) + \delta(C) - \delta( B) .\]
\end{proof}

If we repeat the proof above without the assumptions $C \le_2 M$ and $A \le_2 M$, then we obtain that
$\delta(A) + \delta(C) - \delta(B) - \delta(\langle A C \rangle)$ is the sum of the following summands:
\begin{enumerate}
\item[$\Sigma1$] $\delta_2(A) + \delta_2(C) - \delta_2(B) - \delta_2(\langle A C \rangle) $
\item [$\Sigma2$] $\mid X_3^1(B)   X_3^1(A) X_3^1(C)\mid  $
\item [$\Sigma 3$] $ ldim(N) + ldim(N(B)) - ldim(N(C)) - ldim(N(A))  - \mid X_3^0(B) \mid $.
\end{enumerate}
In this general case we have $0 \le \Sigma1$ by Lemma \ref{subdelta2}. If $\Sigma3 \ge 0$, then we have subadditivity.

We use this remark  to get the following:

\begin{lemma}\label{6.3}
Let $M$ be a 3-nilpotent graded Lie algebra  in $\K_3$,  $M = \langle C A \rangle$,
$C \cap A = B$,  and $B \le_2 A$. 
We choose $X_B$ as a maximal subset of $ B^*_2 \cap \langle C_1 A_1 \rangle_2$  linearly independent over 
$\langle C_1 \rangle_2 + \langle A_1 \rangle_2$. Then let 
$X_C $  be maximal subset of $ C^*_2 \cap \langle C_1 A_1 \rangle_2$  that is linearly independent over 
$\langle C_1 \rangle_2 + \langle A_1 \rangle_2 + \langle X_B \rangle_2$, and
$X_A $  be maximal subset of $ A^*_2 \cap \langle C_1 A_1 \rangle_2$  that is linearly independent over 
$\langle C_1 \rangle_2 + \langle A_1 \rangle_2 + \langle X_B \rangle_2$. \\
Furthermore there are $D \subseteq C^*$ 
$H \subseteq B^*$, $K \subseteq A^*$  
such that 
\begin{enumerate}
\item $B^* = \langle H X_B \rangle$, $C^* = \langle D X_B X_C \rangle$, 
$A^* = \langle K X_B X_A \rangle$, 
\item $X_B$ is linearly independent over $H_2$,\\
$X_B X_C$ is linearly independent over $D_2$,\\
$X_B X_A$ is linearly independent over $K_2$,
and
\item $\langle C^* A^* \rangle = \langle D K \rangle = D \otimes_H  K.$
\end{enumerate}
\noindent Then
\begin{enumerate}
\item $\delta( \langle C A \rangle) \le \delta(C) + \delta(A) - \delta(B)$. 
\item If $\mid X_C X_B X_A \mid > 0$,
then $\delta( \langle C A \rangle) < \delta(C) + \delta(A) - \delta(B)$. 
\item If $C \le \langle C A \rangle$, then $B \le   A $.
\item If $C^* = D$, $B^* = H$, $A^* = K$, $B \le A$, and $\langle C A \rangle = C \otimes_B A$, 
then $\delta( \langle C A \rangle) =  \delta(C) + \delta(A) - \delta(B)$ and $C \le \langle C A \rangle$.
\item If $C^* = D$ , $B^* = H$, $A^* = K$, $B \le A$, and  $\delta( \langle C A \rangle) =  \delta(C) + \delta(A) - \delta(B)$, then $\langle C A \rangle = C \otimes_B A$.
\end{enumerate}
\end{lemma}

\begin{proof}
Note that $X_B X_C X_A$ is linearly independent over $D_2 + K_2$.\\
ad (1) We use the remark below Lemma \ref{subdelta3}.
We have only to show that $0 \le \Sigma3$. 
By Lemma \ref{6.1} we get
\[(\otimes) f: \F(\langle C^* A^* \rangle) = \F(\langle D K \rangle) \cong \F(D) \otimes_{\F(H)/ker(\gamma)} \F(K)/ker(\gamma), \]
where $\gamma$ is the homomorphism of $\F(H)$ into $\F(D)$. Note that $B \le_2 A$ implies 
$\langle H X_B \rangle  \le_2 \langle K X_B \rangle$. By Lemma \ref{3.2.a} we get
$H \le_2 K$. Therefore we can consider $\F(H)$ as a substructure of $\F(K)$
and $\F(B^*)$  as a substructure of $\F(A^*)$ by Theorem \ref{key3}. 
By Theorem \ref{c=2,c} $D \le_2 \langle D  K\rangle$.  Hence we consider $\F(D)$ as a 
substructure of $\F(\langle D K \rangle)$. Furthermore 
we consider substructures of $M^*$ as subspaces  of $\F(M^*)$:\\
We can rewrite $(\otimes)$ as 
\[ \F(\langle D K \rangle) = \F( D ) \otimes_{\langle H \rangle^{ \F(\langle D K \rangle)}} 
\langle K \rangle^{ \F(\langle D K \rangle)}, \]
where $\langle H \rangle^{ \F(\langle D K \rangle} \cong \F(H)/ker(\gamma)$ and $\langle K \rangle^{ \F(\langle D K \rangle)} \cong \F(K)/ker(\gamma)$.\\
By assumption  $X_B X_C X_A$  is also linearly independent over  $D_2 + K_2$. Hence by 
Lemma \ref{3.2.b} and ($\otimes$) we get that \\
($\otimes \otimes$) 
\[ \F(\langle C^* A^* \rangle)_3 = \F(\langle D K \rangle)_3  \supseteq \langle D X_B X_C \rangle^{ \F(\langle D K \rangle)}_3 \oplus_{ \langle H X_B \rangle_3^{ \F(\langle D K \rangle)} } \langle K X_B X_A \rangle^{ \F(\langle D K \rangle)}_3, \]
where $\langle H X_B \rangle^{ \F(\langle D K \rangle)} = 
\langle H  \rangle^{ \F(\langle D K \rangle)} \otimes \langle  X_B \rangle^{ \F(\langle D K \rangle)} =
\langle B^* \rangle^{ \F(\langle D K \rangle)} 
\cong \F(B^*)/ker(\gamma)$, \\
$\langle D X_B X_C \rangle^{ \F(\langle D K \rangle)} = 
\langle D  \rangle^{ \F(\langle D K \rangle)} \otimes \langle  X_B  X_C \rangle^{ \F(\langle D K \rangle)}  =
\langle C^* \rangle^{ \F(\langle D K \rangle)} 
\cong \F(C^*)$, \\
$\langle K X_B X_A\rangle^{ \F(\langle D K \rangle)} = 
\langle K  \rangle^{ \F(\langle D K \rangle)} \otimes \langle  X_B  X_A \rangle^{ \F(\langle D K \rangle)} =
\langle A^* \rangle^{ \F(\langle D K \rangle)} 
\cong \F(A^*)/ker(\gamma)$.\\
As in the proof of Lemma \ref{subdelta3} we define $\langle C A \rangle^- = \F(\langle C^* A^* \rangle) / N$, 
$C^- = \F(C^*)/N(C)$, $B^- = \F(B^*)/N(B)$, and  $A^- = \F(A^*)/N(A)$.       By $(\otimes \otimes)$ 
\[ N \supseteq  M(C) \oplus_{M(B)} M(A)    \oplus U,\]
where $M(B)$ is the image of $N(B)/ker(\gamma)$ in 
$\langle H X_B \rangle^{ \F(\langle D K \rangle)}_3 $ ,\\
$M(C)$ is the image of $N(C)$ in  $\langle D X_B X_C \rangle^{ \F(\langle D K \rangle)}_3$, \\
$M(A)$ is the image of $N(A)/ker(\gamma)$ in $\langle K X_B X_A\rangle^{ \F(\langle D K \rangle)}_3 $,
and \\
$U$ is generated by  a maximal number $r = \mid X_3^0(B) \mid$ of   $\psi_C + \psi_A \in N$, where 
the $\psi_C \in \langle C_1 C_2 X_B X_C\rangle_3$ are linearly independent over $M(C)_3$,  and the  
$\psi_A \in \langle K_1 K_2 X_B X_A\rangle_3$ are linearly independent over $M(A)_3$. \\
Note that $ldim(N(B) \cap ker(\gamma)) = ldim(N(A) \cap ker(\gamma)) = s$. 
Since s will be canceled we get 2) as above:
\[ ldim(N) \ge ldim(N(C)) + ldim(N(A)) - ldim(N(B)) + \mid X_3^0(B) \mid .\] Hence $0 \le \Sigma3$. \\
ad (2) If $\mid X_C X_B X_A \mid > 0$, then $\Sigma 1 > 0$\\

\noindent ad (3) 
We have $B \le_2 A$  by assumption. Then (1) 
and $C \le \langle C A \rangle$ imply $0 \le \delta(A/C) \le \delta(A/B)$.\\
To prove the  assertion choose $B \subseteq E \subseteq A$. 
We have to show that $\langle C E \rangle$ satisfies the assumptions of the Lemma. We have 
$B \le_2 E$. We choose $X_E$ as a  subset of $\langle C_1 E_1 \rangle \cap E_2$ maximal linearly independent over $\langle C_1 \rangle_2 + \langle E_1 \rangle_2 + \langle X_B \rangle$.
Note that we work in $M^*$. Let
$K(E)$ be $K \cap E^*$. Then we have $E_1 = K(E)_1$ and $\langle D K(E) \rangle = D \otimes_H K(E)$ by { \bf Mon}. \\
Choose $Y \subseteq E_2$ maximal linearly independent over $\langle K(E)_2 X_B X_E \rangle$. 
and define $K_E = \langle K(E) Y \rangle$. Then $\langle K_E X_B X_E \rangle = E^*$ and also
\[ \langle C^* E^* \rangle = \langle D K_E \rangle = D \otimes_H K_E. \]
Hence the assumptions of the Lemma are fulfilled and we get by (1) $\delta(\langle C E \rangle) \le 
\delta(C) + \delta(E) - \delta(B)$. Since $C \le \langle C E \rangle$ we get the assertion for $E$.\\

\noindent ad (4) 
As in (1) we use the  remark below the proof of Lemma \ref{subdelta3}. We have $\langle C^* A^* \rangle = C^* \otimes_{B^*} A^*$. By Corollary \ref{c=2,y} this implies $\delta_2(\langle C A \rangle) = \delta_2(C) + \delta_2(A) - \delta_2(B)$. This means $\Sigma1 = 0$. By Corollary \ref{c=3,a} we get $\mid X_3^1(B) X_3^1(C) X_3^1(A) \mid = 0$. 
Hence $\Sigma2 = 0$. \\
As in the proof of (1) we have 
\[ (\otimes) \F(\langle C^* A^* \rangle)  =
\langle C^* \rangle^{\F(\langle C^* A^* \rangle) } \otimes_{\langle B^* \rangle^{\F(\langle C^* A^* \rangle) }} \langle A^* \rangle^{\F(\langle C^* A^* \rangle) },\]
where $\langle C^* \rangle^{\F(\langle C^* A^* \rangle) } \cong \F(C^*)$, since $C^* \le_2
C^* \otimes_{B^*} A^*$,\\
$\langle B^* \rangle^{\F(\langle C^* A^* \rangle) } \cong \F(B^*)/ker(\gamma)$, and \\
$\langle A^* \rangle^{\F(\langle C^* A^* \rangle) } \cong \F(A^*)/ker(\gamma)$, 
where $\gamma$ is the homomorphism of $\F(B^*)$ into $\F(C^*)$ and w.l.o.g. 
$\F(B^*) \subseteq \F(A^*)$ by $B^* \le_2 A^*$. \\
Since $\langle C A \rangle = C \otimes_B A$ we get similarly as in (1) 
\[ N = M(C) \oplus_{M(B)} M(A) \oplus U, \]
where 
$M(C)$, $M(B)$, and $M(A)$ are the preimages of $N(C)$, $N(B)$, and $N(A)$ respectively in 
$\F(\langle C^* A^* \rangle) $.
$U$ is generated by  $r = \mid X_3^0(B) \mid$ many  $\psi_C + \psi_A \in N$, where the 
$\psi_C \in \langle C_1 C_2\rangle_3$ are linearly independent over $M(C)$, and 
the $\psi_A \in \langle A_1 A_2\rangle_3$ are linearly independent over  $M(A)$.
Then $ldim(N(C)) = ldim(M(C))$, $ldim(N(B)) = ldim(M(B)) + s$, and $ldim(N(A)) = ldim(M(A)) + s$, 
with $s = ldim(ker(\gamma) \cap N(B)) = ldim(ker(\gamma) \cap N(A))$.\\
Hence \[ ldim(N) = ldim(N(C)) + ldim(N(A)) - ldim(N(B)) + \mid X_3^0(B) \mid .\]
With other words $\Sigma3 = 0$. We obtain 
$\delta(\langle C A\rangle) = \delta(C) + \delta(A) - \delta(B)$.\\
We use  this to show $C \le \langle C A \rangle$.
For $C \subseteq E \subseteq C \otimes_B A$ we have to show $\delta(C) \le \delta(E)$. 
By {\bf Mon} $\langle C (E \cap A) \rangle = C \otimes_B (E \cap A)$.
Then $\delta(C \otimes_B (A\cap E)) \le \delta(E)$ by Corollary \ref{c=3,a}. Hence
it is suficient to consider the case  
$E = C \otimes_B (E \cap A)$. Furthermore $B \le A \cap E$. Hence we can assume that w.l.o.g. 
$E = \langle C A \rangle$. \\
It follows $\delta(C) \le \delta(\langle C A \rangle)$,
since $B \le A$ by assumption. \\

\noindent ad (5) We use again the remark below Lemma \ref{subdelta3}. By$ \langle C^* A^* \rangle = C^* \otimes_{B^*} A^*$ 
we get $\Sigma 1 = 0$.  From the proof of (1) we know, that $\Sigma 3 \ge 0$. Hence  by assumption $\Sigma 2 = 0$ and $\Sigma 3 = 0$.
As above by Lemma \ref{6.1} we get an isomorphism
\[(\otimes) f: \F(\langle C^* A^* \rangle)  \cong \F(C^*) \otimes_{\F(B^*)/ker(\gamma)} \F(A^*)/ker(\gamma), \]
where $\gamma$ is the homomorphism of $\F(B^*)$ into $\F(C^*)$. Note that $B \le_2 A$ implies $\F(B^*) \subseteq \F(A^*)$. By Theorem \ref{c=2,c} $C \le_2 \langle A C \rangle$.  Hence $\F(C^*) \subseteq
\F(\langle A C \rangle^*)$.\\
As in the proof of Lemma \ref{subdelta3} we define $\langle C A \rangle^- = \F(\langle C^* A^* \rangle) / N$, 
$C^- = \F(C^*)/N(C)$, $B^- = \F(B^*)/N(B)$, and  $A^- = \F(A^*)/N(A)$.       By $(\otimes)$ 
\[ f(N) \supseteq  N(C) \oplus_{N(B) / ker(\gamma)} (N(A) / ker(\gamma))  \oplus U,\]
where $U$ is generated by  $r = \mid X_3^0(B) \mid$ many  $\psi_C + \psi_A$, where 
$\psi_C \in \langle C_1 C_2\rangle_3$ and $\psi_A \in \langle A_1 A_2\rangle_3$, that are linearly independent over $N(C) + N(A)$.
Note that $ldim(N(B) \cap ker(\gamma)) = ldim(N(A) \cap ker(\gamma)) = s$. 
Since s will be canceled we get :
\[ ldim(N) = ldim(N(C) + ldim(N(A) - ldim(N(B) + \mid X_3^0(B) \mid ,\] 
since $0 = \Sigma3$. \\ Then
\begin{itemize} 
\item[(=)]  $f(N) = N(C) \oplus_{N(B) / ker(\gamma)} (N(A) / ker(\gamma))  \oplus U.$
\end{itemize}
Assume we have homomorphisms $j$ of $C$ into $E$ and $g$ of $A$ into $E$, that coinside on $B$.\\
There are homomorphisms $j^+$ of $\F(C^*)$ into $E$ and $g^+$ of $\F(A^*)/ker(\gamma)$ into $E$, that coinside on
$\F(B^*)/ker(\gamma)$. \\ 
By $(\otimes)$ there is an homomorphism $h^+$ of $\F(\langle C^* A^* \rangle)$ into $E$. \\
By $(=)$ and $\Sigma 2 = 0$ we get an homomorphism $h^-$ of $\langle C^- A^- \rangle$ into $E$.\\
$h^-$ can be extended to  a homomorphism $h$ of $\langle C A \rangle$ into $E$ . Hence $\langle C A \rangle = C \otimes_B A$.
\end{proof}

\begin{lemma}\label{6.5}
Assume $D = C \otimes_B A$ and $A = \langle B \bar{a} \rangle $, where $\bar{a}$ is a 
generating o-system of $A$ over $B$. Then we get:
\begin{enumerate}
\item $D^* = C^* \otimes_{B^*} A^*$
\item If $E \subseteq B$, $D = C \otimes_E \langle E \bar{a} \rangle$, and $B \le A$, then $C \le  D$ and 
$E \le \langle E \bar{a} \rangle$.
\end{enumerate}
\end{lemma}

\begin{proof}
ad(1) This is a consequence of the Lemmas \ref{c=3,a} and \ref{c=2,y}.\\
ad(2) By Lemma \ref{3.5} we have $B^* \le A^*$. By (1)  
$D^* = C^* \otimes_{E^*} \langle  E \bar{a} \rangle^*$.
By Lemma \ref{6.3}(4) we have $C \le D$.
By Lemma \ref{c=2,b} $B^* \le A^*$ implies $E \le_2 \langle E \bar{a} \rangle$. Then by Lemma \ref{6.3} (3)
we get $E \le \langle E \bar{a} \rangle$.
\end{proof}

\noindent As in Definition \ref{5.6} we define

\begin{definition} For $A, C, N \subseteq M \in \K_3$:
\begin{enumerate}
\item $\delta(A/C) = \delta\langle A C \rangle - \delta(C)$.
\item $\delta(A/N) = min\{\delta(\langle A C \rangle) - \delta(C) : C \subseteq N \}$.
\end{enumerate}
\end{definition}

\noindent From the end of the proof of Lemma \ref{subdelta3}  and Corollary \ref{corsubdelta2} we get

\begin{cor}\label{corsubdelta3}
Let $M$ be a 3-nilpotent  graded Lie algebra in $\K_3$ and $A$ and $C$ are 2-strong substructures. Then
$\delta(A/C) = \delta(A/A\cap C)$ if and only if $\delta_2(A/C) = \delta_2(A/A\cap C)$, $A_3$ and $C_3$
do not contain any elements from $\langle A_1 A_2 C_1 C_2 \rangle _3 \setminus \langle A_1 A_2 \rangle_3 + \langle C_1 C_2 \rangle_3$ and if $N$ is choosen as in the proof of \ref{subdelta3}, then  $N \subseteq \F(A^*)_3 + \F(C^*)_3$. 
\end{cor}

\begin{cor}\label{char}
Let $M$ be a 3-nilpotent graded Lie algebra in $\K_3$ and $A$ and $C$ are 2-strong substructures. Then
$\delta(A/C) = \delta(A/A\cap C)$ if and only if $\langle AC \rangle = A \otimes_{A \cap C} C$. 
\end{cor}

\begin{proof}
Let $B$ be $A \cap C$.
Assume $\langle AC \rangle = A \otimes_{A \cap C} C$. Then $\langle A^* C^* \rangle = A^* \otimes_{A^* \cap C^*} C^*$. By Corollary \ref{c=2,y} we have $\delta_2(A/C) = \delta_2(A/B)$. By Corollary  \ref{c=3,a} 
$A_3$ and $C_3$
do not contain any elements from $\langle A_1 A_2 C_1 C_2 \rangle _3 \setminus \langle A_1 A_2 \rangle_3 + \langle C_1 C_2 \rangle_3$. Let $N$ be chosen as in the proof of \ref{subdelta3} .Then $N \subseteq \F(A^*)_3 + \F(C^*)_3$. Hence
$\delta(A/C) = \delta(A/A\cap C)$ by Corollary \ref{corsubdelta3}. \\

\noindent For the other direction we use the other direction of Corollary \ref{corsubdelta3}. Then $\delta_2(A/C) = \delta_2(A/B)$ 
implies  $\langle A^* C^* \rangle = A^* \otimes_{A^* \cap C^*} C^*$ by Corollary \ref{c=2,y}. 
By assumption $A,  C$ are 2-strong in $\langle A C \rangle$. 
Then $B$ is also 2-strong in $A$, $C$, and $\langle A C \rangle$. 
Hence we can consider $\F(A^*), \F(B^*), \F(C^*)$ as subspaces of $\F(\langle A^* C^* \rangle)$. $\F(A^* \cap C^*) \subseteq \F(A^*) \cap \F(C^*)$.
Since $(\F(A^*) \cap \F(C^*))^* = A^* \cap C^*$ , there is a homomorphism of $\F(A^* \cap C^*)$ onto
$\F(A^*) \cap \F(C^*)$. Hence it is an isomorhism. We get  equality and 
\[ \F(A^* C^*) = \F(A^* \otimes_{B^*} C^*) = \F(A^*) \otimes_{\F(B^*)} \F(C^*) \]
using the definiton of $\F$ and of the free amalgam. We choose $N \subseteq \F(\langle A^* C^* \rangle)_3 $ such that $\langle A_1 A_2 C_1 C_2 \rangle =
\F(\langle A^* C^* \rangle)/N$. 

\noindent Let $f$ and $g$ be homomorphisms of $A$ and $C$ respectively into $E$, such that $f$ and $g$ coincide on $B$. 
Then there are homomrphisms $f^+$ of $\F(A^*)$ into $E$ and $g^+$ of $\F(C^*)$ into $E$ that coinside on $\F(B^*)$. Since we have a free amalgam we get 
$h^+$ of $\F(A^*) \otimes_{\F(B^*)} \F(C^*) $ into $E$. 
By Corollary \ref{corsubdelta3} $N \subseteq \F(A^*)_3 + \F(C^*)_3$ and $A_3$ and $C_3$
do not contain any elements from $\langle A_1 A_2 C_1 C_2 \rangle _3 \setminus \langle A_1 A_2 \rangle_3 + \langle C_1 C_2 \rangle_3$. 
Hence
we get a homomorphism $h$ from $\langle A_1 A_2  C_1 C_2 \rangle$ into $E$ that extends $f$ and $g$.
Then  $h$ can be extended to $\langle A C \rangle$. 
Hence  $\langle AC \rangle = A \otimes_{A \cap C} C$.
\end{proof}

\noindent By Lemma \ref{rulesstrong}, Lemma \ref{c=2,b}, Corollary \ref{3.5} and Lemma \ref{subdelta3} above we get (1), (2) and (3) in the next Corollary.

\begin{cor}\label{rulesstrong3} 
Let $M$ be  a 3-nilpotent graded Lie algebra in $\K_3$.
Then for all  substructures
$A, C, E$ of $M$, where $A$ and $C$ are finite the following is true:
\begin{enumerate}
\item If $C \le M$ and $E \le_2 M$, then $E \cap C \le E$.
\item If $A \le C \le M$, then $A \le M$.
\item If $A , C \le M$, then $A \cap C \le M$.
\item If $A \le C \le^k M$, then $A \le^k M$.
\item If $B \le^k M$, $C \le^k M$, $o-dim(B^*)<k$, and $o-dim(C^*) <k$, then $B \cap C \le^k M$.
\end{enumerate}
\end{cor}

\begin{proof}
To prove (4), consider $ A \subseteq E \subseteq M$ with $o-dim(E^*/A^*) \le k$. Then
$o-dim(\langle C E \rangle^*/C^*) \le k$. By Corollary \ref{3.2.d} we get $C \le \langle C E \rangle$.
By (2) we have $A \le \langle C E \rangle$. Hence $\delta(A) \le \delta(E)$.\\
ad(5): By Corollary \ref{3.2.d} and the assumptions we get, that $B  \le \langle B C \rangle$
and $C  \le \langle B C \rangle$. Hence by (3) $B \cap C  \le \langle B C \rangle$ and $B \cap C \le B$.
Now assume $B \cap C \subseteq E \subseteq M$ and $o-dim(E^*/B^* \cap C^*) < k$. Then
$o-dim(\langle E^* B^* \rangle/B^*) < k$. By Corollary \ref{3.2.d} $B \le \langle E B \rangle$. By 
$B \cap C \le B$ we get $B \cap C \le \langle EB \rangle$. Hence  $\delta(B \cap C) \le \delta(E)$,
as desired.
\end{proof}

\noindent The Corollary allows the following definition.

\begin{definition}
If $A$ is a finite substructure of $M$ in $\K_3$, then there exists a unique minimal $C$ with
$A \subseteq C \le M$. We define $C = CSS(A)$ - the selfsufficient closure. \\
Let  $A^- $ be $\langle A_1 A_2 \rangle$
\end{definition}

\begin{lemma}\label{css}
Let $A$ be a finite substructure $M \in \K_3$.
\begin{enumerate}
\item $CSS(A^-) = CSS(A^-)^-$ and $CSS(A^-) \subseteq CSS(A)$.
\item Either $CSS(A) = \langle CSS(A^-), A \rangle$ or there is some $C$ such that 
$CSS(A^-) \le C \le M$, $C = C^-$, $CSS(A) = \langle C A \rangle$, and 
\[ 0 \le \delta(C) - \delta(CSS(A^-)) < o-dim_3(A/CSS(A^-)) - o-dim_3(A/C). \]
\item $CSS_2(A) \subseteq CSS(A)$.
\item If $B \subseteq A$, then $CSS(B) \subseteq CSS(A)$.
\end{enumerate}
\end{lemma}

\begin{proof}
ad (1) For every substructure $C \subseteq M$ we have $\delta(C) = \delta(C^-) + o-dim_3(C)$. \\
By Lemma \ref{subdelta3}
\[  0 \le \delta(\langle CSS(A) CSS(A^-) \rangle) - \delta(CSS(A)) \le \delta(CSS(A^-)) - 
\delta(CSS(A^-)  \cap CSS(A)).   \]
If $CSS(A^-) \not\subseteq CSS(A)$, then $\delta(CSS(A^-)) - 
\delta(CSS(A^-)  \cap CSS(A)) < 0$, 
and therefore   $ \delta(\langle CSS(A) CSS(A^-) \rangle) - \delta(CSS(A)) < 0$,a contradiction.\\  
ad(2) If $CSS(A) \not= \langle CSS(A^-), A \rangle$, then there is some proper strong extension $C$
of  $CSS(A^-)$ with $C = C^-$, $CSS(A) = \langle C, A \rangle$, and the inequation in 
(2) is true.\\
ad (3) By Lemma \ref{subdelta3} we have 
\[ 0  \le \delta(\langle CSS(A) CSS_2(A) \rangle) - \delta(CSS(A)) \le \delta(CSS_2(A)) - 
\delta(CSS(A) \cap CSS_2(A)). \]
If $CSS_2(A) \nsubseteq CSS(A)$, then $CSS_2(A) \not= A$ and 
$\delta_2(CSS_2(A)) - 
\delta_2(CSS(A) \cap CSS_2(A)) < 0$. 
Hence $\delta(\langle CSS(A) CSS_2(A) \rangle) - \delta(CSS(A)) < 0$,
a contradiction.\\
ad (4) This is clear by definition.
\end{proof}

\noindent In the following Lemma we summarize facts we have already proved and used:

\begin{lemma}\label{useful}
Assume $A, B, C \subseteq M \in \K_3$ , $A, C \le_2 M$, and $A \cap C = B$.
W.l.o.g. we use $\subseteq$ instead of the corresponding embeddings.
\begin{enumerate}
\item $\F(B^*) \subseteq \F(A^*) \subseteq \F(M^*)$.
\item $\F(B^*) \subseteq \F(C^*) \subseteq \F(M^*)$.
\item $\F(A^*) \cap \F(C^*) = \F(B^*)$.
\item $ If \delta_2(A^*/C^*) = \delta_2(A^*/B^*)$, then $\langle A^* C^* \rangle = A^* \otimes_{B^*} C^*$
and $\F(\langle A^* C^* \rangle) = \F(A^*) \otimes_{\F(B^*)} \F(C^*) $
\item $\delta(A/C) = \delta(A/B)$ implies $\delta_2(A/C) = \delta_2(A/B)$ and $\langle AC \rangle =
A \otimes_B C$.
\item In general we have $\delta_2(A/C) \le \delta_2(A/B)$ and $\delta(A/C) \le \delta(A/B)$. Furthermore
$\delta_2(A/C) < \delta_2(A/B)$ implies $\delta(A/C) < \delta(A/B)$.
\end{enumerate}
\end{lemma}

\begin{proof}
ad (1),(2) By Lemma \ref{c=2,b} we have $B \le_2 A \le_2 M$ and $B \le_2 C \le_2 M$. By Theorem \ref{key3}
we get the assertion. \\

ad (3) $\F(A^* \cap C^*)_i = (\F(A^*) \cap \F(C^*))_i = (A \cap C)_i $ for $i = 1,2$.\\

ad(4) The first assertion follows from Corollary \ref{corsubdelta2}. Then we apply Lemma \ref{6.1}. By our assumption $\alpha$ and $\gamma$ are embeddings. \\

ad(5) It follows from Corollary \ref{corsubdelta3} and \ref{char}.\\

ad(6) Here we have the subaddivity of $\delta_2$ and $\delta_3$. The last assertion follows from (5).
\end{proof}

\begin{lemma}\label{6.10}
Let $M$ be a 3-nilpotent graded Lie algebra in $\K_3$.
For all $A \le^{2n+n^2}  C \le M$ we have $A \le^n M$.
\end{lemma}

\begin{proof}
We consider any $A \subseteq E \subseteq M$, such that 
$o-dim(E^*/A^*) \le n$.  By $C \le \langle C E \rangle$ we get $\delta_2(E^*/C^*) \ge 0$. Furthermore $o-dim_1(E^*/C^*) \le n$. By Lemma \ref{5.8}  we get:\\
Let $B$ be $E \cap C$. Then there exists $D$, $H$,  $K$ in $M^*$, 
and $X_C X_B X_E$, such that 
\begin{enumerate}
\item $ H \subseteq D  \subseteq C^*$, $B_1 \subseteq H $, $D_1 = C_1$, and 
$o-dim(H/B^*) \le n^2 + n$,
\item $H \subseteq K \subseteq \langle H E^* \rangle$, $K_1 = \langle H E^* \rangle_1$,
\item $\langle C^* E^* \rangle  = \langle D K \rangle = D \otimes_H \langle K \rangle$,
\item $X_C X_B X_A \subseteq \langle C_1 A_1 \rangle_2$ is linearly independent over 
$\langle D_1 \rangle_2 + \langle K_1 \rangle_2$, 
$X_B \subseteq B_2$, $\langle H X_B \rangle = \langle H B^* \rangle$,
\item $X_C \subseteq C_2$, $ \langle D X_B X_C \rangle = C^*$, and 
\item $X_E \subseteq E_2$, $\langle K X_B X_E \rangle = \langle H E^* \rangle$. 
\end{enumerate}
Now we can apply Lemma \ref{6.3}(3)  
to $C$, $\langle H_1 H_2 E \rangle  = E_H$, and  $C \cap E_H$
and obtain that $C \cap E _H \le E$. Hence $\delta(C \cap E_H) \le \delta(E)$. 
Note that \\ 
$C \cap E_H = \langle C \cap E, H_1, H_2 \rangle = \langle B, H_1, H_2 \rangle$.
Furthermore  in $M^*$ holds
\[o-dim(C^* \cap \langle E_H \rangle^*/A^*) \le 
o-dim(C^* \cap \langle E_H \rangle^*/B^*)  + o-dim(B^*/A^*)  \le  n^2  + 2n.\] 
Since  $A \le^{2n+n^2}  C$, we have $\delta(A) \le \delta(C \cap \langle E_H \rangle) \le \delta(E)$.
\end{proof}

\begin{lemma}\label{voramalk3}
Assume that $A \subseteq E \subseteq D = C \otimes_B A$. Then $\delta((E \cap C) \otimes_B A)
\le \delta(E)$.
\end{lemma}

\begin{proof}
By {\bf Mon} $E^0 = \langle (E \cap C) A \rangle = (E \cap C) \otimes_B A$. Since $A \subseteq E$
we have $E_1 = E^0_1 = \langle (E \cap C)_1 A_1 \rangle^{lin}$. $(E^0)^* = \langle E_1 X \rangle^*$,
where $X \subseteq E^0_2$ linearly independent over $\langle E_1 \rangle_2$.
Next we have
$E^* = \langle E_1 X Y \rangle^*$, where $Y \subseteq E_2$ is linearly independent over
$(E^0_2)$. Hence $\delta_2(E) = \delta_2(E^0) + \mid Y \mid$. \\
Since $A \subseteq E$, we get that $Y$ is linearly independent over
$\langle E^0_2 C_2 \rangle^{lin} \supseteq A_2$. Therefore this linear independence is given by the 
linear indpendence of $\{[x,y] : x \in X_1(C) , y \in X_1(A) \}$ as in Corollary \ref{c=3,a}
described. Hence
$\F(E^*) = \F((E^0)^*) \otimes \F(\langle Y \rangle^{lin}$. By Corollary \ref{c=3,a} again 
$E^- = \F(E^*)/N$ implies that $N \subseteq \F((E^0)^*)$. Hence $\delta(E^0) \le \delta(E)$.
\end{proof}

\noindent In the following we identify the isomorphic copies of $B$ in $A$ and $C$ and denote them by $B$.

\begin{theorem}\label{amalk3}
Assume that $B \le A$ and $B \le^{2(ldim(A/B)) + 2 +n} C $ are in $\K_3^{fin}$.
Then an  amalgam $D$ of $A$ and $C$ over $B$ exists in $\K_3^{fin}$ with $C \le D$ and $A \le^n  D$ .
If $B \le C$, then $A \le D$. If no divisor problem of $B$ has a solution in both $A$ and $C$, then 
$D = A \otimes_B C$ has the desired properties. In this case $\delta(D) = \delta(A) + \delta(C) - \delta(B)$.
\end{theorem}

\begin{proof}
Let $[b,?] = e$ be divisor problem in $B$ with solutions $a \in A$ and $c \in C$. 
Since $B \le A$ we have $\langle B a \rangle^{A} \cong B[e:b] = B'$ and by $B \le^{2(ldim(A/B)) + 2 + n} C$ also
$\langle B c \rangle^{C} \cong B[e:b] = B'$ . We use $B'$ for $\langle B a \rangle^{A}$ as a substructure of 
$A$ and also $B'$ for $\langle B c \rangle^{C}$ as a substructure of 
$C$. By Lemma \ref{e:b} $\delta(B') = \delta(B)$ and therefore
$B' \le A$ and $B' \le^{2(ldim(A/B') + 2 + n}   C$. Hence it is sufficient to prove the assertion for $A, B' , C$. Using an appropriate induction we can assume that there is no divisor problem in $B$ with solutions in both $A$ and $C$. In this case we show that 
$D = A \otimes_B C$  has the desired properties. By Theorem \ref{Lieam1} $D = A \otimes_B C$ exists.
By Theorem \ref{c=2,c} we know that $D^* \in \K_2^{fin}$.\\
Lemma \ref{6.3} (4) implies $C \le D$. Similarly we have $A \le D$, if $B \le C$.\\ 
Next we show $A \le^n D$. We consider $A \subseteq E \subseteq D$ with $o-dim(E^*/A^*) \le n$. By {\bf Mon} 
$\langle (C \cap E) A \rangle = (C \cap E) \otimes_B A$.
By Lemma \ref{voramalk3} we have $\delta((E \cap C) \otimes_B A) \le \delta(E)$. Since $B \le^{2(ldim(A/B))+ 2 +n} C $ it holds $B \le (E \cap C)$.
 Lemma \ref{6.3}(4) implies $A \le (E \cap C) \otimes_B A$.
Hence $\delta(A) \le \delta((E \cap C) \otimes_B A) \le \delta(E)$.\\

\noindent 
By Theorem \ref{c=2,c} we have $D^* \in \K_2$. 
To show that $D \in \K_3^{fin}$ we consider $E \subseteq D$. We have to show: If $0 < o-dim(E) $, then 
$\delta(E) \ge min\{2, o-dim(E)\}$. By Corollary \ref{3.5} we have to consider only $\delta$.
Since $D^* \in \K_2$ and
by Corollary \ref{3.5} we can assume w.l.o.g. that $E \le_2 D$. By Lemma \ref{rulesstrong3} we have $(E \cap C) \le E$. Therefore $\delta(E \cap C) \le \delta(E)$. \\
Furthermore w.l.o.g. $o-dim_3(E) = 0$.
Since 
$C \in \K_3$ 
$0 < \delta(E)$ for $(E \cap C ) \not= \langle 0 \rangle$ and $1 < \delta(E)$ for $1 < o-dim(E \cap C)$.\\
Now asume that $o-dim(E \cap C) \le 1$. Then $o-dim(E \cap B) \le 1$.   
An generating o-system for $E \cap A$ has the form $X(A)$ or $b, X(A)$, where \\
$b$ is a non-zero element of $B_1 \cap E_1$ or $B_2 \cap E_2$ and\\
$X(A) = X(A)_1 X(A)_2$, where \\
$X(A)_1$ consists of elements $a_i^1$ with $a_i^1 \in (A_1 \cap E_1)$ linerarly independent over $B_1$\\
and $X(A)_2$ consists of elements $a_i^2 \in (A_2 \cap E_2)$ linearly
independent over $\langle B X(A)_1 \rangle_2 $.\\
To get an   o-system for the vector space $E \cap \langle C  A\rangle^{lin}$ we have to extend the o-system above in the following way:\\
Either we add $c \in C_1 \setminus B_1$ or $c \in C_2 \setminus B_2$, if $E\cap C \not= \langle 0 \rangle$ and
$E\cap B = \langle 0 \rangle$. \\
Furthermore we have $Y = Y_1 Y_2$, where $Y_j$ consists of $d_i^j + e_i^j$, where the $d_i^j \in A_j$ are linearly 
independent over $\langle B_j,  \ldots a_i^j \ldots \rangle^{lin}$ and $e_i^j \in C_j$ linearly independent over $\langle B_j ,c \rangle^{lin}$, (sometimes there is no $c$).\\
Let $X$ be $X(A) , b$, if there is some $b$ and $X = X(A)$ otherwise.
By Corollary \ref{c=3,a} we can extend $X,Y$ or $X,c,Y$ respectively by some Z to get a generating  
o-system for $E$, where
$Z$ is linearly independent $\langle C  A \rangle^{lin}$.\\
Let $\Phi = \Phi_2 \Phi_3$ be an generating ideal system such that $E = F(XYZ)/\Phi$ or \\$E = F(X,c,Y,Z)/\Phi$.\\
Since $o-dim(C \cap E) \le 1$ we have  $\Phi \subseteq F(X)$. Hence $\delta_2(E) = \delta_2(E \cap A) + \mid c, Y , Z \mid$ or  $\delta_2(E) = \delta_2(E \cap A) + \mid  Y , Z \mid$ respectively, and\\
$\delta(E) = \delta(E \cap A) + \mid c, Y , Z \mid$ or $\delta(E) = \delta(E \cap A) + \mid c, Y , Z \mid$ respectively.
Then $min\{2,o-dim(E)\} \le \delta(E)$.

\end{proof}


\section{The strong Fra\"{\i}ss\'e Hrushovski Limit}

\noindent $\K^{fin}_3$ is countable and there are only countably many strong embeddings for Lie algebras in $\K^{fin}_3$. By Theorem \ref{amalk3} we have the amalgamation property for strong embeddings. As well known we get
the following strong Fra\"{\i}ss\'e Hrushovski limit. For more details see \cite{Zi11}.

\begin{theorem}{7.1}
There exists a countable structure $\M$ in $\K_3$ that satisfies the following condition:

\begin{itemize}
\item[rich] If $ B\le A$ are  in $\K_3^{fin}$ and there is a strong embedding $f$ of $B$ into $\M$, then it is possible to extend $f$ to strong embedding of $A$ into $\M$.
\end{itemize}
\end{theorem}

\noindent We call $\M$ the strong Fra\"{\i}ss\'e Hrushovski limit of $\K_3^{fin}$.

\begin{cor}\label{7.2}
$\M$ is unique up to isomorphisms.
\end{cor}

\noindent The proof is similar to the proof of the next result.

\begin{theorem}\label{7.3}
For any two rich structures $M$ and $N$ in $\K_3$ with $\langle \bar{a} \rangle \le M$,   $\langle \bar{b} \rangle \le N$, and $\langle \bar{a} \rangle^M \cong \langle \bar{b} \rangle^N$  we have $(M,\bar{a}) \equiv_{L_{\infty,\omega}} (N, \bar{b})$.
\end{theorem}

\begin{proof}
We will show that we can play the  Fra\"{\i}ss\'e Ehrenfeucht game infinitely many rounds reproducing the starting condition $\langle \bar{a} \rangle \le M$,   $\langle \bar{b} \rangle \le N$, and $\langle \bar{a} \rangle^M \cong \langle \bar{b} \rangle^N$ again and again. Assume w.l.o.g. that player I has chosen an element $c$ in $M$. Choose $\bar{c}$ in $M$ such that $\bar{c}$ contains $\bar{a}$ 
and $c$, and $\langle \bar{c} \rangle^{M} \le M$. Then 
$\langle \bar{a} \rangle^M \le \langle \bar{c} \rangle^M$. By assumption we have some isomorphism $f$ of $\langle \bar{a} \rangle^M$ onto $ \langle \bar{b} \rangle^N \le N$. By  richness $f$ can be enlarged to  embedding of $\langle \bar{c}  \rangle^M$ onto some $\langle \bar{d} \rangle^N \le N $.
\end{proof}

\noindent It follows that there is a common complete theory  of all rich Lie algebras in $\K_3$. In Corollary \ref{7.7} we 
will see that the following gives an axiomatization.

\begin{definition}
\begin{enumerate}
\item [$T_3(1)$] Let $T_3(1)$ be the elementary description of $\K_3$.
\item [$T_3(2)$] For every pair $B \le A $ in $\K_3$ we fix a function $g(n)$ (see next lemma), such that: \\
If $B \le^{g(n)} M$, then this embedding of $B$ in $M$ can be extendent to a $\le^n$-strong embedding of $A$ in $M$. 
\item[$(T_3)$ ]  We define $T_3 = T_3(1) \cup T_3(2)$.
\end{enumerate}
\end{definition}

\noindent Theorem \ref{amalk3} implies the following

\begin{lemma}\label{7.5}
Every rich structure $M$ in $\K_3$ fulfills the elementary axioms $T_3(2)$.

\end{lemma}

\begin{proof}
Let $C$ be the selfsufficient closure of $B$ in $M$. Then we have  $B \subseteq C \le M$. 
If $B \le M$, then $B = C$. We assume that $g(n) = 2(ldim(A/B)) + 2 + n^2 + n$ and 
$B \le^{2(ldim(A/B)) + 2 + n^2 + n} C $. 
We apply 
Theorem \ref{amalk3}  
to get an amalgam $D$ in $\K_3$ of $A $ and $C$ over $B$, such that $C \le D$ and 
$A \le^{n^2 + n} D$. Since $M$ is rich and $C \le M$ there is a strong embedding of $D$ in $M$ over $C$. 
By Lemma \ref{6.10} we get $A \le^n M$.

\end{proof}

\begin{theorem}\label{7.6}
A Lie algebra $M$ in $\K_3$ is rich iff it is a $\omega$-saturated model of $T_3$.
\end{theorem}

\begin{proof}
By Lemma \ref{7.5} a rich $M \in \K_3$ satisfies $T_3$. A $\omega$-saturated model of  $T_3$ is a rich Lie algebra in $\K_3$ by $\omega$-saturation. It remains to show that rich $M \in \K_3$ are $\omega$-saturated. Let $N$ be a $\omega$-saturated model of $T_3$. Then $N$ is rich. and by Theorem \ref{7.3} 
we have $M \equiv_{L_{\infty,\omega}} N$. Hence $M$ is $\omega$-saturated.
\end{proof}

\begin{cor}\label{7.7}
$T_3$ is an axiomatization of the complete theory of all rich graded Lie algebras in $\K_3$. If $M$ is a model of $T_3$ and $A$ is a strong substructure of $M$, then the elementary type of $A$ is completely
determined  by its isomorphism-type. 
\end{cor}

\begin{lemma}\label{7.8}
A model $M$ of $T_3$ satisfies the following:
\begin{enumerate}
\item For all $1 \le i < j \le 3$: $\forall xz \exists y(U_i(x) \wedge U_j(z) \rightarrow [x,y] = z)$. 
\item $M^* \models T_2$
\item If $M$ is rich, then $M^*$ is rich.
\end{enumerate}
\end{lemma}

\begin{proof}
To prove (1) assume $U_i(a)$ and $U_j(c)$. Let $B \le M$ with $a,c \in B$. If $[a,?] = c$ has no solution in $B$, then $B \le B(c:a)$
is in $\K_3$ by Lemma \ref{e:b}. By $T_3(2)$ there is a solution in $M$. 
For (2) and (3) it is sufficient to show that $M^*$ is rich. This follows since $M$ and therefore $M^*$ are 
$\omega$-saturated.
\end{proof}

\begin{lemma}\label{7.4a}
Let $A$ and $B$ be subalgebras of some $M \in \K_3$, where $B = \langle B_1, B_2, b_1 \ldots b_k \rangle$ and $b_1 \ldots b_k \in B_3$ are linearly independent over $\langle B_1 B_2 \rangle_3$.
\begin{enumerate}
\item If $B \le^{n + k} A$, then $\langle B_1 B_2 \rangle \le^n A$.
\item If $B \le A$,  then $\langle B_1 B_2 \rangle \le A$.
\end{enumerate}
\end{lemma}

\begin{proof}
(1) Choose $\langle B_1 B_2 \rangle \subseteq E \subseteq A$ with $o-dim(E^*/\langle B_1 B_2 \rangle^*) \le n$. Assume w.l.o.g. $B \cap E = \langle B_1 B_2 b_1 \ldots b_i\rangle$.
By assumption 
\[ \delta(B) = \delta(\langle B_1 B_2 \rangle) + k \le \delta(\langle E b_{i+1} \ldots b_k \rangle) = \delta(E) + k - i .\]
Hence $\delta(\langle B_1 B_2\rangle ) \le \delta(E)$.\\
\noindent (2) followa from (1).
\end{proof}



\section{Non-forking}
\begin{definition}
Let $X$ be any subalgebra  of $M \models T_3$. We say $X$ is strong in $M$ (short: $X \le M$), if for all finte $A \subseteq X$ we have $CSS(A) \subseteq X$. We define $CSS(Y) = \cup_{(A \subseteq Y,A finite)} CSS(A)$
for $Y \subseteq M$. 
\end{definition} 

\noindent Note that this definition generalizes the definition for finite $X$. Furthermore $CSS(Y) \le M$. 
Remember that $A,B,C,D$ are only used for finite substructures. 
By Corollary \ref{3.5} $B \le M$ can be defined using  $\delta$ only in the following way:\\
For all $A$ with $ B \subseteq A \subseteq M$ we have $\delta(B) \le \delta(A)$.

\begin{lemma}\label{8.2}
Let $M$ be a model of $T_3$. 
\begin{enumerate}
\item $X \le M$ if and only if for all finite $B \le M$ $\delta(B \cap X) \le \delta(B)$.
\item If $X \le M$ and $A \le M$, then $X \cap A \le M$.
\end{enumerate}
\end{lemma}

\begin{proof} 
ad (1): 
First assume $X \le M$. Consider $B \le M$. By assumption  $CSS(B \cap X) \subseteq X$. Hence 
$CSS(B \cap X) = B \cap X$. It follows $\delta(B \cap X) \le \delta(B)$. \\
To prove the other direction consider $A \subseteq X$. Since $CSS(A) \le M$ we get 
$\delta(CSS(A) \cap X) \le \delta(CSS(A))$  by assumption. Since $A \subseteq (CSS(A) \cap X)$ we get
$CSS(A) =  CSS(A) \cap X \subseteq X$. \\
ad (2): Since $X \le M$, there is some $D \le M$ such that $A \cap X \subseteq  D \subseteq X$. Then 
$A \cap X = A \cap D \le M$. 
\end{proof}

\noindent Let $\C$ be a monster model of $T_3$ and $M \preceq \C$. If $A \subseteq M$, then $CSS(A)$ belongs to the algebraic closure of $A$. Hence $CSS(A) \subseteq M$ and $M \le \C$. If $\bar{a} \in \C$, then by $M \le \C$ and definition  
\[ CSS(M \bar{a}) =  \cup_{B \subseteq M} CSS(B \bar{a}) = \cup_{B \le M} CSS(B \bar{a}) .\]
Note that $\delta(CSS(B \bar{a})) - \delta((CSS(B \bar{a}) \cap M) \ge 0$.

\begin{lemma}\label{8.3}
Assume $X \le \C$ and $\bar{a} \in \C$. 
Choose $\bar{a} \subseteq A \le CSS(\langle X \bar{a} \rangle)$, such that $\delta(A) - \delta(A \cap X)$ 
is minimal.\\ 
Then $CSS(\langle X \bar{a} \rangle) = X \otimes_{A \cap X} A$ 
and $A = CSS(\langle (A \cap X) \bar{a} \rangle)$.
\end{lemma}

\begin{proof}
Note $A \le \C$.
If $ (A \cap X) \le C \le X$, then by Lemma \ref{subdelta3} 

\[ \delta(CSS(\langle A C\rangle )) \le \delta( \langle A C \rangle )\le \delta(A) + \delta(C) - \delta(A \cap C) .\]

Note that $A \cap C = A \cap X$. Since $C \le CSS(\langle AC \rangle ) \cap X \le X$, we have

\begin{itemize}
\item [(*)] $\delta(CSS(\langle AC \rangle ))/(CSS(\langle AC \rangle ) \cap X) \le 
\delta(CSS(\langle A C \rangle ))/C)\\ \le \delta(\langle A C \rangle/C  ) \le
\delta(A/A \cap C)$. 
\end{itemize}

\noindent By minimality of $\delta(A/A \cap X)$ we have equality in (*) for all such $C$. Therefore
$\delta(CSS(\langle AC \rangle)) = \delta(\langle A C \rangle)$. Hence  $CSS(AC) = \langle A C \rangle$ and $CSS(AX) = \langle A X \rangle$. By Lemma \ref{char} we get $\langle A C \rangle = C \otimes_{A \cap C} A$, since 
$\delta (A/C) = \delta(A/A \cap C)$ by (*). We get $CSS(AX) = \langle A X \rangle = X\otimes_{A \cap X} A$\\
Since $CSS(\langle (A \cap X) \bar{a} \rangle) \subseteq A$ and 
$CSS(\langle (A \cap X) \bar{a}\rangle) \cap X = A \cap X$, we have for 
$A' = CSS(\langle (A \cap X) \bar{a} \rangle)$ that 
\[ CSS(\langle X \bar{a} \rangle) = X \otimes_{X\cap A'} A',  \]
by the considerations above. Since $\delta(A) = \delta(A')$, this implies 
$A = CSS(\langle (A \cap X) \bar{a} \rangle) $.

\end{proof}

\begin{cor}\label{8.4}
$T_3$ is $\omega$-stable.
\end{cor}

\begin{proof}
Let  $M$ be any countable model of $T_3$ and $M \preceq \C$. By Lemma \ref{8.3} for $a \in \C$ there are 
$B \le M$ and $A$ such that $B = (M \cap A) \le \C$, $A = CSS(Ba)$, and $CSS(Ma) = M \otimes_B A$. 
$tp(a/M)$ is uniquely determined by the isomorphism type  of $CSS(Ma)$ by Theorem \ref{7.3}.
Hence it is sufficient to count for all $B \le M$ all pairs $B \le A$ in $\K_3^{fin}$ with $A = CSS(Ba)$ for some $a$.
\end{proof}

\begin{lemma}\label{8.5}
Let $C$ be a subalgebra of $\C$. The algebraic closure of $C$ in $\C$ ($acl(C)$) is the transitive closure of 
$CSS(C)$ under adjunction of homogeneous divisors.
\end{lemma}

\begin{proof}
$CSS(C)$ is in the algbraic closure of $C$. A homogeneous  divisor is in the algebraic closure.
Let $X$ be the transitive closure of $CSS(C)$ under homogeneous divisors. Then $X$ is in the algebraic closure of $C$ and  $X \le \C$ by Lemma 
\ref{e:b}.\\
Assume that $a$ is not in $X$. Choose $A$ as in Lemma \ref{8.3}, such that \\
$A = CSS(\langle (A \cap X)  a\rangle)$ and 
$CSS(\langle Xa \rangle) = X \otimes_{X \cap A} A$. Let $X \cap A \le A^1$ be isomorphic to $A$ over $X \cap A$. All homogeneous divisors for $X \cap A$ in $A^1$ are in $X \cap A$. Hence
$(X \otimes_{X \cap A} A) \otimes_{X \cap A} A^1$ is in $\K_3$ by Theorem \ref{amalk3} applied  to all 
$X \cap A \le D \le X$. Since $\C$ is saturated we find $A^1$ in $\C$. Then 
$(X \otimes_{X \cap A} A) \otimes_{X \cap A} A^1 \le \C$. Since we can iterate this argument we see that
$A$ and $a$ are not in the algebraic closure of $C$.
\end{proof} 

\noindent Above we have already defined: \\
For subsets $A, B, C$ in a structure $M$ we define
\[ A \unab_B C\] if and only if  
\[\langle A B C \rangle = \langle A B \rangle \otimes_{\langle B \rangle} \langle B C \rangle .\]

\noindent Since $T_3$ is $\omega$-stable, we have non-forking $A \ind_B C$.

\begin{lemma}\label{8.6}
Let $B \subseteq A$ and $B \subseteq C$ be substructures of a monster model $\C$ of $T_3$. Then the following are equivalent:
\begin{enumerate}
\item $A \ind_B C $.
\item There are $A', B', C'$, such that
\begin{enumerate}
\item $CSS(A) \subseteq  A'  \subseteq acl(A)$,
\item $CSS(C) \subseteq C'  \subseteq acl(C)$, 
\item $B' = A' \cap C' \subseteq acl(B)$,
\item $\langle A' C' \rangle \cong A' \otimes_{B'} C'$, and $\langle A' C' \rangle \le \C$.
\end{enumerate}
\end{enumerate}
\end{lemma}

\begin{proof}
First we assume that (2) is true. Using Lemmas \ref{e:b}, \ref{c=2,a}, and \ref{4.15} we can extend $B'$ and $C'$, such that  every solution  of a division problem of $B'$  in $A'$  is already in $B'$. If we denote the new structures again by $A', B', C'$,
then they are again strong in $\C$ and (2) remains true.\\
Let $C_0 = C', C_1, \ldots$ be any insiscernible sequence over $B'$. 
There are automorphisms $f_i$ of $\C$ that fix $B'$ pointwise with $f_i(C') = C_i$.
Let $E_i$ be $CSS(\langle C_0 \ldots C_i \rangle )$. There are no new homogeneous divisors in $A'$ for problems in $B'$.
By Theorem \ref{amalk3} $A' \otimes_{B'} E_i$ exists in $\K_3$. 
Since $\C$ is rich, there is some $A_i \subseteq  \C$ such that \\
$\langle A' E_i \rangle \cong
\langle A_i E_i \rangle \cong A_i \otimes_{B'} E_i$ and $\langle A_i E_i \rangle \le \C$.
By saturation  there exists $A_{\omega}$ such that $\langle A_{\omega} E_i \rangle \le \C$ and 
$\langle A' E_i \rangle \cong \langle A_{\omega} E_i \rangle \cong A_{\omega} \otimes_{B'} E_i$ 
for all $i < \omega$. \\
By Theorem \ref{amalk3} $A_\omega \le A_{\omega} \otimes_{B'} E_i$.  $C_i \le \C$, since $C' \le \C$. We have 
$CSS(\langle A_\omega C_i \rangle) \subseteq A_{\omega} \otimes_{B'} E_i \le \C$. By the structure of a free 
product we get $\langle A_\omega C_i \rangle \le  A_{\omega} \otimes_{B'} E_i \le \C$: \\
To show this consider $\langle A_\omega C_i \rangle \subseteq D \subseteq  A_{\omega} \otimes_{B'} E_i$. Then
\[ \delta(\langle A_\omega C_i \rangle = \delta(A_\omega) + \delta(C_i) - \delta(B') \le \]
\[ \delta(A_\omega) + \delta(E_i \cap D) -  \delta(B') = \delta(A_\omega \otimes_{B'} E_i \cap D) .\]
By Lemma \ref{voramalk3} we get
\[\delta(\langle A_\omega C_i \rangle)  \le  \delta(A_{\omega} \otimes_{B'} E_i \cap D) \le \delta(D). \]
Now we can applay Theorem \ref{7.3} and get  $tp(A_{\omega}/C_i) =   tp(A'/C')$ for all $i < \omega$. Then $A' \ind_{B'} C'$ and therefore 
$A \ind_B C$, 
since $A \subseteq A' \subseteq acl(A)$, $C \subseteq C' \subseteq acl(C)$, and $B \subseteq
B' \subseteq acl(B)$.\\

\noindent Assume (1) holds. We extend  $A', B'$  in (2 a, b, c) in  such a way, that every homogeneous solution in $C'$ of a divisor problem of 
$B'$ is already in $B'$. Use Lemmas \ref{4.15}, \ref{c=2,a}, and \ref{e:b} . We use the  same notation.
Then $A' \le \C$ and $C' \le \C$
By assumption $A' \ind_{B'} C'$ and $B' \subseteq acl(B)$. Define $C_0 = C'$. By Theorem \ref{amalk3}
$C_0 \otimes_{B'} C_1$ exists in $\K_3$, where $C_0 \cong C_1$ and both are strong in $C_0 \otimes_{B'} C_1$.
Since $\C$ is rich there exits $C_1$ in $\C$, such that $C_0 \otimes_{B'} C_1 \le \C$. Then $C_1 \le \C$ and there is an 
automorphism$f_1$ of $\C$ with $f_1(C_0) = C_1$ and $f_1(C_1) = C_0$. By induction we get a sequence $C_0, C_1, \ldots$ in $\C$, such that $C_i \le \C$ and $E_i = C_0 \otimes_{B'} C_1 \otimes_{B'} \ldots \otimes_{B'} C_i \le \C$.
Every substructure generated by a finite subset of the $C_i$ is strong in $\C$ and isomorphic to the free amalgam of these $C_i$ over $B'$. (Use Theorem \ref{amalk3} and Theorem \ref{7.3}). Hence the sequence of the $C_i$ is an indiscernible sequence over $B'$. Let $f_i$ be the automorphism of $\C$, that exchanges $C_0$ and $C_i$ and fixes all the other $C_j$.
Since $A' \ind_{B'} C'$ there is some $A_{\omega}$ in $\C$, such that $tp(A_{\omega}/C') = tp(A'/C')$ 
and $C_0, C_1 \ldots $ is indiscernible over $A_{\omega}$.
Note that $B' \subseteq A_{\omega}$. 
$f_i(tp(A_{\omega}/C') = tp(A_{\omega}/C_i)$. The following claim implies the assertion:

\[(+) \langle  A_{\omega} C_i \rangle = A_{\omega} \otimes_{B'}  C_i  \le \C .\]

(+) implies 
\[(++)  \langle A' C' \rangle = A' \otimes_{B'} C' \le \C.\]
Then we have $\delta(\langle A' C' \rangle) = \delta(A') + \delta(C') - \delta(B')$. The $\delta$'s of the old and the 
new $A', B', C'$ are the same. Hence we have (++) also for the original $A', B', C'$.

\noindent First we assume \[(*) \langle  A_{\omega} C_0 \rangle \not= A_{\omega} \otimes_{B'} C_0.\] 
Then Lemma \ref{char} implies

\[ \delta(A_{\omega}/C_0) < \delta(A_{\omega}/B'). \]
Next we show, that (*) implies 
\[\langle  A_{\omega} C_0 C_1\rangle \not= \langle A_{\omega}  C_0 \rangle \otimes_{C_0} \langle C_0 C_1 \rangle.\]
Otherwise 
\[\langle  A_{\omega} C_0 C_1\rangle = \langle A_{\omega} C_0 \rangle \otimes_{C_0} \langle C_0 C_1 \rangle.\]
and
\[ \langle C_0 C_1 \rangle = C_0 \otimes_{B'} C_1 \]
would imply by {\bf Sym} and {\bf Trans}
\[\langle  A_{\omega} C_0 C_1\rangle = \langle A_{\omega} C_0 \rangle \otimes_{B'} \langle  C_1 \rangle\] 
and by {\bf Mon}
a contradiction to (*).
Then 
\[\delta(A_{\omega}/ \langle C_0 C_1 \rangle) < \delta(A_{\omega}/ C_0) .\]
By induction we get:
\[\delta(A_{\omega}/ \langle C_0 \ldots C_i \rangle) < \delta(A_{\omega}/ \langle C_0 \ldots C_{i-1} \rangle) ,\]
since otherwise equality would be  equivalent to $A_{\omega} \unab_{\langle C_0 \ldots C_{i-1}\rangle } C_i$
and this implies together with $C_i \unab_{B'} \langle C_0 \ldots C_{i-1} \rangle$ \\
\[ C_i \unab_{B'} \langle C_0 \ldots C_{i-1} A_{\omega} \rangle .\]
in contradiction to (*) by {\bf Mon}.\\
Hence there is some j such that
\[\delta(A_{\omega}/ \langle C_0 \ldots C_j \rangle) < 0 .\]
This contradicts
\[ \langle C_0 \ldots C_j \rangle) \le \C .\]
We have shown the first part of the claim. \\

\noindent Using the same steps we can show, that
\[ \langle A_{\omega} E_i \rangle = A_{\omega} \otimes_{B'} E_i.  \] 
Then \[   \delta(\langle A_{\omega} E_i \rangle) =  \delta (A_{\omega})  + \delta(E_i) - \delta(B').\]

\noindent It remains to prove 
\[ \langle A_{\omega} C_i \rangle \le \C. \]
Otherwise we have
\[ \delta(CSS(\langle A_{\omega} C_i \rangle) \le  \delta (A_{\omega})  + \delta(C_i) - \delta(B')  - 1.\]
We show  by induction on i that this implies
\[ \delta(CSS(\langle A_{\omega} E_i \rangle) \le  \delta (A_{\omega})  + \delta(E_i) - \delta(B')  - (i + 1).\]
\noindent Then 
\[ \delta(CSS(\langle A_{\omega} E_j \rangle) <  \delta(E_j) \]
for sufficiently large j - a contradiction. 
We denote
\[ CSS(\langle A_{\omega} E_{i+1} \rangle) \cap CSS(\langle A_{\omega} C_{i+1} \rangle) \]
by $D_i$. Then $\delta(A_{\omega}) \le \delta(D_i)$ , since $A_{\omega } \le \C$. 
For the induction we have
\[ \delta(CSS(\langle A_{\omega} E_{i+1} \rangle) \le  \delta(\langle CSS(\langle A_{\omega} E_i \rangle) 
CSS(\langle A_{\omega} C_{i+1} \rangle) \rangle) \le  \]
\[\delta (A_{\omega})  + \delta(E_i) - \delta(B')  - (i + 1) + \delta (A_{\omega})  + \delta(C_{i+1}) - \delta(B')  - 1 
-\delta(D_i)  \le\]
\[ \delta (A_{\omega})  + \delta(E_{i+1}) - \delta(B')  - (i + 2)  =.\]
\[ \delta(\langle A_{\omega} E_{i+1} \rangle) - (i + 2). \]

\end{proof}

\begin{cor}\label{8.7}

\begin{enumerate}
\item Assume $B \le A \le \C$ and $B \le X \le \C$, such that $A \cap X = B$ where $X$ can be infinite.
Then $A \ind_B X$ \iff $ A \unab_B X$  and $\langle A X \rangle \le \C$ \iff $\delta A/X) = \delta(A/B)$
and $\langle A X \rangle \le \C$ .
\item Types over strong substructures are stationary.
\end{enumerate}
\end{cor}

\noindent We come back to Lemma \ref{8.3}. Note that every element in $\C$ is interdefinable with 
a sequence of homogeneous elements.

\begin{lemma}\label{8.8} Assume $U \le V \le \C $, $\bar{a} \subseteq \C$,
$B \le U$,
$CSS(B \bar{a}) =  A_U$,  $A_U \cap U = B$, and 
\[ CSS(U \bar{a}) =    U \otimes_B A_U . \]
Assume $acl(\langle U \bar{a} \rangle) \cap V = U$,
$C \le V$,  , $A _V = CSS(C \bar{a})$, $A_V \cap V = C$, and
\[  CSS(\langle V  \bar{a} \rangle) = V   \otimes_C A_V. \]
If $D = B \cap C$ and $A = A_U \cap A_V$, then $A=  CSS( \langle D \bar{a} \rangle) $ and
\[ CSS( \langle U\bar{a}  \rangle) = U \otimes_D  A.\]
\end{lemma}

\begin{proof}
Extend $\bar{a}$ to $\bar{e}$, such that 
$\langle \bar{e}\rangle  \cap U = \langle \bar{a}\rangle  \cap U = \langle \bar{a} \rangle \cap V$ and 
$\bar{e}$  generates $A$ over $D$.\\ 
By {\bf Mon} $\langle U \bar{e} \rangle = U \otimes_B \langle B \bar{e}
\rangle$ and $\langle V \bar{e} \rangle = V \otimes_C \langle C \bar{e} \rangle$. Then
Theorem \ref{cap} implies 
\[ \langle U \bar{e} \rangle = U \otimes_D \langle D \bar{e} \rangle. \]
We get $\langle B A \rangle = \langle B \bar{e} \rangle = B \otimes_D \langle D \bar{e} \rangle$.\\
If $\langle B A \rangle = A_U$, then $\langle U \bar{e} \rangle = CSS(\langle U \bar{a} \rangle)$,
as desired. \\
Otherwise there is some $c \in A_U$, that is not in $\langle B A \rangle$. Then $c$ is not in $U$,
not in  $V$, and not in $A_V$.\\
If we consider $c$ as a linear combination over a vector space basis of 
\[  CSS(\langle V  \bar{a} \rangle) = V   \otimes_C A_V \]
as described in Lemma \ref{c=3,a}, then there are mixed monomials with elements in $V \setminus C$
and $A_V \setminus C$. But then $c$ is not involved in any relations. Therefore $c \notin CSS(\langle
B \bar{a} \rangle = A_U$, a contradiction.
\end{proof}

\noindent In the next Corollary we use Lemma \ref{8.8} for the case $U = V$.

\begin{cor}\label{8.9}
Assume $U \le \C$ and $\bar{a} \in \C$. According to Lemma \ref{8.3}we have:
If $A$ is 
chosen with  $\bar{a} \subseteq A \le CSS(\langle U \bar{a} \rangle)$, such that $\delta(A) - \delta(A \cap U)$ 
is minimal,\\ 
then $CSS(\langle U \bar{a} \rangle) = U \otimes_{A \cap U} A$ 
and $A = CSS(\langle (A \cap U) \bar{a} \rangle)$.\\
There is an $A$ as above, that is minimal with these properties. We call $B = A \cap U$ the 
self-sufficient  canonical base algebra of $tp(\bar{a}/U)$.
\end{cor}

\begin{lemma}\label{8.10}
We consider $U \le \C$ and $tp(\bar{a}/U )$.
Let $B$ be the self-sufficient base algebra of $tp(\bar{a}/U)$, as in Corollary \ref{8.9}.
Then $B$  is a weak canonical base of  $tp(\bar{a}/U)$. With respet to the given properties it is unique.

\end{lemma}

\begin{proof}
We claim that $B$ is a weak canonical base of $tp(\bar{a}/U)$.
By Corollary \ref{8.7} we have $U \ind_B A$.
Since $\C$ is saturated, we can replace $U$ by a saturated model $U \le M \preceq \C$ such that $M \ind_B A$. We have the same situation for $M$ as for $U$:\\
$A = CSS(\langle B \bar{a} \rangle) = \langle B \bar{e} \rangle$, $A \cap M = B$, $CSS(M \bar{a})  = M \otimes_B A$.
and $A$ minimal with this properties.\\
Now we use Lemma \ref{8.8} for $M = N$. 
Let $f$ be an automorphism of $\C$ that fixes $M$ setwise. \\
First we assume, that  it fixes  $B$ pointwise. 
If $g$ is $f$ for $M$ and the identity for $\bar{a}$, then it can be extended to $CSS(M \bar{a})  = M \otimes_B A$. 
Since $ M \otimes_B A \le \C$, $g$ is an elementary automorphism and therefore
$tp(\bar{a}/M) = f(tp(\bar{a}/M)$ .\\
Conversely assume that $f$ fixes $tp(\bar{a}/M)$. 
W.l.o.g. we can assume that $f$ fixes $\bar{a}$ pointwise. 
Then $f(A) \cap M = f(B)$. By Lemma \ref{8.8} we get $CSS(M \bar{a}) = M \otimes_{B \cap f(B)} \langle (B \cap f(B)) \bar{e} \rangle $. Hence $B = f(B)$. 
\end{proof}

\begin{cor}\label{8.11}
$T_3$ has the weak elimination of imaginaries. 
\end{cor}

\noindent In section 10 we will introduce minimal strong prealgebraic extensions over strong subalgebras. 
We have canonical base algebras for them.
\noindent We use $CB(\bar{a}/X)$ to denote the canonical base  of $tp(\bar{a}/X)$. \\

\noindent Let us come back to the Lemma  \ref{8.9}.Now we assume 
that $\bar{a} $ is generating o-system over $B$ and over $X$.
\[ \langle X \bar{a} \rangle = X \otimes_B \langle B \bar{a} \rangle . \] 
Let $X^B$ be a generating o-system  of $B$. Then  $B^* = F(X^B)^*/\Phi_2^B$, 
where $F(X^B) $ is the free algebra freely generated by $X^B$ and $\Phi_2^B \subseteq F(X^B_1)^*_2$.
Furthermore $B^- = \F(B^*)/\Phi_3^B$, where $\Phi_3^B \subseteq \F(B^*)_3$. 
Again we use elements of $\langle B \bar{a} \rangle$ to denote 
their preimages in $F(X^B \bar{a})$
and in $\F(\langle B \bar{a} \rangle^*)$.\\
The extension of $X$ above is completely determined by $\langle B \bar{a} \rangle$. There are
sets $\Phi_2$ and $\Phi_3$ such that
\[ \langle B \bar{a} \rangle^* \cong F(X^B \bar{a})^*/ \langle \Phi_2^B \Phi_2 \rangle,\] 

\[\langle B \bar{a} \rangle^- \cong \F(\langle B \bar{a} \rangle )/ \langle \Phi_3^B \Phi_3 \rangle,\]
where $\Phi_2 \subseteq F(X^B \bar{a})^*_2$ is linearly independent over $F(X^B)$ and
$\Phi_3 \subseteq \F(\langle B \bar{a} \rangle^*)$ is linearly independent over $\F(B^*)$.\\
See Definitions \ref{2.7} and \ref{2.8}.

\begin{definition}\label{8.12}
The elements of $\Phi_i$ for $i = 2,3$ are called relations of degree i. \\
Their structure is described in the proof of Theorem \ref{cap}. We have coefficients and ends.
\end{definition}

\begin{cor}\label{8.13} In Lemma \ref{8.9}
$B$ is the self-sufficient closure of the substructure of $X$ generated by
the ends and the coefficients of an ideal basis $\Phi_2 \Phi_3$ in the situation above. 
\end{cor}

\begin{cor}\label{8.14}
$T_3$ is CM-trivial.
\end{cor}

\begin{proof}
Assume $M \preceq N \preceq \C$ and $acl(\bar{a} M) \cap N = M$. We have to show that 
$CB(\bar{a}/M) \subseteq acl^{eq}(CB(\bar{a}/N))$. \\
Let $B$ be the self-sufficient canonical base algebra of $tp(\bar{a}/M)$ and $C$ be the 
self-sufficient canonical base algebra of
$tp(\bar{a}/N)$. We have to show, that $B \subseteq C$.\\
Then we have 
$B \le M$,
$CSS(B \bar{a}) =  A_M$,  $A_M \cap M = B$, and 
\[ CSS(M \bar{a}) =    M \otimes_B A_M . \]
Furthermore $acl(\langle M \bar{a} \rangle) \cap N = M$,
$C \le N$,  , $A _N = CSS(C \bar{a})$, $A_N \cap N = C$, and
\[  CSS(\langle N  \bar{a} \rangle) = N   \otimes_C A_N. \]
By Lemma \ref{8.8} 
\[  CSS(\langle M  \bar{a} \rangle) = M   \otimes_{B \cap C} A, \]
where $A = CSS((B \cap C) \bar{a})$. Since $B$ is minimal we get $B \subseteq C$, as desired.

\end{proof}


\section{geometry}
\noindent Let $\C$ be a monster model of $T_3$ again. 
Let $U, V, W$ be substructures  of $\C$ in $\h_3$. They are generated by their projections into $\C_1$ and $\C_2$ . If they are finite we use often $H, J, K$. $\h_3^{fin}$ is  the subset of the finite substructures in $\h_3$.  Since substructures are graded, the set of all $H \in \h_3$  of o-dimension $ \le1$  is $\{\langle a \rangle: a \in(\C_1 \cup \C_2) \}$. Note that for $a_1 \in \C_1$ and $a_2 \in \C_2$ we have $\langle a_1 + a_2 \rangle = \langle a_1, a_2 \rangle$ and  o-dimension and $\delta$ of this substructure is 2.

\noindent To compute $\delta(H)$ for $H \in \h_3^{fin}$ , we consider $ H  \cong \F(H^*)/N$. Then $\delta(H) = \delta_2(H) - ldim(N)$. 

\noindent We define a pregeometry on $R(\C) = \C_1 \cup \C_2$. For this we define a dimension function for all  $H \in \h_3^{fin}$. 
\begin{definition}
The dimension of a finite subset  $X = X_1 \cup X_2$ with $X_1 \subseteq \C_1$ and
$X_2 \subseteq \C_2$
of $R(\C)$ is the dimension of the substrucure $H = \langle X \rangle  \in \h_3^{fin}$ generated by these elements:
$ d(H) = \delta(CSS(H))$.\\
$d(A/H) = d(\langle A H \rangle) - d(H)$ and 
$d(A/U) = min\{d(A/H) : H \subseteq U\}$.
\end{definition}

\noindent We have  $CSS(H) = \langle CSS(H)_1 CSS(H)_2 \rangle$. Furthermore $d(\langle 0\rangle) = 0$\\

\begin{lemma}\label{9.2}
$d$ defines a dimension function on $R(\C)$.
\end{lemma}

\begin{proof}
For $H$ and $K$ in $\h_3$ we have to show:

\begin{enumerate}
\item $d(\emptyset) = d(\langle 0\rangle) = 0$.
\item $d(H) >0$, if $H \not= \langle 0 \rangle$, since $\C \in \K_3$.
\item $d(\langle H a \rangle) \le d(H) + 1$  for $a\in R(\C)$.
\item If $K \subseteq H$, then $d(K) \le d(H)$.
\item $d(\langle H K \rangle) \le d(H) + d(K) - d(\langle (H_1 \cap K_1) (H_2 \cap K_2)\rangle)$.
\end{enumerate}

\begin{itemize}
\item[ad(3)] If $a \in CSS(H)$, then $d(H) = d(\langle H a \rangle)$. Otherwise we use $CSS(H) \le \langle CSS(H) a \rangle$. Then $\F(CSS(H)^*) \subseteq \F(\langle CSS(H) a \rangle^*)$. Hence
\[ d(\langle H a \rangle) \le \delta(\langle CSS(H) a \rangle) \le \delta(CSS(H)) + 1 = d(H) + 1 .\]

\item[ad(4)]By Lemma \ref{css} we get $CSS(K) \subseteq CSS(H)$. Hence  $d(K) \le d(H)$.

\item[ad(5)] $CSS(H)$ and $CSS(K)$ are strong subsructures of $CSS(\langle H K \rangle)$. By Theorem \ref{subdelta3}
\[ d(\langle H K \rangle) \le \delta(\langle CSS(H) CSS(K) \rangle)\] 
\[ \le \delta(CSS(H)) + \delta(CSS(K)) - \delta(CSS(H) \cap CSS(K)). \]

\noindent $CSS(\langle (H_1 \cap K_1), (H_2 \cap K_2) \rangle) \le CSS(H) \cap CSS(K)$. Hence 
$d(\langle (H_1 \cap K_1) (H_2 \cap K_2)\rangle) = \delta(CSS(\langle (H_1 \cap K_1) (H_2 \cap K_2)\rangle))  \le \delta(CSS(H) \cap CSS(K))$

\end{itemize}

\end{proof}

\noindent We define $cl$ for substructures $H$ and $U$ of $\C$ in $\h_3$ and  on the set $R(\C)$.

\begin{definition}\label{9.3}
For $H \subseteq \C$ and $H \in \h_3^{fin}$  we define 
\[ cl(H) = \langle \{a \in R(\C):   d(\langle H a \rangle = d(H) \} \rangle.   \]
$cl(U)$ is the union of all $cl(H)$ for $H \subseteq U$.\\
For  $X \subseteq R(\C)$ we define 
\[ cl(X) = cl(\langle X \rangle) \cap R(\C).  \]

\end{definition}

\noindent Then the following is well-known:

\begin{cor}\label{9.4}
$cl$ is a pregeometry on $R(\C)$. Furthermore $acl(H) \subseteq cl( H )$. If $H \le K \le \C$, then $K \subseteq cl(H)$ \iff $\delta(K) = \delta(H)$.
\end{cor}

\begin{lemma}\label{9.5}
Assume $U \le \C$, $a \not\in U$, and $a \in R(\C)$. Then there are the following possibilities:
\begin{enumerate}
\item $d(a/U) = 1$. In this case $\langle U a \rangle = CSS(U a)$. (Transcentental Case)
\item $d(a/U) = 0$.
\begin{enumerate}
\item $CSS(Ua) = \langle U a \rangle$ . Then there are $u_1 \in U_i(U)$ and $u_2 \in U_j(U)$ such that
$i <  j$ and $\C \models [u_1, a] = u_2$. (Algebraic Case)
\item $\delta(a/U) > 0$.
\end{enumerate}
\end{enumerate}
\end{lemma}

\begin{definition}\label{9.6} We consider $U \in \h_3$, $U \le \C$  and  
$V \in \h_3$ finitely generated over $U$. 
\begin{itemize}

\item If $V  = \langle U a \rangle $ for some $a \in R(\C)$ and $a \not\in cl(U)$, then  
we call $V$ a minimal  trancentental  extension.
\item If $V  = \langle U a \rangle $ for some $a \in R(\C)$ and $a \in acl(U)$, then 
we call $V$ a minimal algebraic extension.
\item If $V  \not= \langle U a \rangle $ for all $a \in R(\C)$, $\delta(V/U) = 0$, and $\delta(W/U) > 0$ for all $U \subseteq W \subseteq V$, $U \not= W \not= V$, and $W \in \h_3$, then $V$ is minimal prealgebraic over $U$. (short prealgebraic)
\item We use the definition of minimal prealgebraic also for $U$ not strong in $\C$.

\end{itemize}
\end{definition}

\noindent If $H \le K$ and $K$ is minimal strong, then $K$ is generated over $H$ by a single element or it is 
prealgebraic minimal (Lemma \ref{9.5}. \\

\begin{lemma}\label{9.7}
Assume $U \le V \le \C$ and $V/U$ is finite. Then there is a sequence $U = V_0 \le V_1 \le \ldots  \le V_n  = V$, such that 
$V_{i+1}$ is minimal strong over $V_i$.  
\end{lemma}



\section{Prealgebraic minimal extensions}

\begin{lemma}\label{10.1}
Assume that  $H \le K \le \C$ is prealgebraic
minimal. 
\begin{enumerate}
\item If $ H \le  U  \le \C$, then $K \le U$ or $U \cap K = H$ and $\langle U K \rangle$ is a minimal
prealgebraic extension of $U$ and $\langle U K \rangle = U \otimes_H K$. $tp(K/U)$ is the unique nonforking
extension of $tp(K/H)$.
\item $tp(K/H)$ is strongly minimal.
\end{enumerate}
\end{lemma}

\begin{proof}
By Lemma \ref{subdelta3} we have 
\[ 0 \le \delta(K/U) \le \delta(K/(K \cap U)) .\] 
If $H \not= K \cap U \not= K$, then $\delta(K/(K \cap U)) < 0$, a contradiction. If $K \cap U = H$, then 
\[ \delta(K/U) = \delta(K/H) = 0 \] 
and therefore $\langle U K \rangle = U \otimes_H K$ by Lemma \ref{char}.

\end{proof}

\noindent Now we consider a prealgebraic minimal strong extension $H \le K \le \C$. W.l.o.g we can assume that there is some $M \preceq \C$, such that  $tp(K/M)$  with $K \cap M = H$ is a nonforking extension. By Lemma \ref{10.1}:
\[ \langle K M \rangle = K \otimes_H M \le \C .\] 
By Lemma \ref{8.9} we get $A \le \langle M K \rangle $ and $B = M \cap A$, such that
\[ \langle K M \rangle = \langle A M \rangle = A \otimes_B M ,\]
$o-dim_3(A/B) = 0$, $A = \langle B \bar{a} \rangle$ where $\bar{a}$ is a generating o-system of $A$ over $B$,
and $B$ is a self-sufficient weak canonical base algebra of $tp(A/M)$.\\
$B$ and $A$ are  strong substructures and minimal with these properties. $B$ is in $\K_3$. 
Now we will only assume that $A$ is $\le^m$-strong for sufficiently large $m$. \\

\begin{lemma}\label{10.2}
Assume $U \le \C$, $\bar{a}$ is a generating o-system of $\langle U \bar{a} \rangle \le \C$ 
over $U$, $o-dim_3(\bar{a}) = 0$, and $\langle U \bar{a} \rangle$ is  prealgebraic minimal extension of $U$.
\begin{enumerate}
\item 
If we have $\langle U \bar{a} \rangle = U \otimes_B \langle B \bar{a} \rangle$ and 
$\langle U \bar{a} \rangle = U \otimes_C \langle C \bar{a} \rangle$.\\ 
Then
$\langle U \bar{a} \rangle = U \otimes_{B \cap C} \langle (B \cap C) \bar(a) \rangle$.
\item  If  we assume that 
$ B \ \le^k \C$ and  $ C  \le^k \C$ 
for $o-dim( B^*), o-dim(C^*) < k$,
then $ (B \cap C)  \le^k \C$.\\
For $D \subseteq B$ we have $ D  \le^k \C$ if and only if 
$ D  \le B  $.

\end{enumerate}
\end{lemma}

\begin{proof}
(1) is is a consequnece of Corollary \ref{cap}.\\
(2) follows from Corollary \ref{rulesstrong3}.

\end{proof}

\begin{definition}\label{10.3}
For every prealgebraic minimal extension $ B \le A = \langle B \bar{a}\rangle  \le \C$, where $B$ is the 
self-sufficient canonical base algebra  as in
Lemma \ref{8.9} and $\bar{a} $ is a generating o-system of $A$ over $B$
we define a formula $\phi(\bar{x},\bar{y})$. $\bar{x}$ stands for $\bar{a}$ and $\bar{y}$ for a generating 
o-system for $B$. $\phi$  describes the following:
\begin{enumerate}
\item The isomorphism-type of $A$ over $B$.
\item There is no proper subalgebra $D \subseteq B$, such that $\langle D \bar{a} \rangle \le A$ 
and $\langle D \bar{a} \rangle$ describes a minimal prealgebraic  extension of $D$.
\item For $A/B$ we fix some $m(A/B) = m$, such that
$2\;lindim(A/B) + 2 < m$ and 
$h(\mid \bar{x} \mid) + \mid \bar{x} \mid <m$ where $h$ is the function in Lemma \ref{5.7}.
Then the formula says, that $B \le^{m + o-dim(A^*/B^*)} \C$.
\end{enumerate}
Let $\X^{home}$ be the set of these formalas. Instead of $\phi(\bar{a},\bar{b})$ we often write
$\phi(A/B)$.
\end{definition}

\noindent In the above definition $\phi(A/B)$ implies , that $A$ is a prealgebraic minimal strong extension 
of $B$, $A$ is  $\le^m$-strong in $\C$, $B$ is minimal with this properties.

\begin{lemma}\label{10.4}
Let $B \le A $ be in $\C$ with $\C \models \phi(A/B)$ for  $\phi(\bar{x},\bar{y}) \in \X^{home}$, as described above.
Assume $B \subseteq  V $ for $V = \langle V_1 V_2 \rangle$,  $A \not\subseteq V$, 
and $\delta(A/V) \ge 0$.
Then 
\[\langle A V  \rangle =  A \otimes_B V. \] 
$A$ provides  a prealgebraic minimal strong extension of $V$.
\end{lemma}

\begin{proof} 
We have  $\langle V A \rangle \in \h_3$. We can asssume w.l.o.g. that $V$ is finite. Otherwise we work with sufficiently large substructures. 
Let $V^s$ be the selfsufficient closure of $V$ in $\langle V A \rangle$. By definiton of the selfsufficient closure we have
$\delta(A/V^s) \ge 0$   and \\
$\langle V A \rangle  \not= V^s$, since $\delta(\langle V A \rangle) \ge \delta(V)$.\\ 
Now $V^s$ fullfils all assumptions on $V$ in the Lemma.  
We show the Lemma for $V^s$. Then it follows $V = V^s$ and the assertion for $V$. Hence we can assume 
w.l.o.g. that $V$ is strong in $\langle VA \rangle$. \\
By Corollary \ref{3.5} $0 \le \delta_2(A/V)$.
\noindent For $\C^*$ Lemma \ref{5.7} implies that there are  $X_V \subseteq \langle V_1 A_1 \rangle_2 \cap 
V_2$ linearly independent  over $\langle V_1 \rangle_2 + \langle A_1 \rangle_2 + (V \cap A)_2$,
$D$, and  $H$, such that
$A^*  \cap  V^* \subseteq   H \subseteq D \subseteq V^*$ , $D_1 = V_1$, 
$V^* = \langle D X_V \rangle $, $ o-dim(H/(A^* \cap V^*)) \le h(o-dim(A^*/V^*)) $ and
\begin{itemize}
\item[(i)] $\langle V A \rangle^*= D \otimes_{H} \langle H A^* \rangle$. 
\end{itemize}

\noindent Let $E_1 \subseteq V_1$ and $E_2 \subseteq V_2$ in $V$ be the  isomorphic preimages of $H_1$ and $H_2$. Then define 
$E = \langle E_1 E_2 A \rangle \cap V$. Then $E^*$ is the given $H $ and $\langle E \cap V \rangle = E$.
Furthermore we get $o-dim(E^*/A^*) \le o-dim(E^*/A^* \cap E^*)
\le h(o-dim(A^*/V^*))  \le m$, where $m$ comes from the definition of $\phi$.
This we need later for (v).\\
Since $V \le_2 \langle A V \rangle$ we get \\
\begin{itemize}
\item[(ii)]$E \le_2 \langle E  A \rangle$
\end{itemize}
by Lemma \ref{c=2,b}.\\
By Lemma \ref{6.3}(1) the conditions  (i) , (ii) and $V \cap \langle E A \rangle = E$  imply
\begin{itemize}
\item[(iii)] $0 \le \delta(A/V) \le \delta(A/E) , $
\end{itemize}
(iv) follows from Lemma \ref{6.3}(3), and (i) and (ii) since $V \le \langle V A \rangle$.
\begin{itemize}
\item[(iv)] $E \le \langle E A \rangle$.
\end{itemize}

\noindent  
As seen before (ii)  we have $o-dim(E^*/A^*) \le m$. Hence
we have $A \le \langle A E \rangle$ by $\C \models \phi(A/B)$.  
\begin{itemize}
\item [(v)]$E$ and $A$ are
strong in $\langle E A \rangle$. 
\end{itemize}
By Lemma \ref{subdelta3} and (iii)
\begin{itemize}
\item[(vi)] $ 0 \le \delta(A/V) \le\delta(A/ E ) \le \delta(A/ E \cap A) .  $
\end{itemize}
\noindent Hence
\begin{itemize}
\item[(vii)] $ E \cap A = B$,
\end{itemize}
since otherwise
the definition of a minimal prealgebraic extension  and (vi) would imply\\
$ 0 \le \delta(A/V) \le\delta(A/ E ) \le \delta(A/ E \cap A)  <0,  $\\
a contradiction. Then (vi) and (vii) imply 
\begin{itemize}
\item[(viii)] $0 \le \delta(A/V) \le \delta(A/E) \le \delta(A/B) = 0$
\end{itemize}
It follows by Corollary \ref{char}:

\begin{itemize}
\item[(ix)] $\langle E A \rangle = E \otimes_B A$

\end{itemize} 
By Lemma \ref{6.3}(2) we get 
\begin{itemize}
\item [(x) ] If $0 < \mid X_V \mid$, then $ \delta(A/V) < \delta(A/E).$
\end{itemize}

\noindent By (x) and (viii) it follows that $X_V$ is empty. Hence we can apply Lemma \ref{6.3}(5) and get:
\begin{itemize}
\item[(xi)]  $\langle V A \rangle = V \otimes_E A$.
\end{itemize}

\noindent By {\bf Trans} (ix) and (xi) imply 
\begin{itemize}
\item[(xii)] $\langle V A \rangle = V \otimes_B A$.
\end{itemize}
\end{proof}

\begin{lemma}\label{10.5}
Assume $\phi \in 
\X^{home}$, and $\C \models \phi(A/B)$. Furthermore $B \subseteq V \le \C$ and $V \in \h_3$. Then

\begin{enumerate}
\item Either $A \subseteq V$ or $\langle A V \rangle = V \otimes_B A$ and 
$\langle A V \rangle \le \C$ is a minimal prealgebraic extension of $V$.

\item In both cases $A \subseteq cl(V)$.

\end{enumerate}
\end{lemma}

\begin{proof}
The assertions  (1) and (2) are direct consequences of Lemma \ref{10.4}. A minimal prealgebraic 
extension of a strong substructure is strong.

\end{proof}

\begin{cor}\label{10.6}
For $B \subseteq \C$ and $\phi \in \X^{home}$ with $\C \models \exists \bar{x}\phi(\bar{x},B)$ we have
$\phi(\bar{x},B)$ is strongly minimal. 
\end{cor}

\begin{proof} 
Assume $\C \models \phi(A/B)$.
The pair $B \le A $ is in $\K_3$
and there is no $x \in A \setminus B$ with $[e,x] =d$, where $e, d \in B$, since the extension is prealgebraic minimal strong.
Let $B \subseteq C \le M \prec \C$. By $\C \models \phi(A/B)$ $B$ is sufficiently strong in $C$ for the application of Theorem \ref{amalk3}. Hence $C \otimes_B A'$ exist as a substructure of $\C$, where $tp(A'/B) = tp(A/B)$. By saturation this remains true, if we replace $C$ by a model $M$:
\[ \langle M A'  \rangle = M \otimes_B A' .\]
Hence there is a realisation of $\phi(x, B)$ generic over $M$. By Corollary \ref{10.5} every 
solution $A$ that is not in $M$ has this form: 
\[ \langle M A\rangle = M \otimes_B A .\]
Hence $\phi(\bar{x}, B)$ is strongly minimal. 
\end{proof}

\begin{lemma}\label{10.7}
Let $\phi(\bar{x},\bar{y}) \in \X^{home}$ and $\C \models \exists \bar{x}\phi(\bar{x}/B)$. Then $B$ is the canonical paramter up to some automorphisms of $B$.
\end{lemma}

\begin{proof}
Let $M \preceq \C$ be saturated, such that $B \subseteq M$. By Lemma \ref{10.6} there  is a solution $\bar{a}$ of $\phi(\bar{x},B)$ generic over $M$. We use  $A = \langle B \bar{a} \rangle$. Then $A \cap M = B $\\
Let $f$ be an automorphism of $\C$ that fixes $B$ pointweise and $M$ setwise. Then $tp(\bar{a}/M) = tp(f(\bar{a})/M)$ by Lemma \ref{10.5} and Corollary \ref{10.6}. \\
Now we consider an automorphism $f$ of $\C$, such that $f$ fixes $M$ setwise and $f(tp(\bar{a}/M)) = tp(\bar{a}/M)$.
W.l.o.g. $f$ fixes $\bar{a}$ pointwise. By Lemma \ref{10.5}
\[ \langle M \bar{a} \rangle = M \otimes_B \langle B \bar{a} \rangle = M \otimes_{f(B)} \langle f(B) \bar{a} \rangle.\]
We assume $B \not= f(B)$ and show a contradiction.   By Lemma \ref{10.2} we get
\[ \langle M \bar{a} \rangle = M \otimes_{B \cap f(B)} \langle (B \cap f(B)) \bar{a} \rangle \]
and $\langle (B \cap f(B)) \bar{a} \rangle$ is again $\le^m$-strong in $\C$ and it is a minimal
prealgebraic extension over a smaller substructure. See Lemma \ref{6.3}.

\end{proof}

\begin{definition}\label{10.8}
Assume $\bar{b} \subseteq  V $ for $V \in \h_3$, $\phi(\bar{x},\bar{y}) \in \X^{home}$, and $\bar{a}$ is a solution of $\phi(\bar{x},\bar{b})$. Then $\bar{a} $ is weakly generic over $V$, if $\langle \bar{a} \bar{b} \rangle \cap V = \langle \bar{b} \rangle $  and $\delta(\bar{a}/V) = 0$. (short  $\ind^w$-generic).
\end{definition}

\noindent Lemma \ref{10.4} implies:

\begin{cor}\label{10.9}
Assume $\bar{b} \subseteq  V $ for $V \in \h_3$, $\phi(\bar{x},\bar{y}) \in \X^{home}$, and $\bar{a}$ is a weakly generic solution of $\phi(\bar{x},\bar{b})$ over $V$. Then  $\langle V \bar{a} \rangle = V \otimes_{\bar{b}} \langle \bar{b} \bar{a} \rangle$. $\bar{a}$ is a prealgebraic minimal strong extension of $V$.
\end{cor}

\noindent By   Lemma \ref{10.7} we obtain 
\begin{cor}\label{10.10}
 Let $\phi(\bar{x},\bar{y}) $ be a formula in  $\X^{home}$. If $\phi(\bar{x},\bar{b}) $ and $\phi(\bar{x},\bar{b'}) $ have common generic solutions, then $\langle \bar{b'} \rangle = \langle \bar{b} \rangle $ and there is an isomorphism of these two Lie - algebras, that can be extended by the identity of the common generic extension.
\end{cor}

\noindent Hence we can define:

\begin{definition}
For every  $\phi(\bar{x},\bar{y}) \in \X^{home}$ there is a unique $\psi(\bar{x},y) \in L^{eq}$, such that for every  $\bar{b}$ with $\C \models \exists \bar{x} \phi(\bar{x},\bar{b}) $ there is some $b \in \C^{eq}$, such that $\phi(\bar{x}, \bar{b})$ and $\psi(\bar{x},b)$ have the same solutions
and $b$ is the canonical parameter of $\psi(\bar{x},y)$. Let $\X$ be the set of these $\psi$.
\end{definition}


\section{$T_3$ satisfies the conditions for the collapse}
\noindent In \cite{Bau09} there are conditions P(I) - P(VII) formulated  for a theory $T$ that provide the existence of an infinite substructure of the monster model of $T$ with  an $\aleph_1$-categorical theory, such that $cl$ becomes  the algebraic closure in the small structure. In that paper $R(\C)$ is a vector space and we work essentially with finite subspaces of it. Here $R(\C)$ is $\C_1 \cup \C_2$ as defined above. We consider 
pairs of subspaces $V_1 \subseteq \C_1$ and $V_2 \subseteq \C_2$ . Often we work with the  $L$ - structures 
$\langle V_1 V_2 \rangle $ in $ \h_3$. The clousure $cl$ can be considered as  a relation beween such $L$ -structures in $\h_3$. Elements coresspond to structures of o-dim 1. That means they 
correspond to the elements of $\C_1 \cup \C_2$.\\
Now   we transform the properties P(I) - P(VII)  from \cite{Bau09} into conditions C(I) - C(VII) and
show that they are true in models of our theory $T_3$.  Then we can follow the proofs in \cite{Bau09}. 
Some modifications are necessary.
We obtain the desired Lie algebra, that  is a counterexample to Zil'ber's conjecture.

\noindent Let $M$ be a model of $T_3$. We work with 
$R(M)$ , the  pregeometry $cl$ on $R(M)$, the  set $\X$ of formulas, 
and $\ind^w$ - genericity, as we have defined  in this paper.  Let $\C$ be again  the monster model of $T_3$.

\begin{enumerate}
\item[C(1)] The models $M$ of $T_3 $ are  graded 
3-nilpotent Lie algebras over 
$\F_q$, where $\F_q$ is a finite field. The graduation is given by $U_1, U_2, U_3$.\\
We work with  $R(x) = U_1(x) \cup U_2(x)$ in $M$. Then $M = \langle R(M) \rangle$.
\end{enumerate} 

\noindent C(1) is clear
\begin{enumerate}
\item[C(2)] There is a pregeometry on $R(M)$, given by $a \in cl(H)$ for  $a \in R(M)$ and $H \in \h_3$ and a notion " $H$ is a strong 
substructure of $M$", short $H \le M$. Both notions are invariant under automorphisms of $\C$. $\langle 0 \rangle \le M$. For every  finite $H$ there exists an finite algebraic extension that is strong in $M$. Algebraic extensions of strong subspaces in $\h_3$ are strong. If $M$ and $N$ are models of $T_3$, $H \subseteq R(M)$, $K \subseteq R(N)$, $tp^M(H) = tp^N(K)$, and $a \in M_i$ and $b \in N_i$ for $i = 1$ or $i = 2$ are geometrically independent of $H$ and $K$ respectively, then $tp^M(a,H) = tp^N(b,K)$. If in this case $H \le M$, then $\langle Ha \rangle \le M$. The geometrical dimension $d(\C)$ of $R(\C)$ is infinite.
\end{enumerate}

\noindent C(2) is proved in this paper, especially in section 9. The next condition is proved in section 10.
$\X$ is  defined there.

\begin{enumerate}
\item[C(3)] There is a set ${\X}$ of formulas $\psi(\bar{x},y)$ in $L^{eq}$ such that $\psi(\bar{x},b)$  is either empty or strongly minimal. \\
If $\psi(\bar{a},b)$ is true, then $b$ is the canonical parameter. It is the canonical base algebra $B$ modulo some 
automorphisms.
$\bar{a} = \bar{a}_1 \bar{a}_2$ with $\bar{a}_1 \subseteq U_1$ and
$\bar{a}_2 \subseteq U_2$ defines an o-system over $B$.
If $\psi(\bar{x},b)$ and $\psi(\bar{x},b')$ have common generic solutions, then  $b = b'$. \\
Length$(\bar{x})=n_{\psi} \ge 2$.
If $b$ is in ${\rm dcl}^{eq}(U)$ and $M\models \psi(\bar{a},b)$, then $\bar{a}\in cl(U)$.
If furthermore $U\le M$, then either $\bar{a}\subseteq U$ or $\bar{a}$ is a generic solution over $U$. In the generic case $\langle U\bar{a}\rangle \le M$.
${\X}$ is closed under affine o-transformations. See below. 
\end{enumerate}
We want to explane, what it means that ${\X}$ is closed under affine o-transformations:\\
We consider $\psi(\bar{x},y)$ in $\X$.
Let $\phi(\bar{x},\bar{y})$ be a corresponding formula  in $\X^{home}$. By Definition \ref{10.3}
and Lemma \ref{10.4}
we have that $\C \models \phi(\bar{a},\bar{b})$, $\bar{a} \notin U$, $\bar{b} \in U,$
and $\delta( \bar{a} / U) \ge 0$ implies  
\[ \langle U \bar{a} \rangle = U \otimes_{\langle \bar{b} \rangle} \langle \bar{b} \bar{a} \rangle \]
If  $\bar{c} \in \langle U \bar{a} \rangle \setminus U$ is an o-system over $U$ with 
\[ \langle U \bar{a} \rangle = \langle U \bar{c} \rangle \] 
then $\bar{c_1} = H_1( \bar{a_1} ) + \bar{m_1}$, where $H_1$ is a vector space 
automorphism of $\langle \bar{a_1} \rangle$ and $\bar{m_1}$ is a tuple in $U_1$.\\
Furthermore $\bar{c_2} = H_2(\bar{ a_2}) + H_d(\bar{a_1}) + \bar{m_2}$, where $H_2$ is
a vector space automorphism of $\langle \bar{a_2} \rangle$, 
$H_d(\bar{a_1}) = (\ldots, \sum_j [d^{i,j}, a_1^j], \ldots) $, 
where $d^{i,j}$ is in $U_1$, and $\bar{m_2}$ is a tuple in $U_2$.
The matrix $H_d$ uses the Lie-multiplication.\\
The tuple $H_1( \bar{a_1} ) , H_2(\bar{ a_2}) + H_d(\bar{a_1})$ we denote by $H(\bar{a})$.

\begin{definition}\label{o-trans}
We call $H(\bar{a})$ an o-transformation over $B$, if all $d^{i,j} \in B$. Furthermore 
$H(\bar{a}) + \bar{m_1 m_2}$ is an affine o-transformation over $B$.
\end{definition} 

\noindent $H(\bar{a}) + \bar{m_1 m_2}$ gives us  a new generating o-system of $\langle U \bar{a} \rangle$ over $U$. As above we get a 
a formula in $\X^{home}$ for $\phi(H(\bar{x}) + \bar{m},\bar{b})$. An o-transformation $H(\bar{x})$
is an homomorphism of the the underlying vector space of $\C^{n_\alpha}$.  \\
It is easily seen:
\begin{lemma}\label{11.2}
Assume that $U \le \C$, $\langle U \bar{a} \rangle = \langle U \bar{e} \rangle$ 
are a minimal prealgebraic extension of $U$ given by 
$\C \models \phi(\bar{a}, \bar{b})$  and   $\C \models \psi(\bar{e}, \bar{c})$ 
with $\phi, \psi \in \X^{home}$ and $\bar{b}$ and $\bar{c}$ are in $B \subseteq U$.
Then the  generating o-system $\bar{e}$ of $\langle U \bar{a} \rangle$ over $U$ 
is given by an affine 
o-transformation  of $\bar{a}$ over $B$.
\end{lemma}

\begin{proof}
By the considerations above we have $\bar{e} = H(\bar{a}) + \bar{m}$, where $H$ is an
o-transformation and $\bar{m} \in U$. By Lemma \ref{10.4}
\[ \langle U \bar{e} \rangle = U \otimes_B \langle B \bar{e} \rangle ,\]
and \[ \langle U \bar{a} \rangle = U \otimes_B \langle B \bar{a} \rangle. \]
Since every element of $\bar{e}_2$ is involved in some relations,  
these relations have to split.
We get that $H$ is an o-transformation over $B$.
\end{proof}

\noindent Also the next condition we have proved:

\begin{enumerate}
\item[C(4)] Let  $H \le M$, $K \le M$. If $\langle H \rangle \cong \langle K \rangle$, then $tp(H) = tp(K)$.
If $H \le K \le M$, $H \subseteq K \subseteq cl(H)$, then there is a chain $H = K_0 \le  K_1 \le \ldots \le K_n$, where $K_n  = K$, $K_i  \le M$, and $K_{i+1} \subseteq acl(K_i)$ or $K_{i+1}$ is obtained from $K_i$ by adding a solution of some $\psi(\bar{x},b)$ in $\X$ generic over $K_i$, where $b \subseteq dcl^{eq}(K_i)$.
\end{enumerate}

\begin{enumerate}
\item[C(5a)] Let $\psi(\bar{x},y)\in \X$, $V$ a subspace in $\h_3$, and $b\in dcl^{eq}(V)$. Then the $\ind$-generic type of $\psi(\bar{x},b)$ over $V$ is $\ind^w$-generic over $V$ and all  $\ind^w$-generics of $\psi(\bar{x},b)$ over $V$ have the same isomorphism type over $V$. They are $\ind^w$-generic over every $U\subseteq V$ with $b\in dcl^{ eq}(U)$. 

\item[C(5b)] If $\psi(\bar{x},y)\in  \X$, $U\le M$, $b \in  dcl^{eq}(H)$, and   $\bar{e}_0,\bar{e}_1,\ldots$  are solutions of $\psi(\bar{x},b)$  with $\bar{e}_i\not\subseteq \langle 
U, H,\bar{e}_0,\ldots,\bar{e}_{i-1}) \rangle$,
then there  are at most $\delta(H/U)$ 
many $i$ such that $e_i$ is not $\ind^w$-generic over $\langle U,H,\bar{e}_0,\ldots,\bar{e}_{i-1}\rangle$.
\end{enumerate}

\noindent The first statements of C(5) follow from \ref{10.4} and \ref{10.9}. To prove the last assertion,
we consider 
\[0 \le \delta( \langle  H, \bar{e_0}, \ldots, \bar{e_{i}} \rangle/ U ) = \delta(H/U) + \sum_{j \le i} \delta^j,\] 

\noindent where $\delta^j = \delta(\bar{e_j}/\langle U, H, \bar{e_0}, \ldots, \bar{e_{j-1}} \rangle)$. \\
If $\delta^j \ge 0$, then   $e_j$ is  $\ind^w$-generic over $\langle U,H,\bar{e}_0,\ldots,\bar{e}_{j-1}\rangle$ by Lemma \ref{10.4}. This proves the assertion. \\

\begin{definition}\label{11.3} Let $X$ be a definable subset  of $\C^n$ of  Morley - degree 1. 
We consider $\C^n$ as a group with respect to $+$.\\
$X$ is called  a group set (respectively  torsor set) if its generic type  is the generic type of a definable subgroup of $\C^n$ (respectively coset of a definable subgroup). $X$ is groupless if it is not a torsor set.
\end{definition}

\noindent Let $D$ be a subalgebra of $\C$. Furthermore $\phi(\bar{x},\bar{e})$ is in $\X^{home}$ with
$\bar{e} \in D$ and there are at least 2 solutions in $D$, that form an o-system over $\bar{e}$. We 
assume, that $D^* \cong F(X)/I$, where $X$ is a generating o-system of $D^*$
that contains the o-system $\bar{e}_1 \bar{e}_2$ and $F(X)$ is the 
free 2-nilpotent graded Lie algebra over $X$. The isomorphism-type of the $U_1 \oplus U_2$-part of the solution set 
of $\phi$  is given by linear combinations $\Theta$ of monomials $[x,y]$ where $x$ is from
$\bar{x}_1$ and $y$ from $\bar{e}_1$ or $\bar{x}_1$ and of one element $g$ from 
$\langle \bar{e} \rangle_2$. We call $\Theta$ a relation of level 2.\\
Let $\bar{a}$ be a solution.
Since $g$ is also a term over the second solution, we can assume 
w.l.o.g. that $\Theta(\bar{a},\bar{e}) \in I$. \\
Then we consider $D^- = \F(D^*)/J$. Similary as above we define relations $\Delta$ of level 3.
Again it contains one summand  $h$ from $\langle \bar{e} \rangle_3$ and we can consider it 
as a term over the second solution. Hence w.l.o.g. $\Delta(\bar{a},\bar{e}) \in J$.

\begin{cor} \label{i+j}
We are in the situation above. A solution of $\phi(\bar{x},\bar{e})$ in $D$ is given by relations
of levels 2 and 3 that are in $I$ or $J$ respectively.
\end{cor}

\begin{enumerate}
\item[C(6)] Assume $C\supseteq B\subseteq A$ are strong subspaces of $\C$ in $\h_3^{fin}$   with $C \cap A = B$ and both minimal strong extensions of $B$ given by generic solutions over $B$ of formulas in $\X$. Furthermore there are $e\in  dcl^{ eq}(E)$, $E \in \h_3$, $E\subseteq \langle A C \rangle$, and 
at least two solutions in $\langle C A \rangle$ of some $\psi(\bar{x},e)$ in ${\X}$.
Let $\bar{d}\in  \langle A C \rangle$ be  a solution  of  $\psi(\bar{x},e)$ in ${\X}$ $\ind^w$-generic over $\langle C E \rangle$ and over $\langle A E \rangle$. \\
Then $\psi(\bar{x},e)$ defines a torsor set. If it defines a group set, then  $e$ is in ${\rm dcl}^{\rm eq}(B)$.
\end{enumerate}

\noindent  We can assume, that $C$ and $A$ are  defined  by formulas in $\X^{home}$ with canonical base 
algebras in $B$. 

\noindent We prove C(6). By assumption we have $\delta(C) = \delta(B) = \delta(A)$ and $B \le \C$ by Corollary \ref{rulesstrong3}. By Lemma \ref{subdelta3} we get
\[ \delta(\langle C A \rangle) \le \delta(C) + \delta(A) - \delta(B) = \delta(B). \]
Since $ B \le \C$ it follows $\delta(\langle C A \rangle) = \delta(B)$ and $\langle C A \rangle \le \C$.
Since $\delta(A/C) = \delta(A/B)$, we have by Lemma \ref{useful} 
\[ D =\langle C A \rangle = C \otimes_B A .\] 
$D$ is in $\h_3^{fin}$.
Assume $(D^*) = F(X)/I$ and $D^- = \F(D^*)/N$, where $X$ is a generating o-system of $D^*$,
$F(X)$ is the free 2-nilpotent Lie algebra over  $X$. Later we will modify $X$.
Then $I \subseteq F(C_1)_2 + F(A_1)_2$. Since $C \le \C$ and $A \le \C$ we can suppose w.l.o.g. that $\F(C^*) \subseteq \F(D^*)$, 
$\F(A^*) \subseteq \F(D^*)$, and $\F(B^*) = \F(C^*) \cap \F(A^*) \subseteq \F(D^*)$. \\
We obtain 
$N \subseteq \F(C^*)_3 + \F(A^*)_3$. \\

\noindent Let $\phi(\bar{x},\bar{e}) \in \X^{home}$ be the corresponding formula for $\psi(\bar{x},e)$, , where $e$ is $\bar{e}$ modulo some automorphisms and $\langle \bar{e} \rangle \subseteq E$.\\ 
We can assume, that $\bar{e} = \bar{e_1} \bar{e_2} \bar{e_3}$ 
is a generating o-system for 
$\langle \bar{e} \rangle$.\\
We assume  that the solution $\bar{d}$ of $\psi(\bar{x},y)$ consists of $d^{i}_1 \in U_1$ for $1 \le i \le m $ and 
$d^{i}_2 \in U_2$ for $m < i \le n$. \\ 
Let $X^B = X_1^B X_2^B X_3^B$ with $X_i^B \subseteq B_i$ be a vector space basis of $B$. \\
We extend $X^B$ by $X^C$ to get a vector space basis for  $C$ and \\
by $X^{A}$ to obtain a vector space basis for $A$. $X^C$ and $X^{A}$ should be graded as $X^B$. \\
We can extend $X^B X^C X^{A}$ to obtain a vector space basis of $C \otimes_B A$ as in Corollary \ref{c=3,a}.   
Corollary \ref{i+j} implies, that $ \bar{d_i}  \subseteq C_i + A_i$ for $i = 1,2$, since these elements  have to be involved in relations of $I$ or $N$. \\
The same is true for elements in $ \bar{e}$ . Here elements in $C_3 + A_3$ are possible. \\
Hence
the elements $d_1^j \in \C_1$ of $\bar{d}$ have  to be in $C_1 + A_1$ and the elements $d_2^j$ have to be in
$C_2 + A_2$. Since $\bar{d}$ is $\ind^w$ - generic over $\langle C E \rangle$ and $\langle A E \rangle$, we get 
$d_1^j = c_1^j + a_1^j$, where the $c_1^j \in C_1$ are linearly independent over $A_1 + E_1$ and the $a_1^j \in A_1$ are linearly independent over $C_1 + E_1$  and \\
$d_2^j = c_2^j + a_2^j$, where the $c_2^j \in C_2$ are linearly independent over $A_2 + E_2$ and the $a_2^j \in A_2$ are linearly independent over $C_2 + E_2$ .\\ 
Since $\bar{d}$ is a weakly generic solution over $\langle C E \rangle $ and $\langle A E \rangle$ 
Corollary \ref{10.9} implies
\[ \langle C E \bar{d} \rangle = \langle C E \rangle \otimes_{\langle \bar{e} \rangle } \langle \bar{e} \bar{d} \rangle \]
and 
\[ \langle A E \bar{d} \rangle = \langle A E \rangle \otimes_{\langle \bar{e} \rangle } \langle \bar{e} \bar{d} \rangle .\]
Note that mixed monomials $[x,y]$ with $x \in X_1^C $ and $y \in X_1^{A}$ are linearly independent over $C_2 + A_2$ and cannot  be involved in any relation of $I$. \\ 
We choose the vector space bases above with the following additional assumptions for $i = 1,2,3$:\\
$X^B$ contains a vector space bases $V_i^B$ with elements $e_i^{B,j}$ for $\langle \bar{e_i} \rangle \cap B$, where $i = 1,2,3$.\\
$ X^C$ contains  a vector space bases $V_i^C$ with elemets  $e_i^{C,j}$ for 
$\langle \bar{e_i} \rangle \cap C$ modulo $B_i$.\\
$ X^{A}$ contains  a vector space bases $V_i^{A}$ with elements $e_i^{A,j}$ for 
$\langle \bar{e_i} \rangle \cap A$ modulo $ B_i$.\\
Furthermore there are $f_i^j \in X_i^C$ and $g_i^j \in X_i^{A}$, such that $V_i^B V_i^C V_i^{A}$ together with the basis $V_i^+$ of elements
$ f_i^j + g_i^j$ form a vector space basis of $\langle \bar{e} \rangle_i \cap (C_i + A_i)$.\\
The $c_i^j$ and $a_i^j$ are linearly independent over $V_i^{B} V_i^{A} V_i^C \ldots f_i^j \ldots g_i^j \ldots$. 
We assume that they are part of our basis $X^B X^C X^{A}$. 
An element $\Theta $ in $I \cap  \langle \bar{e_1} \bar{d_1} \rangle^{F(D_1)}_2$ has the form:

\[ \Theta = \sum_{1 \le j \le m}   [c_1^j + a_1^j, h_1^j] + \Theta_c + \Theta_a ,\]

\noindent where $h_1^j \in \langle \bar{e_1} \rangle \cap B_1$ , $\Theta_c \in F(C_1)_2$ , and $\Theta_a \in F(A_1)_2$ , and $\Theta_c + \Theta_a \in \langle \bar{e} \rangle_2$. Here we use  the description of a vector space basis of 
$C \otimes_B A$ and $I \subseteq  \langle C_1 \rangle_2^{F(D_1)} + \langle A_1 \rangle_2^{F(D_1)}$. 
We call the $h_1^j$ the coefficients of $\Theta$. \\ 
Corollary \ref{c=3,a} 
implies that the description of a vector space basis for the representation of $\Theta$ as above  work also  for $ \langle C E \bar{d} \rangle = \langle C E \rangle \otimes_{\langle \bar{e} \rangle } \langle \bar{e} \bar{d} \rangle $
and 
$ \langle A E \bar{d} \rangle = \langle A E \rangle \otimes_{\langle \bar{e} \rangle } \langle \bar{e} \bar{d} \rangle $.
We see that the relations of level 2 define a torsor set. In the case of a group set there are no $\Theta_c + \Theta_a $ and we have only the coefficients in $B$.

\vspace{1cm}
\noindent Now we turn to level 3:  $\F(D^*)_3$.
We need an order of a subspace of $D_1$. We define $V_1^B < V_1^C < V_1^{A} < V_1^+  < c_1^j < a_1^j$. Inside $V_1^B$, $V_1^C$, $V_1^{A}$, and $V_1^+$ the order is given  by the  indizies. Furthermore 
$c_1^j < a_1^j < c_1^k < a_1^k$, if $j < k$. We consider 
\[ \F(D^*) = \F(C^*) \otimes_{\F(B^*)} \F(A^*) .\]
$ \langle \bar{e} \bar{d} \rangle_2^{\F(D^*)} \cap (\F(C^*) + \F(A^*))_2$ has the following vector space basis:\\
$V_2 = V_2^B V_2^C V_2^{A} V_2^+ V_2^0 V_2^1$, where $V_2^1 = \{  d_2^j : m < j \le n\} $ and $V_2^0$ is a subset of $\{ [d_1^j, e_1^{B,k}] : 1 \le j \le m, k \}$. By assumption the $[c_1^j, e_1^{B,k}]$ as summands of $[d_1^j, e_1^{B,k}] \in V_2^0$ are linearly independent over
$A_2 + E_2$. The same is true for the $[a_1^j, e_1^{B,k}]$ over $C_2 + E_2$. \\
In the next application of Corollary \ref{c=3,a} we  use $[[d_1^j, d_1^k], x] = [[d_1^j,x] , d_1^k] - [[d_1^k,x], d_1^j]$.\\
Then the following set $W$ is a vector space basis of $\langle \bar{e} \bar{d} \rangle_3^{\F(D^*)}$
over $\F(C^*)_3 + \F(A^*)_3$:
\[ W = W_1 \cup W_2 \cup W_3 \cup W_4 \cup W_5 \cup W_6,\] 
where 
\begin{enumerate}
\item $W_1 = \{ [d_2^j, x] : m < j \le n , x \in V_1^C \cup V_1^{A}  \cup V_1^+ \cup \{ d_1^k : 1 \le k \le m \}, \} $
\item $W_2 = \{ [[d_1^{i}, d_1^j ] , d_1^k ] : j < i, j \le k \}$,
\item $W_3 = \{[y, d_1^k] : 1 \le k \le m, y \in V_2 \setminus (V_2^B \cup V_2^1) \}, $
\item $W_4 = \{[y,d_1^k] :  1 \le k \le m, y = [x,z] , ((x \in V_1^C \wedge z \in V_1^{A} \cup V_1^+) \vee 
(x \in V_1^{A} \wedge z \in  V_1^+) \vee (x = f_1^{i} + g_1^{i} \wedge z = f_1^j + g_1^j \wedge i < j ),$
\item $W_5 = \{[[d_1^{i}, z], v] : 1 \le i \le m, z < v, z,v \in V_1^C V_1^{A} V_1^+ \} $.
\item $W_6 = \{[[d_1^{i}, z], d^k_1] : 1 \le i  \le k \le m,  z \in V_1^C V_1^{A} V_1^+ \} .$
\end{enumerate}
Finally we consider elements $\Delta \in N \cap \langle \bar{e} \bar{d} \rangle_3^{\F(D^*)}$. They need to have 
the following form:
\[ \Delta = \Delta_d + \Delta_c + \Delta_a \in N \cap \langle \bar{e} \bar{d} \rangle_3^{\F(D^*)}, \]
where $\Delta_c \in \F(C^*)_3$, $\Delta_a \in \F(A^*)_3$, 
$\Delta_c + \Delta_a \in \langle \bar{e} \rangle_3^{\F(D^* )}$,
and $\Delta_d$ is a sum of monomials over
$\bar{e} \bar{d}$ that contain at least one $d_i^j$. By the linear independence of $W$ over  $\F(C^*)_3 + \F(A^*)_3$ we get that 
$\Delta_d$ is a linear combination of monomials $[d_2^{i}, h_1]$, $[d^{i}_1,h_2]$, and $[[d_1^{i},h_1^j],h_1^k]$, where the coefficients $h_1$, $h_1^j$, and $h_1^k$ are all in $B_1 \cap \langle \bar{e} \rangle_1$ and $h_2$ is in $B_2 \cap \langle \bar{e} \rangle_2$.
By the form of the $\Theta$ and $\Delta$ we see that $\psi(\bar{x},e)$ defines a torsor set in $R(\C)^n$.
If it is a group set, then we have only the coefficients in $B$.\\

\noindent \begin{enumerate}
\item[C(7)] For every substructure $X$ of $M$ with $acl(R(X)) \cap R(M) = R(X)$,
$X_1 \not= \langle 0 \rangle$, and $\langle R(X) \rangle = X$, we have 
that $X $ is the set of all $  d = a_1 + a_2 + [c_1, c_2] $, where $a_1, c_1 \in  R(X)_1$ and $a_2, c_2 \in R(X)_2$.
\end{enumerate}

\noindent C(7) is clearly fulfilled. We have shown:

\begin{theorem}\label{11.4}
$T_3$ satiesfies the conditions C(1) - C(7).
\end{theorem}

\noindent The proof of (C6) implies the following

\begin{cor} \label{11.5}
Assume $\phi(\bar{x},\bar{y}) \in \X^{home}$ corresponds to $\psi(\bar{x},y) \in \X$.\\
If $\phi(\bar{x},\bar{y}) \in \X^{home}$ gives us   a torsor set, then isomorphism type of $\bar{x}$ over
$\bar{y}$ is given by equations of the following form:
\[ \sum_{1 \le j \le m} [x_1^j, z^j]  + y^2 = 0\]
where $z^j \in \langle  \bar{y} \rangle_1$  and $y^2 \in \langle \bar{y} \rangle_2 $ or
\[ \sum_{1 \le j \le m} w^j + y^3 = 0\]
where the $w^j$ are monomials of the form $[x_2^k, v_1]$ with $v_1 \in \langle  \bar{y} \rangle_1$ 
or of the form $[x_1^k,v_2]$ with $v_2 \in \langle  \bar{y} \rangle_2$, or of the form
$[[x_1^k,v_1^l],v_1^{i}]$ with $v_1^l, v_1^{i} \in \langle  \bar{y} \rangle_1$ and 
$y^3 \in \langle \bar{y}^3 \rangle$. \\
A group set, given by a formula in $\X$,  has no $y^2$ and $y^3$ in the equations. \\
If the set defined by $\phi(\bar{x},\bar{y}) \in \X^{home}$ is a torsor set, then $\phi(\bar{x},\bar{y}) \in \X^{home}$ defines itself a coset of a group. In case of a group set it defines a group.
\end{cor}

\begin{proof} First we assume that $\phi$ gives a group set.
In the setting of (C6) let $A$ and $C$ be extension of $B$, given by  generic solutions $\bar{a} $
and $\bar{c}$ of $\phi(\bar{x},\bar{b})$ over $B$. By definition $\bar{a} + \bar{c}$ is a solution
of  $\phi(\bar{x},\bar{b})$ weakly generic over $A$ and $C$. Then the proof of (C6) gives the 
assertion. If $\phi$ gives a torsor set, then let $\bar{a}$ be again a generic solution  of $\phi$ and 
$\bar{c}$ a solution of the corresponding group set.
\end{proof}


\section{Codes and Difference Sequences}
\noindent To get our main result we follow the proof in \cite {Bau09}. We replace  the conditions P(I) - P(VII) 
in that paper by C(1) - C(7).  
$R(\C)$ is the  union of two  connected vector spaces now.  
According to this we have two kinds of elements in $R(\C)$. 
Therefore we consider substructures in $\h_3$ instead of  the vector subspaces  in \cite {Bau09}. \\
Our approach in this paper is not strict axiomatic using C(1) - C(7), as in \cite{Bau09}. The next Lemma corresponds to Lemma 3.6 in \cite{Bau09}.

\begin{lemma}\label{add3.6}
Assume $M$ is a model of $T_3$,   $A,U\in \h_3$ with   $A\le M$and $U \le M$.   Then
\begin{enumerate} \item[( i)] $A\cap U\le M$, and 
\item [(ii)] If $A \cap U  \in \h_3$, $d(A/U) = d(A/ A \cap U)$, then $\langle U A \rangle \le M$.
There  is a sequence of minimal strong extensions for $A$
\[B_0 = A \cap U \le B_1 \le \ldots \le B_k = A \le M.\]
For every such sequence for $A$ over $A \cap U$
\[ U \le \langle U B_1 \rangle \le \ldots \le \langle U B_k \rangle  \le M \]
is a sequence of minimal strong extensions for $\langle U A \rangle$ over $U$.

\end{enumerate}
\end{lemma}

\begin{proof} ad(i) For finite $U$ (i) is Lemma \ref{rulesstrong3}. Otherewise there is a finite $U_0 \le M$
with $ A \cap U \subseteq U_0 \subseteq U$. Then $A \cap U = A \cap U_0 \le M$.\\
ad(ii) 
By Lemma \ref{9.7} we get a geometrical
sequence $A\cap U=B_0\le B_1\le \ldots \le
B_k=A$. That means $B_{i+1} $ is minimal strong over $B_i$. All $B_i$ are strong in $M$.
We show by induction on $i$ that $\langle U B_i \rangle \le M$. We have  
$\langle U B_0 \rangle =U\le M$. If $B_{i+1}=\langle B_i,b\rangle$ where $b$
is algebraic over $B_i$, then  $\langle
UB_ib\rangle\le M$ by C(2). In the prealgebraic case the
assertion follows from $\langle U B_i\rangle \le M$ and C(3). In the
transcendental case $B_{i+1}=\langle B_ib\rangle$ with
$b\notin{\rm cl}(B_i)$ we have $b\notin {\rm cl}(\langle B_i U \rangle )$ since
$d(A/U)=d(A/B_0)$. By C(2) $\langle B_{i+1} U \rangle \le M$.
\end{proof}

\begin{definition}\label{12.2} Two definable sets $X$ and $Y$ of Morley degree 1 are equivalent, if ${\rm MR}(X)={\rm MR}(Y)$ and ${\rm MR}(X\Delta Y)<{\rm MR}(X)$. We write $X\sim Y$.
\end{definition}

\begin{definition}\label{12.3}
We consider definable sets $X \subseteq \C^n$ where the  elements $\bar{a} = \bar{a_1} \bar{a_2}$  are o-systems with  $U_1(a_1^j)$ and  $U_2(a_2^j)$. If $H(\bar{x})$ is an o-transformation, then we use
$H(X) =\{H(\bar{a})  : \bar{a} \in X \}$.
 $X + \bar{m} = \{ \bar{a_1} + \bar{m_1} 
,\bar{a_2} + \bar{m_2} : 
\bar{a} \in X \}$ for $\bar{m} = \bar{m_1} \bar{m_2}$ with $U_1(\bar{m_1})$ and $U_2(\bar{m_2})$.

\end{definition}

\noindent The next 3 lemmas concern countable $\omega$-stable theories. In \ref{add2.2} 2) we consider 
only $T_3$. 
For background and proofs    see also \cite{Bau09}:

\begin{lemma}\label{add2.1}(Version of Martin Ziegler \cite{Zi06})\\
Assume that $T$ is a countable $\omega$-stable theory of expansions of abelian groups. 
Let $M$ be a model of $T$  and $a$, $b$, $c$ be elements of $M$ with $a+b+c=0$ and pairwise independent over some set $B$. Then we have:
\begin{enumerate}
\item[{\rm 1)}] The strong types of the elements $a$, $b$, $c$ over $B$ have the same stabilizer $U$ and $U$ is connected.
\item[{\rm 2)}] $a$, $b$, and $c$ are generic elements of ${\rm acl}(B)$-definable cosets of $U$.
\item[{\rm 3)}] It follows that $a$, $b$, and $c$ have the same Morley rank over $B$ namely ${\rm MR}(U)$. $U$ is definable over ${\rm acl}(B)$.
\end{enumerate}
\end{lemma}

\begin{lemma}\label{add2.2}
Let $X$, $Y$ be definable sets of Morley degree $1$.
\begin{enumerate}
\item[{\rm 1)}] If $X\sim Y$, $X,Y\subseteq \C^n$, and $X$ is a group set $($resp. torsor set$)$, then $Y$ is a group set $($resp. torsor set$)$.
\item[{\rm 2)}] We work in $T_3$. We assume $X$ is defined by some $\phi(\bar{x},\bar{b})$ in 
$\X^{home}$ in $\C^n$ and $H(X) $ is a  o-transformation over $\bar{d}$.
If $X$ is a group set then $H(X) $ is a group set. If $X$ is a torsor set then $H(X) + \bar{m}$ is a torsor
set.

\end{enumerate}
\end{lemma}

\begin{proof}
ad(2) If $X$ is a group set, then $H(X)$ is a group set by Corollary \ref{11.5}. For $H(X)$ there is a 
formula in $\X^{home}$ over $\langle \bar{b} \bar{d} \rangle$.
\end{proof}


\begin{lemma}\label{add2.3}
Let $\varphi(\bar{x},\bar{y})$ be a formula such that $\C\vDash\exists\bar{x}\,\varphi(\bar{x},\bar{b})$ implies that $\varphi(\C,\bar{b})$ is a strongly minimal subset of $\C^n$.
Then $\{\bar{b}:\varphi(\C,\bar{b})\mbox{ is a group set}\}$ is definable. Similarly for torsor sets.
\end{lemma}

{\em Proof\/.} We consider the group case. The following statements are equivalent:
\begin{enumerate}
\item[i)] $\varphi(\C,\bar{b})$ is a group set.
\item[ii)] There exist two generic $\bar{b}$-independent realizations $\bar{a}_1$ and $\bar{a}_2$ of $\varphi(\bar{x},\bar{b})$ such that $\C\vDash\varphi(\bar{a}_1+\bar{a}_2,\bar{b})$.
\item[iii)] $\C\vDash\exists^\infty\,\bar{x}_1\in \C^n\;\exists^\infty\,\bar{x}_2\in\C^n(\varphi(\bar{x}_1,\bar{b})\wedge
\varphi(\bar{x}_2,\bar{b})\wedge\varphi(\bar{x}_1+\bar{x}_2,\bar{b}))$.
\end{enumerate}
The equivalence of i) and ii) is shown in
\cite{BMPZ3}. iii) is first order since
$\varphi(\bar{x},\bar{b})$ is strongly minimal. It is clearly
equivalent with ii).\hfill$\Box$
\bigskip

\noindent Note that $X$ is a torsor set if for some (every) $x\in X$ the set $X-x$ is a group set.
\bigskip

\begin{definition} Given a group set $X$ defined by a formula in $\X^{home}$ over $\bar{b}$, 
its invariant group is the set ${\rm Inv}(X)$ of o-transformations $H$ over 
$\bar{b}$ with $H(X)\sim X$.
\end{definition}
\bigskip

\noindent For strongly minimal $\varphi(\bar{x},\bar{b})$ \, "$H\in{\rm Inv}(\varphi(\bar {x},\bar{b}))$"\ is an elementary property of $\bar{b}$. As in \cite{BMPZ3} Lemma \ref{add2.1} implies


\begin{lemma}\label{add2.4}
Let $X \subseteq \C^n$ be defined by a formula $\phi(\bar{x}, \bar{b}) \in \X^{home}$, where 
$\bar{b} \subseteq B$. $H$ is an o-transformation over $\bar{b}$.\\
Let $\bar{e}_0$ and $\bar{e}_1$ be two generic $B$-independent elements in $X$. If $\bar{e}_0-H\bar{e}_1\ind\limits_B \bar{e}_0$ , then $X$ is a torsor set. Moreover, if $X$ is a group set, then $H$ is in ${\rm Inv}(X)$.
\end{lemma}

\begin{proof}
Note that $H(\bar{e_j})$ and $\bar{e_j}$ are interdefinable over $B$. We have 
\[ MR(H(\bar{e_1})/B, \bar{e_0} - H( \bar{e_1})) = MR(\bar{e_0}/B, \bar{e_0} - H( \bar{e_1}))
= MR(\bar{e_0}/B) = MR(\bar{e_1}/B).   \]
Hence by Lemma \ref{add2.1} X is a torsor set. By Corollary \ref{11.5} $\phi(\bar{x}, \bar{b})$ defines 
the coset. In the  group case $H(\bar{e_1})$ is a generic element of a coset of $X$. By Lemma \ref{add2.2}
$H(X)$ is a group set, in fact a group. Hence $X = H(X)$. 
\end{proof}

\begin{lemma}\label{addl4.1}
\begin{enumerate}
\item[{\rm a)}] Let $\varphi(\bar{x},y)$ and
$\psi(\bar{x}_0,\ldots,\bar{x}_\mu,y)$ be formulas where
$\bar{x}$ and $\bar{x}_i$ are in the home sort. Assume that
$\varphi(\bar{x},b)$ is strongly minimal where $b$ is
in $\C^{\rm eq}$. Then we can express that any Morley sequence
$\bar{a}_0,\ldots,\bar{a}_\mu$ of $\varphi(\bar{x},b)$
fulfils $\vDash\psi(\bar{a}_0,\ldots,\bar{a}_\mu,b)$.
\item[{\rm b)}] $X\sim Y$ for strongly minimal sets can be
expressed.
\end{enumerate}
\end{lemma}

\begin{proof} b) follows from a) and
\[
\exists^\infty\,\bar{x}_0\;\exists^\infty\bar{x}_1\ldots\exists^\infty\,\bar{x}_\mu\Big(\bigwedge\limits_{i\le\mu}\varphi(\bar{x}_i,\bar{b})\wedge\psi(\bar{x}_0,\ldots,\bar{x}_\mu,b)\Big)
\]
is the desired formula in a).
\end{proof}

\begin{definition} If $X$ is a strongly minimal subset of $\C^n$ and $X\sim\varphi(\bar{x},b)$ where $b\in\C^{\rm eq}$, then we say that $X$ is encoded by $\varphi(\bar{x},y)$.
\end{definition}

\noindent We define codes similarly as in \cite{BMPZ3} and \cite{Bau09}. It is a modification of E.~Hrushovski's definition \cite{Hr2} to our context. The set $\cc$ of good codes is a modification
of $\X$.
\bigskip

\begin{definition}  $\varphi_\alpha(\bar{x},{y})$ is a code formula or short a code, if it fulfils the following conditions:
\begin{enumerate}
\item[a)] Length$(\bar{x})=n_\alpha\ge 2$, $y$ is an element in $\C^{eq}$ that encodes a substructure 
$\langle \bar{y} \rangle$, and $\bar{x} = \bar{x_1} \bar{x_2}$ is an o-system over $\langle \bar{y} \rangle$, where $\bar{x_1}$
is in $U_1$ and $\bar{x_2}$ is in $ U_2$.
\item[b)] The set $\varphi_\alpha(\bar{x},b)$ is either empty or strongly minimal.
\item[c)] $n_\alpha$ is the o-dimension for all solutions.
\item[d)] $\varphi_\alpha(\bar{x},b)\sim\varphi_\alpha(\bar{x},b')$ implies $b =b'$.
\item[e)] If some non-empty $\varphi_\alpha(\bar{x},b)$ is groupless, then all $\varphi_\alpha(\bar{x},b')$ are.
\item[f)] $\varphi_\alpha(\bar{x_1}+\bar{m_1},
\bar{x_2} + \bar{m_2},b)$ is encoded by $\varphi_\alpha$ for all 
$\bar{m_1} \in \C_1^{\mid \bar{x_1} \mid}$ and all
$\bar{m_2} \in \C_2^{\mid \bar{x_2} \mid}$.
\item[g)] For all o-transformations $H$ over $\langle \bar{d} \rangle$  the set $\varphi_\alpha(H(\bar{x}), 
,b)$ is encoded by $\varphi_\alpha$.
\end{enumerate}
\end{definition}
\medskip

By b) and d) $b$ is the canonical parameter of $\varphi_\alpha(\bar{x},b)$.


\begin{lemma}\label{add4.2}
There is a set ${\cc^0}$ of codes $\varphi_\alpha(\bar{x},y)$ that encodes the strongly minimal sets that are encoded by the formulas in ${\X}$.
\end{lemma}

\begin{proof} The proof is similar as the proof in \cite{Bau09} and that proof comes from \cite{BMPZ3}.
The formulas in
${\X}$ have the properties a)--d). Using Lemma \ref{add2.3} we
can assume w.l.o.g. that the formulas in ${\X}$ satisfy
a)--e). Since ${\X}$ is closed under suitable affine o-transformations given in f) and g) by
compactness there are finitely many $\varphi_1,\ldots,\varphi_r$
in ${\X}$ that encode all possible suitable affine o-transformations of
some given $\varphi(\bar{x},b)$.
Moreover we know that either all or none encode groupless sets by Lemma \ref{add2.2}.\\
Choose a sequence $w_1,\ldots,w_r$ of different definable elements in $T^{\rm eq}$. Define
\begin{eqnarray*}
\theta_i^1(\bar{b})&=&"\mbox{No }\varphi_j\; (j<i)\mbox{ encodes }\varphi_i(\bar{x},b)"\\
\theta_i^2(\bar{b})&=&"\varphi_i(\bar{x},b)\mbox{ is equivalent to some }\varphi(H(\bar{x})+\bar{m},b')"\\
\varphi'_i(\bar{x},y)&=&\varphi_i(\bar{x},y)\wedge\theta_i^1(y)\wedge\theta_i^2(y)\end{eqnarray*}
Finally let $\varphi_\alpha(\bar{x},y_1,y)=\bigvee\limits_{i=1}^r(\varphi'_i(\bar{x},y)\wedge y_1=w_i)$. 
$\varphi_\alpha$ has the properties a)--e). To show f) and g) let $b$, $\bar{m_1},\bar{ m_2}$, $H_1, H_d, H_2$ be given. By construction $\varphi_\alpha(\bar{x},b_1,b)$ is equivalent to some $\varphi(H'_1(\bar{x_1})+\bar{m_1}',  H'_2(\bar{x_2)}+ H'_{d'}(\bar{x_1}) + \bar{m_2}'  ,b')$. Hence
\[
\varphi_\alpha(H_1(\bar{x_1})+\bar{m_1 } ,H_2(\bar{x_2}) + H_d(\bar{x_1}) +\bar{m_2 },b_1, b)\sim\]
\[ \varphi((H'_1H_1)(\bar{x_1})+H'_1(\bar{m_1})+\bar{m_1}', (H'_2H_2)(\bar{x_2}) + H'_2(H_d(\bar{x_1})) +
H'_{d'}(H_1(\bar{x_1})) +
H'_2(\bar{m_2}) +\bar{m_2}',b')
\]
and the right side is encoded by $\varphi_\alpha$ by construction,
since 
\[   H'_2(H_d(\bar{x_1})) +H'_{d'}(H_1(\bar{x_1})) = H_c(\bar{x_1}) \]
for some $H_c$.
\end{proof}


\begin{theorem}\label{add4.3}
There is a set ${\cc}$ of good codes such that for every $\varphi(\bar{x},b)$ in ${\X}$ there is a unique $\varphi_\alpha(\bar{x},c)$ in ${\cc}$ such that $\varphi(\C,b)\sim\varphi_\alpha(\C,c)$.
\end{theorem}

\begin{proof} (modification of the proofs in \cite{Bau09} and \cite{BMPZ3})\\
Let $\alpha_i$ be a list of all codes from $\cc^0$  (see Lemma~\ref{add4.2}). Again
define:
\[
\theta_i(b)="\mbox{No }\varphi_{\alpha_j}\;(j<i) \mbox{ encodes }\varphi_{\alpha_i}(\bar{x},b)"\mbox{ and }\varphi'_{\alpha_i}(\bar{x},y)=\varphi_{\alpha_i}(\bar{x},y)\wedge\theta_i(y).
\]
$\varphi'_{\alpha_i}$ satisfies a)--e) we have to show f) and g). By construction $\varphi'_{\alpha_i}(H(\bar{x})+\bar{m},b)$ is encoded by $\varphi_{\alpha_i}$. We need only to show that no $\varphi_{\alpha_j}$ with $j<i$ encodes it. Suppose that
\[
\varphi_{\alpha_i}(H_1(\bar{x}_1)+\bar{m}_1,
H_2(\bar{x}_2) + H_d(\bar{x}_1) + \bar{m}_2,b)\sim\varphi_{\alpha_j}(\bar{x},b').
\]
Then
\[
\varphi_{\alpha_i}(\bar{x},b)\sim\varphi_{\alpha_j}(H^{-1}_1(\bar{x}_1 - \bar{m}_1),
H^{-1}_2(\bar{x}_2 - H_d(\bar{x}_1) - \bar{m}),b')\sim\varphi_{\alpha_j}(\bar{x},b'')
\]
for some $b''$. This contradicts the definition of $\varphi'_{\alpha}$. Hence 
${\cc^+}=\{\alpha'_i:i<\omega\}$ has the desired properties.
\end{proof}
\bigskip



\begin{cor}\label{add4.4}
In the definition of C(1) - C(7) we can replace ${\X}$ by a set ${\cc}$ of good codes.
\end{cor}

\begin{cor}\label{add4.7}
If $D\le\C$ and $D'$ is a prealgebraic minimal extension of $D$, then there is a 
unique  good code $\alpha$ such that there is some   $b$ in ${\rm dcl}^{\rm eq}(D)$ and a generic solution $\bar{a}$ of $\varphi_\alpha(\bar{x},b)$ that generates $D'$ over $D$. If  $\bar{c}$ is a generating 
o-system of $D'$ over $D$, then there is some $d \in {\rm dcl}^{\rm eq}(D)$ such that 
$\bar{c}$ is a generic solution of $\varphi_\alpha(\bar{x},d)$
\end{cor}

\begin{proof} This follows from the properties of good codes, Lemma \ref{add4.3}, and
Corollary ~\ref{add4.4}. 
\end{proof}

\noindent For each $\alpha\in{\cc}$ we choose a natural number $m_\alpha$ such that
the existence of $m_\alpha$ common solutions of $\varphi_\alpha(\bar{x},\bar{b})$ and $\varphi_\alpha(\bar{x},\bar{b}')$ implies $\varphi_\alpha(\bar{x},\bar{b})\sim\varphi_\alpha(\bar{x},\bar{b}')$. This is possible by the strong minimality of $\varphi_\alpha(\bar{x},\bar{y})$.


\begin{theorem}\label{add4.5}
For each $\alpha\in{\cc}$ and $\lambda\ge m_\alpha$ there is a formula $\psi_\alpha(\bar{x}_0,\ldots,\bar{x}_\lambda)$ with the following properties:
\begin{enumerate}
\item[{\rm a)}] For any initial segment $\{\bar{e}_0,\ldots,\bar{e}_\lambda,\bar{f}\}$  of a Morley sequence of $\varphi_\alpha(\bar{x},b)$
\[
\psi_\alpha(\bar{e}_0-\bar{f},\ldots\bar{e}_\lambda-\bar{f})
\]
holds.
\item[{\rm b)}] For each realization $(\bar{e}_0,\ldots,\bar{e}_\lambda)$ of $\psi_\alpha$ there is a unique $b$ with $\vDash\varphi_\alpha(\bar{e}_i,b)$ for $0\le i\le\lambda$. Moreover $b\in{\rm dcl}^{\rm eq}(\bar{e}_{i_1},\ldots,\bar{e}_{i_{m_\alpha}})$ for any $i_1<\ldots<i_{m_\alpha}$.\\
$($We call $b$ the canonical parameter of the sequence $\bar{e}_0,\ldots,\bar{e}_\lambda)$.
\item[{\rm c)}] $b$ encodes a unique canonical substructure $\langle \bar{b} \rangle$ for a corresponding formula in $\X^{home}$ and each realization of $\psi_\alpha$ forms an o-system over $\langle \bar{b} \rangle$.
\item[{\rm d)}]  If $\vDash\psi_\alpha(\bar{e}_0,\ldots,\bar{e}_\lambda)$, then for $i\in\{0,\ldots,\lambda\}:$
\[
\vDash\psi_\alpha(\bar{e}_0-\bar{e}_i,\ldots,\bar{e}_{i-1}-\bar{e}_i,-\bar{e}_i,\bar{e}_{i+1}-\bar{e}_i,\ldots,\bar{e}_\lambda-\bar{e}_i).
\]
\item[{\rm e)}] Given a realization $(\bar{e}_0,\ldots,\bar{e}_\lambda)$ of $\psi_\alpha$ with canonical parameter $b$ as in {\rm b)}, we have the following:\\
Suppose $\alpha$ is groupless:
\begin{enumerate}
\item[{\rm 1)}] If $\bar{e}_i$ is a generic solution of $\varphi(\bar{x},\bar{b})$, then $\bar{e}_i-H\bar{e}_j\nind\limits_{\bar{b}}\bar{e}_i$ for all o-transformations  $H$ over $\bar{b}$ and $j\ne i$.
\end{enumerate}
Suppose $\alpha$ is a coset code, then:
\begin{enumerate}
\item[{\rm 2)}] $\varphi_\alpha(\bar{x},b)$ is a group-set.
\item[{\rm 3)}] $\psi_\alpha(\bar{e}_0,\ldots,\bar{e}_{i-1},\bar{e}_i-\bar{e}_j,\bar{e}_{i+1},\ldots,\bar{e}_\lambda)$ for $j\ne i$.
\item[{\rm 4)}] $\psi_\alpha(\bar{e}_0,\ldots,\bar{e}_{i-1},H\bar{e}_i,\bar{e}_{i+1},\ldots,\bar{e}_\lambda)$ for all suitable o-transformations  $H\in{\rm Inv}(\varphi_\alpha(\bar{x},b))$ over $\bar{b}$.
\item[{\rm 5)}] Moreover, if $\bar{e}_i$ is generic in $\varphi_\alpha(\bar{x},\bar{b})$, then $\bar{e}_i-H\bar{e}_j\nind\limits_{\bar{b}}\bar{e}_i$ for all $j\ne i$ and 
all suitable o-transformations 
$H$ over $\bar{b}$  not in ${\rm Inv}(\varphi_\alpha(\bar{x},b))$.
\end{enumerate}
\end{enumerate}
\end{theorem}

\begin{proof} (Copy of the corresponding proof in \cite{BMPZ3} but in another theory.)\\
We consider the following partial type
\[
\Sigma(\bar{e}_0,\ldots,\bar{e}_\lambda)=\parbox[t]{8cm}{"There is some $b'$ and some Morley sequence $\bar{a}_0,\ldots,\bar{a}_\lambda,\bar{f}$ of $\varphi_\alpha(\bar{x},b')$ with $\bar{e}_i=\bar{a}_i-\bar{f}$."}
\]
{\em Claim\/.} $\Sigma$ has the properties a) -- e).
\medskip

\noindent {\em Proof of the claim\/.} a) is clear.\\
Given such a realization $\bar{e}_0,\ldots,\bar{e}_\lambda$ of $\Sigma$, 
where  $b'$ and $\bar{a}_0,\ldots,\bar{a}_\lambda,\bar{f}$ 
are given as above. Hence $\{\bar{e}_i\}_{0\le i\le\lambda}$ is a Morley sequence of $\varphi_\alpha(\bar{x}+\bar{f},\bar{b}')$. Then $\varphi_\alpha(\bar{x}+\bar{f},b')\sim\varphi_\alpha(\bar{x},b)$ for some $b$ by f) in the definition of codes. 
Since $b$ is the canonical parameter of the generic type determined by $\varphi_\alpha(\bar{x},b)$, the sequence $\{\bar{e}_i\}_{0\le i\le\lambda}$ is a Morley sequence for $\varphi_\alpha(\bar{x},b)$. Given another $b^*$ which satisfies $\varphi_\alpha(\bar{e}_i,\bar{y})$ for $m_\alpha$ many $i$'s, it follows that $\varphi_\alpha(\bar{x},b^*)\sim\varphi_\alpha(\bar{x},b)$ by the choice of $m_\alpha$. By d) in the code-definition $b^*= b$. Hence b) is true for $\Sigma$.
\smallskip

\noindent c) is clear.
\smallskip

\noindent Since $\bar{a}_0,\ldots,\bar{a}_{i-1},\bar{f},\bar{a}_{i+1},\ldots,\bar{a}_\lambda,\bar{a}_i$ is again a Morley sequence for $\varphi_\alpha(\bar{x},b')$ we have
\[
(\bar{a}_0-\bar{a}_i,\ldots,\bar{a}_{i-1}-\bar{a}_i,\bar{f}-\bar{a}_i,\bar{a}_{i+1}-\bar{a}_i,\ldots,\bar{a}_\lambda-\bar{a}_i)\vDash\Sigma
\]
and hence
\[
(\bar{e}_0-\bar{e}_i,\ldots,\bar{e}_{i-1}-\bar{e}_i,-\bar{e}_i,\bar{e}_{i+1}-\bar{e}_i,\ldots,\bar{e}_\lambda-\bar{e}_i)\vDash\Sigma.
\]
We get d).
\smallskip

\noindent To prove e) we assume first that $\alpha$ is groupless (e1). That means $\varphi_\alpha(\bar{x},b)$ is not a torsor set. By Lemma~\ref{add2.4} the assertion follows.\\
Otherwise $X=\varphi_\alpha(\C,b')$ is a torsor set. Hence $X-\bar{f}\sim\varphi_\alpha(\bar{x},b)$ is a group set since $\bar{f}$ is in $X$. This is (e2).\\
To prove (e3) we extend the Morley sequence $\{\bar{e}_i:0\le i\le\lambda\}$ by an element $\bar{d}$. Then
\[
\bar{e}_0+\bar{d},\ldots,\bar{e}_{i-1}+\bar{d},\bar{e}_i-\bar{e}_j+\bar{d},\bar{e}_{i+1}+\bar{d},\ldots,\bar{e}_\lambda+\bar{d},\bar{d}
\]
is again a Morley sequence for $\varphi_\alpha(\bar{x},\bar{b})$. Hence
\[
\Sigma(\bar{e}_0,\ldots,\bar{e}_i-\bar{e}_j,\ldots,\bar{e}_\lambda).
\]
Similarly we get e4).
\smallskip

e5) follows again by Lemma~\ref{add2.4}.\hfill$\Box$(Claim)
\medskip

Using compactness we get a finite part $\psi'_\alpha$ of $\Sigma$ that implies a), b), c), e1), e2), e5).
\smallskip

If $\alpha$ is groupless consider the following operations:
\[
V_i(\bar{x}_0,\ldots,\bar{x}_\lambda)=(\bar{x}_0-\bar{x}_i,\ldots,\bar{x}_{i-1}-\bar{x}_i,-\bar{x}_i,\bar{x}_{i+1}-\bar{x}_i,\ldots,x_\lambda-\bar{x}_i)
\]
and $\mathcal{V}$ be the subgroup generated by these operations. $\mathcal{V}$ is finite. Then \[
\psi_\alpha(\bar{x}_0,\ldots,\bar{x}_\lambda)=\bigwedge\limits_{V\in \mathcal{V}}\psi'_\alpha(V(\bar{x}_0,\ldots,\bar{x}_\lambda))
\]
satisfies d) and is also part of $\Sigma$.\\
If $\alpha$ is an coset code, property d) follows from e3) and e4). Hence it is sufficient that $\psi_\alpha$ satisfies e3) and e4). Let ${ \mathcal W}(\bar{x}_0,\ldots,\bar{x}_\lambda)$ be the subgroup  generated by the operations mentioned in e3) and e4). Again ${ \mathcal W}$ is finite,  and depends on ${\rm Inv}(\varphi_\alpha(\bar{x},\bar{b}))$. Note that $\lambda\ge m_\alpha$, hence $b$ remains constant in b) after applying these operations. Set therefore:
\[
\psi_\alpha(\bar{x}_0,\ldots,\bar{x}_\lambda)=\bigwedge\limits_{W\in{ \mathcal W}(\bar{x}_0,\ldots,\bar{x}_\lambda)}\psi'_\alpha(W(\bar{x}_0,\ldots,\bar{x}_\lambda)),
\]
which has the required properties.
\end{proof}
\bigskip

\begin{definition} Let $\alpha$, $\lambda$ and $\psi_\alpha$ be as above. A realization of $\psi_\alpha$ is called a difference sequence for $\alpha$. Moreover, given a realization $\bar{e}_0,\ldots,\bar{e}_\lambda$ of $\psi_\alpha$, we denote by a derived difference sequence one obtained by composition of the following operations:
\[
\bar{e}_0-\bar{e}_i,\ldots,\bar{e}_{i-1}-\bar{e}_i,-\bar{e}_i,\bar{e}_{i+1}-\bar{e}_i,\ldots,\bar{e}_\lambda-\bar{e}_i.
\]
If $\nu\le\lambda$ and we use the operations above only for $i\le\nu$, then we speak about a $\nu$-derived sequence.
\end{definition}


\begin{cor}\label{add4.6}
A permutation of a difference sequence is a difference sequence.
\end{cor}

\begin{proof} Note that all permutations of a difference sequence
are obtained by the operation in d) of
Theorem~\ref{add4.5}.
\end{proof}


\section{Bounds for difference sequences}
\noindent A difference sequence for a good code $\alpha \in \cc$ forms an o-system. See Theorem \ref{add4.5}c).
The next lemma is  Lemma 5.1 in \cite{Bau09}. In our contex it needs a modified proof.

\begin{lemma}  \label{add5.1}
For every code formula $\varphi_{\alpha}(\bar{x},y)$ and every natural number $r$, there is some 
$\lambda(r,\alpha) = \lambda > 0$ such that for every $D \le \C$ and every difference sequence
$\bar{e_0}, \ldots, \bar{e_{\mu}}$ for $\varphi_{\alpha}(\bar{x},y)$ with canonical parameter $b$ and 
$\mu \ge \lambda$ either
\begin{enumerate}
\item[(i)] the canonical parameter of some $\lambda$-derived sequence of $\bar{e_0}, \ldots, \bar{e_{\mu}}$ 
lies in $dcl^{eq}(D)$
\end{enumerate}
or
\begin{enumerate}
\item[(ii)] for every $\alpha$-o-system $\bar{m}$ the sequence $ \bar{e_0}, \ldots, \bar{e_{\mu}}$ contains 
a subsequence $\bar{e_{i_0}}, \ldots, \bar{e_{i_{r-1}}}$, such that $m_{\alpha} \le i_j$ and $e_{i_j}$ is 
$\ind^w$-generic over $\langle D B \bar{e_{i_0}}, \ldots, \bar{e_{i_{j-1}}}  \rangle$, where 
$B = \langle  \bar{m}, \bar{e_0}, \ldots, \bar{e_{m_{\alpha}-1}} \rangle$.
\end{enumerate}
\end{lemma}

\begin{proof}
For $T_3$
the proof in \cite{Bau09} needs some new ideas. We assume that the fixed finite field is $\F(q)$ with $q$ elements.
We have $\bar{e_j} = \bar{e_j^1} \bar{e_j^2}$,
where $U_1(\bar{e_j^1})$, $U_2(\bar{e_j^2})$  and analogusly $\bar{x} = \bar{x^1} \bar{x^2}$.
Let  $\bar{x^1}$ be of length $n_{\alpha,1}$ and
$\bar{x^2}$ be of length $n_{\alpha,2}$.\\ 
If assertion (i) is not true, then every suitable coset of $\C^{n_{\alpha}}/D^{n_{\alpha}}$ contains  at most 
$m_{\alpha}$-many elements $\bar{e_i}$ with $i \le \lambda$ of the difference sequence under consideration.
Otherwise we could substract an $\bar{e_j}$ with $j \le \lambda$ of a larger coset and would get a 
$\lambda$-derived sequence with the canonical parameter in $D^{eq}$.\\

\noindent {\bf Claim} Assume $n_{\alpha,1} \not= 0$. Then there is a natural number $k_{\alpha}$, such that for every $\mu$ and 
every $i \le \mu$ we have $\mid \{ j : i \not= j, \bar{e_j}^1 = \bar{e_i}^1 mod D \} \mid \le k_{\alpha}$.\\

\begin{proof}
If $n_{\alpha,2} = 0$, then $k_{\alpha} = m_{\alpha}$ fulfills the assertion.\\
Now we assume, that $n_{\alpha,2} \not= 0$. 
We choose $i$ such that $\mid \{ j : i \not= j, \bar{e_j}^1 = \bar{e_i}^1 mod D \} \mid$ is maximal.
Since the permutation of a difference sequence is again a difference 
sequence, we can assume that $i = 0$. W.l.o.g. there is some $l$, such that $\bar{e_0}^1 = \bar{e_l}^1 mod D$,
since otherwise we have only to ensure $m_{\alpha} \le k_{\alpha}$ for this case.\\
By (d) in the definition of a difference sequence
\[ \bar{e_0} - \bar{e_l}, \ldots, \bar{e_{i-1}} - \bar{e_l}, - \bar{e_l},  \bar{e_{i+1}} - \bar{e_l}, \ldots, 
\bar{e_{\lambda}} - \bar{e_l} \]
is again a difference sequence.\\
Hence w.l.o.g. $\bar{e_0}^1 \in D$ and we consider $Z = \{j : \bar{e_j}^1 \in D \}$.
Our aim is, that we can find $k_\alpha$ such that,  $\mid Z \mid \le k_\alpha$ for all possible $Z$.
By assumption there are at most $m_\alpha$-many elements $e_j^2$ in the same coset modulo $D$.\\
After some permutations we can assume that there is some $\beta$, such that
\[ Z = \{ j : j \le \beta  \}.\]
We consider $E(\beta)_2 = \langle \bar{e_0}^2, \ldots, \bar{e_{\beta}^2} \rangle$. Let $s_2$ be 
$ldim(E(\beta)_2/B_2 + D_2)$, where $B = \langle  \bar{m}, \bar{e_0}, \ldots, \bar{e_{m_{\alpha}-1}} \rangle$, as above.
Then $ldim(E(\beta)_2/D) \le s_2 + (m_{\alpha} +1)n_{\alpha,2})$. All $\bar{e_j}^2$ with $j \le \beta$ are in the 
finite vector space $E(\beta)_2/D_2$ and at most $m_{\alpha}$ many are in 
the same  coset of $D_2$.    We get
\[  \beta + 1  \le m_{\alpha} q^{(s_2 + (m_{\alpha} + 1) n_{\alpha,2}) n_{\alpha,2}} .\]
Now define 
\[ Y = \{ i : m_{\alpha}  \le i \le \beta, \bar{e_i}^2 \notin D_2 + B_2 + \langle \bar{e_0}^2, \ldots \bar{e_{i-1}}^2 
\rangle_2 \} .\]
Then $s_2 \le \mid Y \mid n_{\alpha,2}$ and hence
\[  \beta + 1  \le m_{\alpha} q^{(\mid Y \mid n_{\alpha,2} + (m_{\alpha} + 1) n_{\alpha,2}) n_{\alpha,2}} .\]
Now we choose $k_{\alpha}$, such that for all $\beta > k_{\alpha}$ we have $\delta(B/D)  < \mid Y \mid$. 
By (C 5b) there is at most one $\bar{e_i}$ that is weakly generic over $\langle D  B  \bar{e_0}, \ldots \bar{e_{i-1}} 
\rangle$. 
This is not possible, since $\bar{e_i}^1 \in D$.

\end{proof}

\noindent {\bf Case $n_{\alpha,1} \not= 0$}.  \\We define $E = E(\lambda) = \langle \bar{e_0}, \ldots, \bar{e_{\lambda}} \rangle$. Let $s$ be $ldim(E_1/B_1 + D_1)$. Then $ldim(E_1/D_1) \le s + (m_{\alpha} + 1)n_{\alpha,1}$.Then
\[ \lambda +1 \le k_{\alpha} q^{(s + (m_{\alpha } + 1)n_{\alpha,1})n_{\alpha,1 } }.\]
Define $X = \{ i : m_{\alpha } \le i \le \lambda, \bar{e_i^1} \notin \langle D_1 B_1 + \langle \bar{e_0^1}, \ldots,
\bar{e_{i-1}^1}\rangle_1 \rangle$. Since $s \le \mid X \mid n_{\alpha,1}$, we get
\[ \lambda +1 \le k_{\alpha} q^{(\mid X \mid n_{\alpha,1} + (m_{\alpha } + 1)n_{\alpha,1})n_{\alpha,1 } }.\]
If we choose $\lambda$ sufficiently large we get
\[ \mid X \mid >  \delta(B/D) + r. \] 
By (C 5b) we get 
$\bar{e_{i_0}}, \ldots, \bar{e_{i_{r-1}}}$, such that $m_{\alpha} < i_j$ and $e_{i_j}$ is 
$\ind^w$-generic over $\langle D B \bar{e_{i_0}}, \ldots, \bar{e_{i_{j-1}}}  \rangle$ \\

\noindent {\bf Case  $n_{\alpha,1} = 0$}.\\
If $E = E(\lambda) = \langle \bar{e_0}, \ldots, \bar{e_{\lambda}} \rangle$ as above, then $E = \langle E_2 \rangle$.
We define $s = ldim(E_2/ B_2 + D_2)$. Then $ldim(E_2/D_2) = s + (m_{\alpha} + 1)n_{\alpha,2}$ and
\[ \lambda +1 \le m_{\alpha} q^{(s + (m_{\alpha } + 1)n_{\alpha,2})n_{\alpha,2 } }.\]
Similarly as above we define
$X = \{ i : m_{\alpha } \le i \le \lambda, \bar{e_i^2} \notin \langle D_2  + B_2 + \langle \bar{e_0^2}, \ldots,
\bar{e_{i-1}^2}\rangle_1 \rangle\}$. Since $s \le \mid X \mid n_{\alpha,2}$ we get
\[ \lambda +1 \le k_{\alpha} q^{(\mid X \mid n_{\alpha,2} + (m_{\alpha } + 1)n_{\alpha,2})n_{\alpha,2 } }.\]
We can choose $\lambda$ as large such that 
\[ \mid X \mid >  \delta(B/D) + r. \] 
By (C 5b) we get 
$\bar{e_{i_0}}, \ldots, \bar{e_{i_{r-1}}}$, such that $m_{\alpha} < i_j$ and $e_{i_j}$ is 
$\ind^w$-generic over $\langle D B \bar{e_{i_0}}, \ldots, \bar{e_{i_{j-1}}}  \rangle$ \\

\end{proof}

\noindent Now we consider all finite-to-one functions $\mu^*$ and $\mu$ defined on the good codes $\alpha\in{\cc}$ with values in $\N$. We assume that the following inequalities hold:
\begin{itemize}
\item $\mu(\alpha)\ge m_\alpha$, \item
$\mu^*(\alpha)\ge\max(\lambda(m_\alpha+1,\alpha)+1,n_\alpha+1)$,
\item $\mu(\alpha)\ge\lambda(\mu^*(\alpha),\alpha)+1$. 
\end{itemize}
For the definition above we fix a function $\lambda(r,\alpha)$
given by Lemma~\ref{add5.1} and we assume that it is monotonous in
the first argument.
\medskip

\noindent Finally we will get for each such function $\mu$ as 
above a  "generic"\ substructure $P^\mu(\C)  \le \C$  such that for all good codes $\alpha \in \cc$ there is  no
difference sequences  of length $\mu(\alpha) + 1$.\\
We will extend the language $L$ by a new predicate $P^\mu$ and consider
the structure $\langle\C,P^\mu(\C)\rangle$ in the new language
$L^\mu$. $P^\mu(\C)$
will be the desired $L$-structure of finite Morley rank. We will get
$P^\mu(\C)$ by amalgamation of strong
subspaces in the class $\K^\mu_{\rm fin}$    defined below. 
\bigskip

\begin{definition} Let $\K^\mu$ be the class of all strong substructures  $U$ of $\C$ in $\h_3$, such that for every good code $\alpha$ there is no difference sequence for $\alpha$ of length $\mu(\alpha)+1$ in $U$. $\K_{\rm fin}^\mu$ are the finite  structures in $\K^\mu$.
\end{definition}
\bigskip

\noindent Note that difference sequences are given by \noindent realizations of the
formulas\linebreak
$\psi_\alpha(\bar{x}_0,\ldots,\bar{x}_{\mu(\alpha)})$ in
Theorem~\ref{add4.5}. 
\medskip

\noindent Here $\K^{\mu}$ is a class of substructures in $\h_3$. In \cite{Bau09} we work with subspaces of 
$ \C_1$. We hav e to modify the proofs of \cite{Bau09}  for  the new contexts.
Let $D\le D'$ be strong subspaces of $\C$ in $\h_3$ with ${\rm ldim}(D'/D)$ finite. 
By Lemma~\ref{9.7} there is a  sequence of minimal strong extensions for 
starting with $D$ and arriving $D'$. We call it
a geometrical sequence.
In the next lemmas we will investigate the minimal steps in this sequence, especially the prealgebraic minimal steps for the case that $D\in\K^\mu$ but $D'\notin\K^\mu$.
\bigskip


\begin{lemma}\label{add5.2}
Assume $D\le D' \le \C$, $D\in\K^\mu$, $D'$ is a prealgebraic minimal extension of $D$ and $D'$ is not $\K^\mu$. Let $\bar{e}_0,\ldots,\bar{e}_{\mu(\alpha)}$ be a difference sequence 
in $D'$ for a good code $\alpha$, such that its canonical parameter $c$ is in ${\rm dcl}^{\rm eq}(D)$. Then we find a difference sequence $\bar{d}_0,\ldots,\bar{d}_{\mu(\alpha)}$ for $\alpha$ in $D'$ with the same canonical parameter such that $\bar{d}_0,\ldots,\bar{d}_{\mu(\alpha)-1}$ are in $D$, and $\bar{d}_{\mu(\alpha)}$ is a $D$-generic realization of $\varphi_\alpha(\bar{x},\bar{c})$ that generates $D'$ over $D$.\\
If we cannot find the new sequence by a permutation of the old one, then $\alpha$ is a group code and the new sequence is obtained using operations as $\bar{e}_j$ is replaced by some $H(\bar{e}_j)-\bar{e}_i$ where $H$ is in ${\rm Inv}(\varphi_\alpha(\bar{x},\bar{c}))$. $\alpha$ is the unique good code that describes $D'$ over $D$.
\end{lemma}

\begin{proof}  Let $\bar{c} $ be a generating o-system of the canonical base algebra coded by $c$.
Since $D\in\K^\mu$, there is some $\bar{e}_i$ not completely in $D$. Since $D\le\C$ by 
C(3) $\bar{e}_i$ is $D$-generic and generates $D'$ over $D$. If there is some other $\bar{e}_j$ not completely in $D$, then again $\bar{e}_j$ is $D$-generic and generates $D'$ over $D$. 
By Lemma \ref{11.2} $\bar{e}_i=H(\bar{e}_j) -\bar{m}_j$ where $H$ is a suitable o-transformation 
over $\bar{c}$ and $\bar{m}_j$ is in $D$. Then $H(\bar{e}_j)-\bar{e}_i$ is in $D$. Since $\bar{e}_j$ is $D$-generic, we have
\[
\bar{e}_j\ind\limits_{\bar{c}} H(\bar{e}_j)-\bar{e}_i.
\]
By the properties of a difference sequence it follows that $\alpha$ is a group code and $H$ is in ${\rm Inv}(\varphi_\alpha(\bar{x},\bar{c}))$. If we replace $\bar{e}_j$ by $H(\bar{e}_j)-\bar{e}_i$ we obtain again a difference sequence with the same canonical parameter and this sequence has one more element in $D$. We can iterate the argument to obtain the assertion.\\
Finally for $D'$ over $D$ there exists
a unique code in $\cc$ by Theorem~\ref{add4.3}.
\end{proof}


\begin{cor}\label{add5.3}
Let $D$ be in $\K^\mu$ and $D\le D' \le \C$ be a minimal extension. If $D'$ has  o-dimension one over $D$, then $D'$ is in $\K^\mu$. Otherwise, in the prealgebraic case, $D'$ is in $\K^\mu$ if and only if none of the following two conditions holds:
\begin{enumerate}
\item[{\rm a)}] There is a code $\alpha\in {\cc}$ and a difference sequence $\bar{e}_0,\ldots\bar{e}_{\mu(\alpha)}$ for $\alpha$ in $D'$ such that
\begin{enumerate}
\item[{\rm i)}] $\bar{e}_0,\ldots,\bar{e}_{\mu(\alpha)-1}$ are contained in $D$.
\item[{\rm ii)}] $D'=\langle D\bar{e}_{\mu(\alpha)}\rangle^\ell$.
\item[{\rm iii)}] In this case $\alpha$ is the unique good code that describes $D'$ over $D$.
\end{enumerate}
\item[{\rm b)}] There exists a code $\alpha\in {\cc}$ and a difference
sequence for $\alpha$ in $D'$ of length $\mu(\alpha)+1$ with
canonical parameter $b$ and  with a subsequence
$\bar{e}_0,\ldots,\bar{e}_{\mu^*(\alpha)-1}$ of length
$\mu^*(\alpha)$ such that $\bar{e}_i$ is $\ind^w$-generic over
$D+B+\langle \bar{e}_0,\ldots,\bar{e}_{i-1}\rangle$ where $B$ is
generated by the first $m_\alpha$ elements of the given difference
sequence.
\end{enumerate}
\end{cor}

\begin{proof} Consider first the case where ${\rm o-dim}(D'/D)=1$. Assume that $D'$ is not in $\K^\mu$. That means there is a difference sequence $\bar{e}_0,\ldots,\bar{e}_{\mu(\alpha)}$. If the canonical parameter $b$ lies in ${\rm dcl}^{\rm eq}(D)$, then all $\bar{e}_i$ would be in $D$, since  $\bar{e}_i \notin D$ 
implies, that it generates a strong minimal extension over $D$ with an o-dimension over $D$ greater than 1. This contradicts $D\in \K^\mu$.\\
Otherwise by Lemma~\ref{add5.1} some $\bar{e}_j$ is a realization of
$\varphi_\alpha(\bar{x},b)$ and an o-system over $D$ since
$\mu(\alpha)\ge\lambda(1,\alpha)$. Again we have a contradiction.
\smallskip

Finally we assume that $D'$ is minimal prealgebraic over $D$.
Again we assume that there is a difference sequence
$\bar{e}_0,\ldots,\bar{e}_{\mu(\alpha)}$ in $D'$ for some good
code formula $\varphi_\alpha(\bar{x},b)$ where $b$ is
the canonical parameter. If $b$ lies in ${\rm dcl}^{\rm
eq}(D)$, then by Lemma~\ref{add5.2} we get case a). Otherwise since
$\mu(\alpha)\ge\lambda(\mu^*(\alpha),\alpha)$ our sequence
contains a subsequence of length $\mu^*(\alpha)$ as described in
b) by Lemma~\ref{add5.1}.
\end{proof}


\section{Amalgamation in $\K^\mu$}\label{chF}

\noindent  Again we work in $T_3$ or $T_3^{\rm eq}$.


\begin{lemma}\label{add6.1}
Let $B\subseteq A$ and $B\subseteq C$ in $\h_3$ all be strong subspaces of $\C$. Assume that $A$ and $C$ are minimal prealgebraic extensions of $B$, $A \cap C = B$,  and that $\bar{e}_0,\ldots,\bar{e}_{\mu(\alpha)}$ is a difference sequence for a good code $\alpha$ in $\langle A C \rangle$. Then there is a derived difference sequence of the above sequence with the canonical parameter in ${\rm dcl}^{\rm eq}(C)$ or in ${\rm dcl}^{\rm eq}(A)$.
\end{lemma}

\begin{proof} 
By assumption $\langle A C \rangle = A \otimes_B C \le \C$, since $A, B, C$ are strong substructures and 
$\delta(A) = \delta(B) = \delta(C)$.
We assume that the assertion of the lemma is not
true. Let $E=\langle \bar{e}_0,\ldots,\bar{e}_{m_{\alpha}-1}\rangle$. By
Lemma~\ref{add5.1} we get a subsequence
$\bar{e}_{i_0},\ldots,\bar{e}_{i_{\mu^*(\alpha)}}$, such that
$\bar{e}_{i_j}$ is $\ind^w$-generic over
$\langle C E \bar{e}_{i_0},\ldots,\bar{e}_{i_{j-1}}\rangle$.
Since $\mu^*(\alpha)\ge\lambda(m_\alpha+1,\alpha)+1$ we get a
subsequence of this sequence of length $m_\alpha+1$ such that
every element is  $\ind^w$-generic over $\langle C E \rangle$ and over $\langle A E  \rangle$. Again we have applied Lemma~\ref{add5.1}. By C(6) $ \varphi_\alpha(x,y)$ defines a torsor set and by the properties of a difference sequence a group set. Hence by C(6) $\bar{b}\in {\rm
dcl}^{\rm eq}(B)$, a contradiction to the assumption.
\end{proof}
\bigskip

\noindent By definition  all structures $A$  in $\K^\mu$ are in $\h_3$ and there are strong embeddings of them into $\C$. For a strong substructure $A$ of $\C$ \, ${\rm tp}^{\C}(A)$ is given by the isomorphism type of $ A $. 
Embeddings of strong substructures with strong images are elementary embeddings.


\begin{theorem}\label{add6.2}
The class $\K_{\rm fin}^\mu$ has the amalgamation property with respect to strong embeddings: If
$B \le C \le \C$, $B' \le A \le \C$, $B \cong B'$, and $B, B', A, C$ are in $\K^\mu_{fin}$, then there is 
some $C \le D \le \C$ in $\K^\mu_{fin}$, such that the strong isomorphism of $B'$ onto $B$ can be extended 
to an strong embedding of $A$ in $D$. 
\end{theorem}

\begin{proof}  
Since $\C$ is rich we can assume w.l.o.g., that $B = B'$ and can exchange the roles of $C$ and $A$.
Splitting $A$ and $C$ into chains of minimal strong extensions in $\K^\mu$ (C4) we can assume w.l.o.g. that $A$ and $C$ are minimal strong extensions. 
\medskip

{\em Case\/} 1: \, ${\rm o-dim}(C/B)=1$ or ${\rm o-dim}(A/B)=1$.\\
W.l.o.g. $A = \langle B a \rangle$. If $a$ is algebraic and a solution of a divisor problem of $B$, that has also a solution in $C$, then $D = C$. Otherwise $D = \langle C a \rangle$ is an amalgam in $\K$.
By Corollary~\ref{add5.3} the amalgam $D$ is in $\K^\mu$.
\medskip

{\em Case\/} 2: Both extensions $C/B$ and $A/B$ are prealgebraic.\\
Then $D = C \otimes_B A$ exists and is in $\K$ by Theorem \ref{amalk3}.
We assume that $D$ is not in $\K^\mu$ and show in this case that $C$ and $A$ have the same type over $B$.
There is a good code $\alpha$ with a difference sequence
$\bar{e}_0,\ldots,\bar{e}_{\mu(\alpha)}$ in $D$. By
Lemma~\ref{add6.1} and symmetry we may assume w.l.o.g. that its
canonical parameter $\bar{b}$ lies in ${\rm dcl}^{\rm eq}(C)$. By
Lemma~\ref{add5.2} we may assume that
$\bar{e}_0,\ldots,\bar{e}_{\mu(\alpha)-1}$ are in $C$ and
$\bar{e}_{\mu(\alpha)}$ is an $C$-generic realization of
$\varphi_\alpha(\bar{x},\bar{b})$ which generates $D$ over $C$.
\medskip

{\em Case\/} 2A: We first assume that the canonical parameter for some $(\mu(\alpha)-1)$-derived difference sequence is in ${\rm dcl}^{\rm eq}(B)$. Since this difference sequence has the same properties, we denote it again by $\bar{e}_0,\ldots,\bar{e}_{\mu(\alpha)}$. By Lemma \ref{10.1} we have the following
two subcases:
\medskip

{\em Case\/} 2.A.1: \, $\bar{e}_{\mu(\alpha)}\in A$.
\smallskip

By minimality of $A$ over $B$ we have $A = \langle B \bar{e}_{\mu(\alpha)} \rangle $. Since $A$ is in 
$\K^\mu$, there exists  an $\bar{e}_i$ that lies in $C$ and not in $B$. By Lemma \ref{10.1} $B \le C$
implies that $\bar{e_i}$ is $B$-generic and therefore isomorphic to $\bar{e}_{\mu(\alpha)}$ over B.
Furthermore $C = \langle \bar{e}_i B \rangle$.
\medskip

{\em Case\/} 2.A.2: \, $\bar{e}_{\mu(\alpha)}$ is a  solution of 
$\varphi_\alpha(\bar{x},b)$ generic over $A$.
\smallskip

We have $D = C \otimes_B A$. The canonical base algebra of $\varphi_\alpha(\bar{x},b)$ is in $B$. 
By Lemma \ref{10.4} we get $D = C \otimes_B \langle B \bar{e}_\alpha \rangle$.
Hence we get $\bar{e}_{\mu(\alpha)} = \bar{a} + \bar{m}$, where $\bar{m}$ is in $C$
and $\bar{a}$ is an o-system  that generates $A$ over $B$. It is the solution of a code formula with
parameter in $B$. Then $\bar{e}_{\mu(\alpha)}, \bar{a}, \bar{m}$ is a $B$-independent tripel.
By Lemma \ref{add2.1} $\bar{e}_{\mu(\alpha)}, \bar{a}, \bar{m}$ are generic elements of cosets of an $acl(B)$-definable group. But then the group
is defined by $\varphi(\bar{x},b)$ by the properties of a derived sequence. $\bar{a}$ and $-\bar{m}$
are in the same coset. Therefore the have the same type over $B$.
\smallskip

{\em Case\/}2.B:\, No $(\mu(\alpha) - 1)$-derived sequence has a canonical basis in $dcl^{eq}(B)$. 
\smallskip

Let $E$ be $\langle B \bar{e}_0, \ldots, \bar{e}_{m_\alpha -1} \rangle$. Then $b \in dcl^{eq}(E)$.
Furthermore 
\[D = C \otimes_E \langle E \bar{e}_{\mu(\alpha)}\rangle, \]
and 
\[D = C \otimes_E \langle E \bar{a}\rangle, \]
where $\bar{a}$ is a generating o-system of $A$ over $B$. It is the solution of a code formula
with canonical parameter $c$ in $B^{eq}$.
By Lemma \ref{11.2} $\bar{a} = H(\bar{e}_{\mu(\alpha)}) + \bar{m}$, where $\bar{m}$ is in $C$ 
and $H(\bar{e}_{\mu(\alpha)})$ is an o-transformation   over $E$. \\
Since $\mu(\alpha)$ is suficciently large, Lemma \ref{add5.1} provides us some $\bar{e}_i$ weakly 
generic over $\langle E \bar{m} \rangle$. Let $f$ be an isomorphism, that fixes $\langle E \bar{m} \rangle$ and sends $\bar{e}_{\mu(\alpha)}$ onto $\bar{e}_i$. Then $C = \langle B f(\bar{a}) \rangle$ is 
isomorphic to $A$.
\end{proof}

\bigskip

\begin{definition} Let $D$ be a strong substructure of $\C$ in $\h_3$, that is in $\K^{\mu}$. $D$ is called $\K^\mu$-rich 
if for every finite $B\le A$ in $K^\mu$ with $B\le D$, there exists an $A'$ with $B\le A'\le D$ and $A$ and $A'$ are isomorphic over $B$. Then $A' \le \C$. 

\end{definition}
\bigskip

\noindent We can use Theorem~\ref{add6.2} to produce a countable $\K^{\mu}$-rich strong substructure
$P^{\mu}(\C)$ in $\C$
via a Fra\"{\i}ss\'e-style-argument (see \cite{Zi11}):


\begin{cor}\label{add6.3}
There is a countable $\K^{\mu}$-rich strong substructure
$P^{\mu}(\C)$ in $\C$. It's isomorphism type is unique. Hence it's type in $\C$ is uniquely determined.
\end{cor}


\begin{cor}\label{add6.4}
Let $D$ be  a $\K^{\mu}$-rich strong substructure of $\C$. Then
\begin{enumerate}
\item[{\rm a)}] ${\rm acl}(D)=D$.
\item[{\rm b)}] $d(D)\ge\aleph_0$.
\end{enumerate}
\end{cor}

\begin{proof} a) Let $A$ be a finite subspace of $D$ and
$a\in{\rm acl}(A)$. W.l.o.g. we assume that $A$ is a strong
subspace of $\C$ since $D$ is a strong subspace. By property C(2)
and Corollary~\ref{add5.3} every extension $A'$ of $A$ by an element
algebraic over $A$ is strong and in $\K^\mu$. By $\K^\mu$-richness follows
the assertion.
\smallskip

b) Let $U\le\C$ be a maximal substructure of $D$  generated by
geometrically independent elements. Then $U$ is strong in $\C$ and
$U$ is in $\K^\mu$ (Axiom C(2) ). $U$ cannot
be finite since in this case an extension $U'$ of $U$ by an
geometrically independent element would be in $\K^\mu$ and had to
be realized in $D$ (Corollary~\ref{add5.3}).
\end{proof}
\bigskip

\section{The theory $T^\mu_3$}

\begin{definition}
We extend our language $L$ by a predicate $P^\mu$. Let $L^\mu$ be the extended language. We are interested in  $L^\mu$-structures $(M, D)$ where $M \preceq\C$ is a rich model of $T_3$,  $D$ is a 
$\K^{\mu}$-rich substructure of $M$, and $d(M/D)\ge\aleph_0$. Then $D \le M$ by definition. The interpretation of $P^\mu$ is 
$ D$. We call such a $L^{\mu}$-structure a $2 \times$rich $L^{\mu}$-structure. Let $R^{\mu}(M)$
be $R(M) \cap P^{\mu}(M)$. We will show that the theory $T^\mu_3$ of all $2 \times$rich 
$L^\mu$-structures is complete. Later we will also use $M, N, \ldots$ for $L^{\mu}$-structures.
\end{definition}

\noindent The $L$-reduct of a $2 \times$rich $L^\mu$-structure is rich in the sense of
$T_3$.
We use ${\rm acl}^L$ and often acl only  for the algebraic closure in the $L$-reducts. ${\rm acl}^\mu$ is the algebraic closure in the full $L^\mu$-structure.


\begin{lemma}\label{add6.5}
We consider a $2\times$rich $L^\mu$-structure $( M,D )$ as above. Then
\[D =\{ e: e = a_1 + a_2 + [c_1,c_2] , a_1, c_1 \in R^{\mu}_1(M), a_2, c_2 \in R^{\mu}_2(M) \}\].

\end{lemma}

\noindent Hence $P^{\mu}$ is definable using $R^{\mu}$.
Corollary~\ref{add6.3} provides us a $2 \times$rich $L^\mu$-structure. For a code formula $\varphi(\bar{x},y) \in \cc$
we use often the coressponding $L$-formula and denote it by $\varphi(\bar{x},\bar{y}) \in \cc$


\begin{lemma}\label{add6.6}
Let $(M,P^{\mu}(M)$ be a   $L^\mu$-structure, where $M $ is a model of $T_3$, $P^\mu(M) \le M$  and 
$P^\mu(M) \in \K^\mu$. Furthermore let 
$\varphi_\alpha(\bar{x},\bar{b})$ be a code formula. Then
$\varphi_\alpha(\bar{x},\bar{b})$ has only finitely many solutions
in $P^\mu(M)$.
\end{lemma}

\begin{proof} 
We assume that there are infinitely many solutions.
Choose a finite strong subspace $B\le M$ in $\h_3$ such that
$\bar{b}\in{\rm dcl}^{\rm eq}(B)$. There is a solution $\bar{e_0}$ 
in $P^\mu(M)$, but not in $B$. By C(3) it is a generic 
solution over $B$. Since $\delta(\langle B \bar{e_0} \rangle) = \delta(B)$ we have $\langle B \bar{e_0} \rangle \le M$. Then we choose 
a solution $\bar{e_1} \in  P^\mu \setminus  \langle B \bar{e_0} \rangle$. Again it is a solution
generic over $\langle B \bar{e_0} \rangle$ by C(3). 
Again $\delta(\langle B \bar{e_0} \bar{e_1}\rangle) = \delta(B)$ and therefore 
$\langle B \bar{e_0} \bar{e_1}\rangle \le M$. In this way we get an infinite $L$-Morley-sequence for
$\varphi_\alpha(\bar{x},\bar{b})$ over $B$ inside $P^{\mu}$. By substraction by one solution we get
an infinite diffenrence sequence inside $P^{\mu}$, a contradiction.
\end{proof}
\bigskip

\begin{definition}\label{15.4} Let $(M,P^{\mu}(M))$  be a $2\times$rich $L^\mu$-structure. and $A$ be
a  finite substructure  of $M$. $A$  satisfies the condition $(*)$ if
\medskip

$(*)$\hspace*{0.5cm} $A\le M$,  $A \in \h_3$, $A \cap P^\mu(M) \in \h_3$,
and  $d(A/A\cap P^\mu(M))=d(A/P^\mu(M))$.
\end{definition}
\bigskip

The condition $d(A/A\cap P^\mu(M))=d(A/P^\mu(M))$ says that the union  of every geometrical basis of $A\cap P^\mu(M)$ and every geometrical basis of $A$ over $P^\mu(M)$ is a geometrical basis of $A$. This implies
\[
 {\rm cl}(P^\mu(M))\cap A={\rm cl}(A\cap P^\mu(M))\cap A.
\]

Note that Lemma~\ref{add3.6} implies for $A$ with property $(*)$
that
\[
A\cap P^\mu(M)\le M\quad\mbox{and}\quad A\cap{\rm
cl}(P^\mu(M))\le M.
\]
Then  $ A\cap cl(P^\mu(M))\in \h_3$, since it is obtained 
by a sequence of algebraic and prealgebraic minimal strong extensions over $A \cap P^\mu(M)$.


\begin{lemma}\label{add6.7}
Let $(M,P^{\mu}(M))$  be a $2\times$rich $L^\mu$-structure. Assume $A \subseteq M$ satiesfies $(*)$
and $\bar{a}\subseteq M$. Then there is some $D$ with $A \subseteq D$, $\bar{a} \subseteq D$,
and $D$ satiesfies $(*)$. Furthermore 
there is a  sequence of minimal strong $\h_3$ - extensions
\[
A = A_0\le \ldots\le A_{i_0}\le \ldots\le
A_{i_1}\le \ldots\le  A_m= D
\]
such that \ \
they satisfy (*),
$A_j = \langle A B_j \rangle $ for $j \le i_0$, where $B_0 = A \cap P^\mu(M)$ ,
$B_{i_0} = D \cap P^\mu(M)$,   and $B_0 \le B_1 \le \ldots \le B_{i_0}$ is a sequence 
of minimal strong extension. \\
Furthermore $A_j = \langle A C_j \rangle $ for $i_0 < j \le i_1$, where \\
$D \cap P^\mu(M) \le C_{i_0 + 1} \le \ldots \le C_{i_1} = D \cap cl(P^\mu(M)$ is a sequence of 
minimal strong extensions.
\end{lemma}

\begin{proof}
Choose $A \le A' \le M$, such that $\bar{a} \subseteq A'$ and $A' \in \h_3$. 
Extend $A'$ to $D \subseteq cl(A')$, such that 
$D \cap P^\mu(M) \subseteq  \langle D \cap R^\mu(M) \rangle$ and
$D \cap cl(R^\mu(M)) \subseteq cl(D \cap R^\mu(M))$.
Then $D \in \h_3$, $D \cap P^\mu(M) \in \h_3$, and $D \cap cl(R^\mu(M)) \in \h_3$. \\
Furthermore $d(D/P^\mu(M) = d(D/P^\mu(M) \cap D)$. Hence $D$ satiesfies $(*)$.
By Lemma \ref{add3.6} 
\[ A \le \langle D\cap P^\mu(M), A\rangle \le \langle D \cap cl(R^\mu(M)), A \rangle  \le D . \]
We get the desired sequence 
\[ B_0 \le \ldots \le B_{i_0} \le C_{i_0 +1} \le \ldots \le C_{i_1}\]
of minimal strong extensions, as described in the Lemma, by (C4).
By  Lemma \ref{add3.6} we have
\[A = A_0 \le \ldots \le A_{i_0} \le \ldots A_{i_1} .\]
By C(4) we get the complete sequence.
\end{proof}


\begin{theorem}\label{add6.8}
Let $M$ and $N$ be $2\times$rich $L^\mu$-structures. Assume $A\le M$ and
$f(A)\le N$ satisfy $(*)$ where $f$ is an $L^\mu$-isomorphism of $A$ onto $f(A)$.Then
$(M,A)$ and $(N,f(A))$ are $L_{\infty,\omega}^\mu$-equivalent.
\end{theorem}


\begin{cor}\label{add6.9}
The $L^\mu$-theory $T^\mu_3$ of the  $2\times$rich $L^\mu$-structures is
complete.
\end{cor}


\begin{cor}\label{add6.10}
Let $M$ and $N$ be $2\times$rich $L^\mu$-structures, $\bar{a}\in R^\mu(M)$
and $\bar{b}\in R^\mu(N)$. If ${\rm tp}_L^M(\bar{a})={\rm
tp}_L^N(\bar{b})$, then $(M,\bar{a})$ and $(N,\bar{b})$ are
$L_{\infty,\omega}^\mu$-equivalent.
\end{cor}

\begin{proof} Adding elements from the algebraic closures we can assume w.l.o.g. that $\langle\bar{a}\rangle$ and $\langle\bar{b}\rangle$ are strong subspaces in $\h_3$. Then they fulfil $(*)$.
\end{proof}
\bigskip

\noindent {\em Proof of Theorem\/}~\ref{add6.8}.\\
Note that the assumptions imply by Corollary \ref{7.7}, that $f$ is an elementary map with respect to the 
L-reducts of $M$ and $N$. $f$ preserves the geometrically situation.\\
We show that the conditions
in the Theorem describe a winning strategy for the
Ehrenfeucht-Fra\"{\i}ss\'e-game between $(M,A)$ and $(N,f(A))$:
Since $M=\langle R(M)\rangle$ and $N=\langle R(N)\rangle$ we can
assume w.l.o.g. that the players choose only elements in $R(M)$
and $R(N)$. The situation is completely symmetric. Hence we can
assume that player I has choosen some element $a$ in $R(M)$. We
show that there are $A\cup\{a\}\subseteq D\le M$ and $g$ extending
$f$ such that again:
\bigskip

$(**)$\hspace*{1cm} \parbox[t]{11.5cm}{$D \le M$ and $g(D) \le N$ fulfil
$(*)$, 
where g is a $L^\mu$-isomorphism.}

\bigskip

$(**)$ describes again the winning strategy of player II in our
Fra\"{\i}ss\'e-Ehrenfeucht-game for $(M,A)$ and $(N,f(A))$.\\
\medskip

We apply Lemma \ref{add6.7} and get the desired D with a sequence of minimal strong extensions 
between $A$ and $D$. We use the notation of that Lemma.
Let $g_0$ be $f$. We show by induction on i:\\
If there is a $L^\mu$-isomorphism $g_i$    of $A_i$ into $N$, that extends $f$ and  $(**)$ is true,\\
then we can extend $g_i$ to  $g_{i+1}$ that sends $A_{i+1}$ into $N$ and fulfils $(**)$. Note that $g_i$ is an 
elementary partiell $L$-isomorphism. $g_i$ is not only a
$L^\mu$-isomorphism. It repects also $cl(P^\mu)$: 
For $i_0 < i$ $A_i $ contains  for every $b \in cl(P^\mu(M) \cap A_i )=  cl(A_{i_0})$ a sequence of minimal strong extensions
over $A_{i_0}$ that contains $b$.  \\ 
 
\begin{enumerate}
\item $i < i_0$.
\begin{enumerate}
\item  $A_{i+1} =\langle A_i b \rangle$, $b \in R^\mu(M)$, and $d(b/A) =  1$.\\
By Corollary \ref{add6.4} there is some $c \in R^\mu(N)$ such that $d(c/g_i(A_i) =   1$ and
with $g_{i+1}(b) = c $ we can extend $g_i$ to an $L^\mu$-isomorphism $g_{i+1}$. Then $g(A_{i+1})$ 
satisfies $(*)$. 
\item $A_{i+1} =\langle A_i b \rangle$, $b \in R^\mu(M)$, $d,e \in B_i$, and $[b,d] = e$. \\
Since $acl(P^\mu(N)) = P^\mu(N)$, we get $g_{i+1}(b) \in P^\mu(N)$.
\item $A_{i+1} = \langle A B_i \bar{b} \rangle $, where $\bar{b} \in R^\mu(M)$ gives a prealgebraic minimal strong 
extension of $B_i$ by a solution of a  formula $\psi(\bar{x},u)$ where $\psi$ corresponds to a 
good code and $u$ is a canonical
base algebra in $B_i$.\\
Since $N$ is $2\times$rich there is a solution $g_{i+1}(b)$ for $\psi(\bar{x},g_i(u))$ 
over $g_i(A_i)$ in $P^\mu(N)$.
We get $(**)$.
\end{enumerate}
\item $i_0 \le i < i_1$.\\
In this case we have no trancendental minimal strong extensions by $(**)$.
\begin{enumerate}
\item $A_{i+1} = \langle A_i b \rangle$, $b \in cl(P^\mu(M)) \setminus  P^\mu(M)$, $[b,d] = e$, and
$d,e \in cl(P^\mu(M)) \cap A_i$. \\
There is some $c \in cl(P^\mu(N))$ with $[c,g_i(d)] = g_i(e)$, since $cl(P^\mu(N))$ is algebraically closed.
If $c \in P^\mu(N)$, then 
\[ (+) \langle g_i(A_i) c \rangle \cap P^\mu(N) = \langle g_i(B_{i_0}) c \rangle \le N.\]
Furthermore $\delta(g_i(B_{i_0}) = \delta(g_i(C_i)) = \delta(\langle g_i(C_i) c \rangle)$. Since 
$g_i(B_{i_0}) \le N$, we get $c \in cl(B_{i_0})$. By (+) it follows $c \in acl(g_i(B_{i_0})$. Then
$b \in acl(B_{i_0}) \subseteq P^\mu(M)$, a contradiction.
Hence $c \notin P^\mu(N)$ and with $g_{i+1}(b) = c$ we get $(**)$ for i+1.
\item $A_{i+1} = \langle A C_i \bar{b} \rangle $, where $\bar{b} \in cl(R^\mu(M))$ gives a prealgebraic minimal strong 
extension of $C_i$ by a solution of a  formula $\psi(\bar{x},u)$ where $\psi$ corresponds to a 
good code and $u$ is a canonical
base algebra in $C_i$, and the solution $\bar{b}$ is an o-system over $P^\mu(M)$.\\
Since $N$ is  a model of $T_3$, there is a solution $g_{i+1}(\bar{b})$ for $\psi(\bar{x},g_i(u))$ 
over $g_i(A_i)$ in $cl(P^\mu(N))$, that is an o-system over $P^\mu(N)$. We use Lemma \ref{add6.6}
We get $(**)$.
\end{enumerate}
\item $i_1 \le i \le m$
\begin{enumerate}
\item $A_{i+1} =\langle A_i b \rangle$, $b \in R(M)$, $b$ is geometrical independent over 
$\langle R^\mu(M) A_i \rangle$.\\
Since $N$ is $2\times$-rich, we find $g_{i+1}(b)$ in $N$, such that 
$g_{i+1}(b)$ is geometrically independent over $\langle R^\mu(N) g_i(A_i) \rangle$. Then 
$(**)$ is true.
\item $A_{i+1} = \langle A_i b \rangle$, $[b,d] = e$, $d,e \in A_i$, and $b$ is not $cl(P^\mu(M))$. \\
There is some $c \in R(N)$, such that $[c,g_i(d)] = g_i(e)$. We show $c \notin cl(P^\mu(N))$.
Then we can extend $g_i$ by $g_{i+1}(b) = c$ and get $(**)$.\\
\begin{enumerate}
\item $d \in cl(P^\mu(M))$ \\
Then $d \in C_{i_1}$ and $b,e \notin cl(P^\mu(M)$.Then $g_i(d) \in g_i(C_{i_1}) \subseteq cl(P^\mu(N))$
and $g_i(e) \notin cl(P^\mu(N)$. Hence $c \notin cl(P^\mu(N))$.
\item $d \notin cl(P^\mu(M)$ and $e \notin cl(P^\mu(M), d)$\\
Then we have the same geometrically situation for $g_i(d)$ and $g_i(e)$ in $N$. Hence 
$c \in cl(P^\mu(N))$ is impossible.
\item $d \notin cl(P^\mu(M)$ and $e \in cl(P^\mu(M), d)$\\
Then $b \in cl(P^\mu(M), d)$. Hence $g_i(d) \notin cl(P^\mu(N))$ and $g_i(e) \in cl(P^\mu(N), g_i(d))$ and 
therefore $c \in cl(P^\mu(N), g_i(d))$. $c \in cl(P^\mu(N))$ would imply $g_i(d) \in cl(P^\mu(N))$, a contradiction.
\end{enumerate}

\item $A_{i+1}$ is a prealgebraic minimal strong extension of $A_i$, generated by an o-system $\bar{b}$
that is an o-system over $P^\mu(M)$: $M \models \psi(\bar{b},\bar{c})$, where $\psi$ corresponds to a code formula
and $\bar{c}$ is the  canonical base algebra, geometrically independent from $cl(P^\mu(M))$. \\
Then   $g_i(\bar{c})$  and  all solutions $\bar{d}$ of $\psi(\bar{x},g_i(\bar{c})$ are geometrically independent from
$P^\mu(N)$. Hence $\bar{d}$ is an o-system over $cl(P^\mu(N))$. (**) is true.

\end{enumerate}
\end{enumerate}

\hfill$\Box$


\begin{cor}\label{add6.11}
Let $M$ be a $2\times$rich $L^\mu$-structure. The code formulas $\varphi_\alpha(\bar{x},b)$
with $b$ in $P^\mu(M)^{\rm eq}$ are  strongly minimal.
\end{cor}

\begin{proof} By Lemma~\ref{add6.6} there are only finitely many
solutions in $P^\mu(M)$. Let $B\le M$ be a strong subspace of
$P^\mu(M)$ such that $b\in{\rm dcl}^{\rm eq}(B)$ and all these solutions in $P^\mu(M)$ are in $B$. We show
that any two solutions $\bar{a}$, $\bar{c}$ that are not in
$B$ have the same $L^\mu$-type over $B$. Let $\bar{a}$ and
$\bar{c}$ be such solutions of $\varphi_\alpha(\bar{x},b)$ not in $B$. By (C3) they are 
generic over $B$. Then $\langle Ba\rangle$ and $\langle Bc\rangle$ with $f(B)=B$ and $f(a)=c$ fulfil the conditions in
Theorem~\ref{add6.8}. Hence ${\rm tp}_{L^\mu}(a/B)={\rm
tp}_{L^\mu}(c/B)$.
\end{proof}


\begin{lemma}\label{add6.12}
For every good code $\alpha$ there is a $L^\mu$-sentence $\chi_\alpha$ such that for all $L^\mu$-structures $M$ where $M\restriction L\vDash T_3$, $P^\mu(M) \le M$, and $ P^\mu(M) \in \K^\mu$ we have :\\
$M\vDash\chi_\alpha$ if and only if every minimal prealgebraic extension of $P^\mu(M)$ given by $\varphi_\alpha(\bar{x},b)$ with $b\in P^\mu(M)^{\rm eq}$  is not in $\K^\mu$.
\end{lemma}

\begin{proof} We replace the formula for the code $\alpha$ by the corresponding $L$-formula and denote it 
by $\varphi_\alpha(\bar{x},\bar{y})$
Let $\bar{a}$ be a solution of
$\varphi_\alpha(\bar{x},\bar{b})$ not in $P^\mu(M)$. By C(3)
$\bar{a}$ is a generic solution. If $\langle P^\mu(M)\bar{a}\rangle$ is not
in $\K^\mu$, then we have the cases a) or b) of Corollary~\ref{add5.3}. In
case a) $\langle P^\mu(M)\bar{a}\rangle$ contains a
difference sequence for $\varphi_\alpha(\bar{x},\bar{b})$ of length $\mu(\alpha)+1$ for $\alpha$. In
case b) there is a difference sequence of length $\mu(\beta)+1$
for a good code $\beta$, which contains a subsequence $\ldots , \bar{e_i}, \ldots  $ of length
$\mu^*(\beta)$ such that $\bar{e_i}$ is $\ind^w$-generic over  $ \langle \bar{e_0}, \ldots, \bar{e_{i-1}} P^\mu(M)\rangle$. Hence
$\mu^*(\beta)n_\beta\le n_\alpha$ in this case. Since $\mu^*$ is
finite-to-one, only a finite set $C_\alpha$ of codes $\beta$ can
occur. Let $C'_\alpha=C_\alpha\cup\{\alpha\}$. Then $P^\mu(M)$ has
no prealgebraic minimal extensions in $\K^\mu$ given by $\alpha$
if and only if
\newpage
\[M \models \forall \bar{b} \in P^\mu \bigvee\limits_{\beta \in C'_\alpha}
\exists\bar{y}_0 \ldots \bar{y}_{\mu(\beta)} \in R^\mu [\exists\bar{x}\varphi(\bar{x},\bar{b})\]
\[\longrightarrow  \exists^\infty\bar{x}(\varphi_\alpha(\bar{x},\bar{b}) \wedge 
\exists\bar{z}_0 \ldots \bar{z}_{\mu(\beta)}\psi_\beta(\bar{y}_0 + \bar{z}_0,
\ldots, \bar{y}_{\mu(\beta)} + \bar{z}_{\mu(\beta)}))]. \]
W.l.o.g. $\varphi_\alpha(\bar{x},\bar{y})$ is in $L$. Note that the formula after $\exists^\infty\bar{x}$ is in $L$ and $\varphi_\alpha(\bar{x},\bar{b})$ is strongly minimal in $M\restriction L$. In this way we express "for $\bar{x}$ generic over $P^\mu(M)$", since by Lemma~\ref{add6.6} $\varphi_\alpha(\bar{x},\bar{b})$ has only finitely many solutions in $P^\mu(M)$.
\end{proof}

\section{Axiomatization of $T^\mu_3$}

\noindent By Corollary~\ref{add6.9} $T^\mu_3$
is a complete theory. 
Using the Amalgamation Theorem~\ref{add6.2} we get a countable $\K^\mu$-rich
subspace $D \le \C$ . Let $P^\mu(\C)$ be $ D$. Then
$(\C,P^\mu(\C^\mu))$ is a $2\times$-rich $L^\mu$-structure. We call it our
standard model. The following is true in $(\C,P^\mu(\C^\mu))$.  It can be
expressed in $L^\mu$.  Let $M$ be  a $L^\mu$-structure:
\bigskip

\begin{tabular}{lp{13cm}}
$T^\mu_3\:1)$ & $M\restriction L$ is a model of $T_3$.\\[2mm]
$T^\mu_3\:2)$ & ${\rm acl}^L(P^\mu(M)) =P^\mu(M)$. $d(P^\mu(M))$ and\\
& $d(M/P^\mu(M))$ are infinite for $\omega$-saturated models.\\[2mm]
$T^\mu_3\:3)$ & $P^\mu(M)$ is in $\K^\mu$.\\[2mm]
$T^\mu_3\:4)$ & If $b$ is in ${\rm dcl}^{\rm eq}(P^\mu(M)))$ and $\bar{a}$ is a solution of $\varphi_\alpha(\bar{x},b)$ in $M$\\
& generic over $P^\mu(M)$ for some code formula $\varphi_\alpha(\bar{x},b)$, \\
& then $\langle P^\mu(M)\bar{a}\rangle$ is not in $K^\mu$.
\end{tabular}
\bigskip

\noindent These sets of axioms are elementary. For $T^\mu_3\:1)$ this is
clear. 
There is no problem to express $T^\mu_3\:2)$. For $T^\mu_3\:3)$ note that $P^\mu(M)$ is strong since it is closed under ${\rm acl}^L$. The absense of difference sequences for $\varphi_\alpha(\bar{x},\bar{y})$ of length $\mu(\alpha)+1$ in $P^\mu(M)$
can be expressed by Theorem~\ref{add4.5}. For $T^\mu_3\:4)$ we use Lemma~\ref{add6.12}. Finally for  all axioms we find $L^\mu$-formulas.\\
It is clear that $2\times$-rich $L^\mu$-structures satisfy $T^\mu_3\:1)$, $T^\mu_3\:2)$ and $T^\mu_3\:3)$. That they also satisfy $T^\mu_3\:4)$ is part of the following theorem:


\begin{theorem}\label{add7.1}
An $L^\mu$-structure $M$ that satisfies $T^\mu_3\:1)$, $T^\mu_3\:2)$ and $T^\mu_3\:3)$ is $2\times$rich 
if and only if it is an $\omega$-saturated model of $T^\mu_3$.
\end{theorem}

\begin{proof} First assume that $M=(M\restriction L, P^\mu(M))$ is an $\omega$-saturated 
model of $T^\mu_3$. 
As a saturated model of $T_3$ $M\restriction L$ is rich in the sense of $L$.
We show that $P^\mu(M)$ is $\K^\mu$-rich.  Let $B\subseteq A$ be in $\K^\mu$.  Assume $B\le P^\mu(M)$. W.l.o.g. $A$ is a minimal strong extension of $B$. There are three cases:
\begin{enumerate}
\item[i)] If $A=\langle Ba\rangle$ and $a$ is algebraic over
$B$, then $A$ is in $P^\mu(M)$ by $T^\mu_3\:2)$. 
\item[ii)] Let $A$
be a minimal prealgebraic extension of $B$: $A=\langle
B\bar{a}\rangle$ where $\bar{a}$ is the generic solution of
some code formula $\varphi_\alpha(\bar{x},b)$ where
$b$ is in ${\rm dcl}^{\rm eq}(B)$.
Let $B \le D \le P^\mu(M)$ in $\h_3$. By Theorem \ref{add6.2} an amalgam of $D$ and $A$ over $B$
exists 
in $\K^\mu$ 
and can be embedded into $M$ over $D$ by the properties of $T_3$. 
If we consider all possible $D$ and  none of these amalgams can be embedded into
$P^\mu(M)$, then we get by $\omega$-saturation some  $P^\mu(M) \otimes_B \langle B \bar{a} \rangle$ in $\K^\mu$, a contradiction to axiom $T^\mu_3\:4)$.
\item[iii)] $A$ is a minimal transcendental extension.
Then Axiom $T^\mu\:2)$ ensures the assertion.
\end{enumerate}
Now let $M$ be a $2\times$rich $L^\mu$-structure.  $M$ satisfies $T^\mu\:1)-T^\mu\:3)$. We show $T^\mu\:4)$.
Let $\varphi_\alpha(\bar{x},b)$ be a good code formula.
Choose a strong subspace $B$ in $P^\mu(M)$ such that $b\in{\rm dcl}^{\rm eq}(B)$. Assume there is a solution $\bar{a}$ of $\varphi_\alpha(\bar{x},b)$ generic over $P^\mu(M)$ such that $\langle 
P^\mu(M),\bar{a}\rangle$ is in $K^\mu$. Since $M$ is $\K^\mu$-rich there is a strong copy $A_0\supseteq B$ of $\langle B\bar{a}\rangle$ over $B$ in $P^\mu(M)$.  In the next step we get a copy $A_1$ of $\langle A_0\bar{a}\rangle$ over $A_0$ inside $P^\mu(M)$. We can continue this process as long as we want and get a contradiction to the fact that $P^\mu(M)$ is in $\K^\mu$. 
\smallskip

\noindent By Corollary~\ref{add6.3} there exists a $2\times$rich $L^\mu$-structure. 
Hence $T^\mu_3$ is
consistent and we have an $\omega$-saturated model $N$ of $T^\mu_3$.
As shown above $N$ is a $2\times$rich $L^\mu$-structure. By
Theorem~\ref{add6.8} $M$ and $N$ are
$L_{\infty,\omega}^\mu$-equivalent. Hence $M$ is an
$\omega$-saturated model of $T^\mu_3$.
\end{proof}


\begin{cor}\label{add7.2}
The deductive closure of $T^\mu\:1)$ -- $T^\mu\: 4)$ is the complete theory $T^\mu_3$.
\end{cor}

\begin{proof} This follows from Theorem~\ref{add7.1} and
Corollary~\ref{add6.9}.
\end{proof}
\bigskip

\noindent Let $\C^\mu$ be the monster model of $T^\mu_3$ where we work in.


\begin{lemma}\label{add7.3}
Let $M\preceq\C^\mu$ be a model of $T^\mu_3$.
\begin{enumerate}
\item[{\rm i)}] $R^\mu(\C^\mu)$ and $R(M)$ are geometrically independent over $R^\mu(M)$.
\item[{\rm ii)}] In $P^\mu(\C^\mu)$ \, ${\rm cl}(X)$ is part of ${\rm acl}^\mu(X)$.
\item[{\rm iii)}]  For $i = 1 , 2$ $R_i^\mu(x) = R^\mu(x) \cap U_i(x)$ is strongly minimal. Hence  $R^\mu(x)$ has Morley rank 1 and Morley degree 2.
\item[{\rm iv)}] $P^\mu(x)$ has  Morley rank 3.
\end{enumerate}
\end{lemma}

\begin{proof} i) Let  $\bar{a}$ in  $R(M)$ be geometrically dependent over $R^\mu(M)$. Then ``$\bar{a}$ is geometrically independent over $R^\mu$'' is part of the $L^\mu$-type of $\bar{a}$. 
Since ${\rm tp}^M(\bar{a})={\rm tp}^{\C^\mu}(\bar{a})$ it follows the assertion. 
\smallskip

ii) W.l.o.g. we assume $\langle B \rangle \le P^\mu(\C^\mu)$, $B\subseteq
R^\mu(\C^\mu)$ and $a\in{\rm cl}(B)\cap P^\mu(\C^\mu)$.
Then $A = CSS(\langle B a \rangle) \le P^\mu(\C^\mu)$. 
By C(4) there is a  geometrical construction from  $B$ to $A$, that
has only algebraic and prealgebraic steps.  By Lemma~\ref{add6.6} $A\subseteq {\rm acl}^\mu(B)$.
\smallskip

iii) To show the strong minimality of $R^\mu_i(x)$, we consider
again some $\omega$ saturated $M\preceq \C^\mu\vDash T^\mu$ and
$a,c\in R^\mu_i(\C^\mu)\setminus M$. By ii) $a$ and $c$ are not in
${\rm cl}(R^\mu(M)) \cap P^\mu(\C)$. By i) they are both not in ${\rm 
cl}(M)$. By Lemma~\ref{add6.7} every finite subspace of $M$ is
contained in some $A\le M$ that satisfies $(*)$. If we
define $f={\rm id}$ on $A$ and $f(a)=c$, then $\langle
Aa\rangle$ and $f$ satisfy the conditions of
Theorem~\ref{add6.8}. Hence ${\rm tp}_{L^\mu}(A,a)={\rm
tp}_{L^\mu}(A,c)$ and therefore ${\rm tp}_{L^\mu}(a/M)={\rm
tp}_{L^\mu}(c/M)$ as desired.
\smallskip

iv) Let $a,b$ be two geometrically independnet elements of $R^\mu_1(\C^\mu)$. The function 
$[[a,b],x]$ gives a defiable bijection between $R^\mu_1(\C^\mu)$ and 
$P^\mu(\C^\mu)_3=\langle R^\mu(\C^\mu)\rangle_3$. Hence 
$P^\mu(\C^\mu)_3$ is strongly minimal. Hence $P^\mu(x)$ has finite Morley
rank 3.
\end{proof}

\noindent Now we give another definition for the property (*) in Theorem \ref{add6.8}.

\begin{lemma}\label{add7.4a}
\begin{enumerate}
\item A finte strong substructure $A \in \h_3$ of $\C^\mu$ satisfies (*) iff $cl(P^\mu(\C)) \cap A =
cl(R^\mu (\C^\mu) \cap A) \cap A$.
\item We use (1) to define (*) for $U \le \C^\mu$ in $\h_3$. Then (*) for $U$ is equivalent to: For every $A \subseteq U$ there is
some $A \subseteq C \le \C^\mu$, that fullfils (*).
\item If $U \le \C^\mu$ and $V \le \C^\mu$ are isomorphic 
strong substructures in $\h_3$, that satisfy (*),
then they have the same type.
\end{enumerate}
\end{lemma}


\begin{theorem}\label{add7.4}
$T_3^\mu$ is $\omega$-stable.
\end{theorem}

\begin{proof} Let $M$ be a countable elementary submodel of $\C^\mu$. We show that there are only
countably many types ${\rm tp}(\bar{a}/M)$ where $\bar{a}$ is a
finite tuple in $\C^\mu$ . W.l.o.g. we can restrict us to
$\bar{a}\subseteq R(\C^\mu)$. Furthermore we will consider finite
subspaces $\bar{a}\subseteq A\subseteq R(\C)$ with certain
properties only. For a given $\bar{a}\subseteq R(\C^\mu)$
it is easy to find a set XYZW of geometrically independent
elements (short geo. basis) such that the following is true:
\begin{enumerate}
\item[(0)] $\bar{a}\subseteq{\rm cl}_d(XYZW)$ 
\item[(1)]
$X\subseteq R^\mu(M)$ 
\item[(2)] $Y\subseteq R(M)$ is
geometrically independent over $R^\mu(M)$. 
\item[(3)] $Z\subseteq
R^\mu(\C^\mu)$ (short $R^\mu$) is geometrically independent over
$R^\mu(M)$. 
\item[(4)] $W$ is geometrically independent over
$\langle R(M), R^\mu \rangle$.
\end{enumerate}
By Lemma~\ref{add7.3}\,i) $Y$ is geometrically independent over
$R^\mu$ and $Z$ over $M$. Now we choose any $A$  such that
$XYZW\subseteq A\subseteq {\rm cl}(XYZW)$, $\bar{a}\subseteq A$
and $A\le \C^\mu$. 
$A$ is obtained by a sequence of minimal strong extensions over $A \cap M$. We use (C2), (C3), and 
(C4). Since $d(A/M) = d(A/A \cap M)$, Lemma \ref{add3.6} implies that $\langle M A \rangle $ is
strong in $\C^\mu$ and it is obtained by the same sequence over $M$.
Note that $\langle M A \rangle$ satisfies (*). If $A \cap M = A' \cap M$, $A$ and $A'$ are isomophic 
over this intersection, and $A' \le \C^\mu$, then they have the same type over $M$
by Lemma \ref{add7.4a}. Hence there are only contably many types over $M$.

\end{proof}

\section{A Lie algebra of finite Morley rank}

Let $T_i$ ($i=0,1$) be complete $L_i$-theories. Let $\Delta$ be an
interpretation of the theory $T_0$ in the theory $T_1$. In \cite{Bau02}
is defined that $\Delta$ is an interpretation of $T_0$ in $T_1$
without new information, if for every $M\vDash T_1$ every subset
$X$ of $\Delta(M)^n$ defined in $M$ by a $L_1$-formula without
parameters is definable by a $L_0$-formula without parameters. If
$T_1$ is stable, then we have the same for formulas with
parameters. In \cite{Bau02} the following result of Lascar is
published:


\begin{lemma}[Lascar]\label{add7.5}
If $T_1$ is stable and $\Delta$ is an interpetation of $T_0$ in $T_1$ without new information, then for every model $N$ of $T_0$ there is some model $M\vDash T_1$ such that $\Delta(M)\cong N$.
\end{lemma}

\begin{definition} \, If $M$ is a model of $T_3^\mu$, then let $\Gamma(M)$ be the $L$-substructure of $M$ with domain $P^\mu(M)$. Let $\Gamma(T_3^\mu)$ be the complete $L$-theory of all $\Gamma(M)$ where $M\vDash 
T_3^\mu$.
\end{definition}
\bigskip

\noindent $\Gamma$ defined above is an interpetation. We get:



\begin{theorem}\label{add7.7}
Every subset of
$P^\mu(\C^\mu)^n$ 0-defined in $\C^\mu$ is $L$-0-definable in
$\Gamma(\C^\mu)$. Hence $\Gamma$ is an interpretation without new
information and every model of $\Gamma(T^\mu)$ has the form
$\Gamma(M)$ with $M\vDash T^\mu$.
\end{theorem}

\begin{proof}
Let $A$ and $C$ be finite substructures of $P^\mu(\C)$ in $\h_3$. It is sufficient to show:\\
If $A$ and $C$ have the same  $L$-type in $\Gamma(\C^\mu)$, then they have the same $L^\mu$-type in $\C^\mu$.\\
By the assumption there is an $L$-isomorphism of $A$ onto $C$, that can be extended to an isomorphism $f$ of $CSS(A)$  onto $CSS(C)$ in $\Gamma(\C^\mu)$. Since $P^\mu(C^\mu) \le \C^\mu$ 
the self-sufficient closure in $\Gamma(\C^\mu)$ is the same as in $\C^\mu$. By Corollary \ref{add6.10}
$A$ and $C$ have the same type in $\C^\mu$.
\end{proof}

\noindent Since $T^\mu_3$ is $\omega$-stable we get:

\begin{cor}\label{add7.7a}
Every subset of $P^\mu(\C^\mu)^n$ defined in $\C^\mu$ with parameters is definable in 
$\Gamma(\C^\mu)$.
\end{cor}

\begin{theorem}\label{add7.6}
For every suitable function $\mu$
$\Gamma(T^\mu_3)$ is
uncountably categorical. $R_1(x)$  and $R_2(x)$ are a strongly minimal formulas in
this theory. The pregeometry of $R$ is given by ${\rm
acl}={\rm cl}$. For models $N$ of
$\Gamma(T^\mu_3)$ we have $N=\langle R(N) \rangle$. The Morley rank of $\Gamma(T^\mu_3)$ is 3.
The geometry of the theory is not locally modular.
\end{theorem}

\begin{proof} $R^\mu_i(x)$ for $i = 1,2$ are strongly minimal for $T^\mu$ by Lemma~\ref{add7.3}\,iii). Hence $R_i(x)$ is strongly minimal in $\Gamma(T^\mu)$. Since  $\Gamma(M)=\langle R_1(\Gamma(M))\rangle$ we have that $\Gamma(T^\mu)$  is uncountably categorical. Since ${\rm cl}$ contains ${\rm acl}$ we get ${\rm acl}={\rm cl}$ by Lemma~\ref{add6.6}.
\end{proof}

\begin{cor}\label{12.2}
\item $\Gamma(T_3^{\mu})$ is CM-trivial.
\item It is not possible to interpret an infinite field in $\Gamma(T_3^{\mu})$.
\end{cor}

\begin{proof}
We follow the proof of the CM-triviality for $T_3$ inside $P^{\mu}(\C)$. Especially we can use 
modified versions of the Lemmas \ref{8.3}, \ref{8.8} and \ref{8.9}. 
\end{proof}


\section{New uncoutably categorical groups}
We assume that $M$ is a graded Lie algebra over the field $\F(p)$, where p is a prime greater than 3 and $\F(p)$ is the field with p elements. The  Baker-Campbell-Hausdorff-formula provides us a group multiplication on the domain of M:
\[ x \circ y = x + y + \frac{1}{2} [x,y] + \frac{1}{12} [x,[x,y ]] + \frac{1}{12} [y,[y,x]].\] 

\noindent Normally we have an infinite sum and characteristic 0.  Since $M$ is 3-nilpotent the sum is finite. The usual proof works.  In  $\F(p)$ all coefficients in the considered finite series exists, since $3 < p$.
Let $\G(M)$ be the group defined above. $0$ is the unit in this group. The inverse element of $x$ is $-x$.\\
We compute the group commutator:
\[ [x,y]^G = (-x - y + \frac{1}{2} [x,y] - \frac{1}{12} [x,[x,y ]] - \frac{1}{12} [y,[y,x]]) \circ\] 
\[ ( x + y + \frac{1}{2} [x,y] + \frac{1}{12} [x,[x,y ]] + \frac{1}{12} [y,[y,x]]) =\]
\[   [x,y] +  [[x,y], x + y] =  [x,y] +  [[x,y], x ]+ [[x,y] y].\] 
Hence 
\[ [[x,y]^G,z]^G = [[x,y],z]. \]
Note that $ r x $ is $x \circ x \circ \ldots \circ x$ with $r-1$ use of $\circ$.
Then we get conversely 
\[ [x,y] = [x,y]^G \circ -([[x,y]^G,x]^G \circ [[x,y]^G,y]^G). \]
Next we define $x+ y$ using only the group-multiplication. As shown above we can use the Lie multiplication and $rx$.
\[x + y = x \circ y \circ (-\frac{1}{2}) [x,y] \circ (-\frac{1}{12}) [x[x,y]] \circ (-\frac{1}{12}) [y[y,x]] \circ
\frac{1}{4} [x[x,y]] \circ \frac{1}{4}[y[x,y]] .\]
If $M$ is a free 3-nilpotent  Lie algebra over $\F(p)$, then $\G(M)$ is free in the variety of 3-nilpotent groups of exponent p.\\ 
If $M$ is generated by $M_1$, then $\G(M)$ is generated by $M_1$. \\
Let $L^-$ be the reduction of $L$, where the predicates $U_i$ and the projections are canceled. Let $M^-$ be the corresponding $L^-$-reduct of $M$. $\circ$ is definable in $M^-$. Conversely the $L^-$-reduct $M^-$  of $M$ is definable using $\circ$ only, as shown above. Hence we can show:

\begin{lemma}\label{13.1}
The elementary theories of $\G(\Gamma(\C^{\mu}))$ and $(\Gamma(\C^{\mu})^-$ are binterpretable.
\end{lemma}

In fact the group $\G(M)$ can be considered as the $L^-$-reduct of $M$.

\begin{theorem}\label{13.2}
The elementary theory of $\G(\Gamma(\C^{\mu}))$ is uncountably categorical of Morley rank 3. 
It is not one-based but CM-trivial.
$Z_1(\G(\Gamma(\C^{\mu})))$ is strongly minimal. 
\end{theorem}

\begin{proof}
We have the results for $\Gamma(T_3^{\mu})$ in Theorem \ref{add7.6} and Corollary \ref{12.2}.
Then it follows for $\Gamma(T_3^{\mu})^-$. 
For CM-trivialty  we can use H.N\"ubling's result in \cite{N}: Reducts of CM-trivial stable theories of finite Lascar rank are CM-trivial. By biinterpretability in Lemma \ref{13.1} we get the desired properties for 
$\G(\Gamma(T_3^{\mu}))$. For CM-triviality we can use \cite{Bau02}. 
\end{proof}

\begin{theorem}\label{13.3} The elementary theory of $\G(\Gamma(\C^{\mu}))$ does not allow the interpretation of an infinite field.
\end{theorem}

\begin{proof}
Such  an interpretation would imply an interpretation of that field in $\Gamma(\C^{\mu})$, a contradiction.
\end{proof}

\end{document}